\documentclass[11pt,reqno]{amsart}

\theoremstyle{plain}
\newtheorem{theorem}{Theorem}[section]
\newtheorem{lemma}[theorem]{Lemma}
\newtheorem{corollary}[theorem]{Corollary}
\newtheorem{proposition}[theorem]{Proposition}

\newtheorem{definition}[theorem]{Definition}
\newtheorem{example}[theorem]{Example}

\newtheorem{remark}[theorem]{Remark}
\theoremstyle{remark}

\newcommand{\C}{C^{\infty}}
\newcommand{\R}{\mathbb{R}}
\newcommand{\Z}{\mathbb{Z}}
\newcommand{\sh}{\mathcal{O}}
\newcommand{\ev}{_{\overline{0}}}
\newcommand{\od}{_{\overline{1}}}

\usepackage{amssymb,amsmath,amsthm}
\usepackage{epsfig,amssymb,amsmath,amsthm,graphicx,psfrag}
\usepackage{MnSymbol}
\usepackage{tikz-cd} 
\usepackage{enumerate}
\usepackage{color}
\usepackage{array}
\usepackage{multirow}
\usepackage{mathtools}
\usepackage[hyperindex,bookmarksnumbered,plainpages]{hyperref}

\begin{document}
\title{$\C$-SuperRings and $\C$-SuperSchemes}

\author{Cristian Danilo Olarte}

\address{Instituto de Matemáticas, FCEyN, Universidad de Antioquia, 50010  Medellín, Colombia}

\email{cristian.olarte@udea.edu.co}

\author{Pedro Rizzo} 

\email{pedro.hernandez@udea.edu.co}

\author{Alexander Torres-Gomez}
 
\email{galexander.torres@udea.edu.co}


\subjclass[2020]{Primary: 13Nxx, 58A50, 17A70, 14A22; Secondary: 14M30}

\begin{abstract}
This paper develops a theory of $\C$-superrings and their associated $\C$-superschemes. We prove a key equivalence between the category of fair affine $\C$-superschemes and the category of fair $\C$-superrings. We place special emphasis on split $\C$-superrings, which generalize the function algebras of supermanifolds and serve as building blocks for more complex, non-split structures. 
\\

\noindent \textsc{Key words:} $\C$-Rings, Superrings, $\C$-Schemes, Supermanifolds, Superschemes. 
\end{abstract}

\maketitle

\noindent

\tableofcontents

\section{Introduction}

Superalgebraic and supergeometric structures generalize classical non-graded structures to $ \mathbb{Z}_{2}$-graded ones. This extension introduces an interplay between commuting and anticommuting elements, enriching tools of algebra, differential geometry, and algebraic geometry. Although its origins lie in theoretical physics, the theory has evolved into an independent field of study. This development led to the creation of supergeometry, which provides a geometric framework for understanding objects like supermanifolds and superschemes \cite{Manin, Westra, CCF, BRUZZO, Noja}.\\


In parallel, synthetic geometry emerged as an approach that studies geometric objects through axioms and logical reasoning, without reliance on numerical or algebraic representations \cite{Kock}. This approach gained significant attention in the
works of E. Dubuc, I. Moerdijk and G. E. Reyes \cite{D1, D2, MRI, MRII, MR}, among others. However, interest in synthetic geometry waned until the recent revival of its central concept, $\C$-rings, within Dominic Joyce’s program
on derived differential geometry \cite{DJ}. Joyce's work has sparked a new wave of research by proposing a version of algebraic geometry based on $\C$-rings instead of traditional commutative rings \cite{J, Olarte, Ler1, Yama}.\\

A natural question arises from the intersection of these two fields: Is it possible to develop a theory of $\C$-algebraic supergeometry that generalizes both classical supergeometry and $\C$-algebraic geometry? Previous attempts by H. Nishimura and D. N. Yetter \cite{NISHIMURA, YETTER} explored this question using Weil superalgebras, which are a specific type of $\C$-superring, but did not adopt a general $\C$-superring perspective. More recently, C.H. Liu and S.T. Yau used a notion of $\C$-superrings in their work on supersymmetric theories for fermionic D-branes \cite{Yau 1,Yau 2}. However, their focus on applications in theoretical physics meant they did not rigorously explore the mathematical properties of these structures.\\

Separately, D. Carchedi and D. Roytenberg approached the topic from a categorical perspective, defining $\C$-superalgebras as an example of a broader construction known as ``super Fermat theories" \cite{Carchedi}. While this approach provides a powerful and general framework with interesting applications in derived differential supergeometry \cite{Carchedi2}, its reliance on intricate category theory can be challenging for a geometric-inclined researcher.\\

This paper establishes a rigorous algebro-geometric framework for the theory of $\C$-superrings and their associated $\C$-superschemes. The foundational theory is built upon the concept of locally $\C$-ringed superspaces.\\

The primary objective is to develop a comprehensive mathematical foundation for these structures and to demonstrate a functorial correspondence between every $\C$-superring and a geometric object. A central result of this work is the proof of an equivalence of categories between the category of fair affine $\C$-superschemes and the category of fair $\C$-superrings. This generalizes a known equivalence for supermanifolds \cite[Proposition 3.4.5]{CCF}.\\

Inspired by the work of Berni and Mariano on ``smooth algebras" \cite{CM1}, we extend their key ideas and results to ``smooth superalgebras". This extension introduces novel tools and constructions that lack counterparts in either classical superring theory or the theory of $\C$-rings.\\

A particular emphasis is placed on split $\C$-superrings, which, while relatively simple, are essential for understanding more complex structures. These split rings are generalizations of the algebras of global sections on supermanifolds. However, we demonstrate that more intricate, non-split $\C$-superrings can be constructed as quotients of these simpler ones, and we analyze the resulting induced $\C$-ringed superspaces.\\

In \cite{ORTG}, we extend Batchelor's theorem \cite{Batchelor}, which states that every smooth supermanifold is globally split, to the setting of $\C$-superrings. This extension highlights the fundamental importance of split $\C$-superrings within the field of supergeometry.\\

The paper is organized as follows: 
\begin{itemize}
\item Section \ref{prelim} provides a concise review of $\C$-rings and superrings.\\

\item Section \ref{cinftysuper} introduces the definition of $\C$-superrings and their morphisms, presenting illustrative examples and a method to construct split $\C$-superrings. After establishing that the category of $\C$-superrings is both complete and finitely cocomplete, we systematically investigate the process of localization. This includes extending several foundational tools to the super-framework, which are necessary for developing a geometric theory based on $\C$-superrings. Special attention is given to the $\C$-radical of a superideal. This notion is significant because  localization at a prime superideal may not yield a local $\C$-superring unless the ideal is $\C$-radical.\\

\item   Section \ref{superspec} is dedicated to constructing a spectrum functor for defining affine $\C$-superschemes. We present two distinct approaches: one based on prime superideals and another on the space of $\R$-points. We then establish the fundamental adjunction between the spectrum functor and the global sections functor. This adjuntion is crucial for proving the equivalence of categories between fair affine $\C$-superschemes and fair $\C$-superrings.

\end{itemize}
\section{Preliminaries}\label{prelim}
\subsection{$\C$-rings and $\C$-algebraic geometry: A short background}
$\C$-rings are $\R$-algebras equipped with a family of \emph{$n$-ary operations} $\phi_f$, indexed by smooth functions $f\in \C(\R^n)$ for all $n\in \mathbb{N}$. This structure facilitates an algebraic geometry approach to smooth algebras, extending techniques used in commutative algebra.
This subsection provides a review of key definitions and results essential for understanding the corresponding concepts in the context of $\C$-superrings, which will be introduced later. The content of this subsection is primarily based on \cite{J} and \cite{Olarte}.

\begin{definition}\label{ring}
A \emph{$\C$-ring} is a pair $\left(\mathfrak{C}, \{\phi_f\}_{f\in\C(\R^n)}\right)$, where $\mathfrak{C}$ is a set and $\{\phi_f: \mathfrak{C}^n \rightarrow \mathfrak{C}|\,f\in \C(\R^n),\,n\in \mathbb{N}\}$ is a collection of maps, satisfying the following:
    \begin{enumerate}
    \item [(i)] \textbf{Projections are Preserved}: If $p_i: \mathbb{R}^n \rightarrow \mathbb{R}$ is the projection onto the 
		$i^{th}-$coordinate, then $\phi_{p_i}(c_1, \ldots, c_n)=c_i$, for all 
		$(c_1, \ldots, c_n)\in \mathfrak{C}^n$.

    \item [(ii)] \textbf{Compositions are Preserved}: If $h_1,\ldots,h_n\in \C(\mathbb{R}^m)$, $g \in \C(\mathbb{R}^n)$ and $g(h_1,\ldots,h_n)$ denotes its composition, then
\begin{equation*}
\phi_{g(h_1,\ldots,h_n)}(c_1, \ldots, c_m)=\phi_{g}(\phi_{h_1}
		(c_1, \ldots, c_m),\ldots, \phi_{h_n}(c_1, \ldots, c_m))
\end{equation*}
for all $(c_1, \ldots, c_m)\in \mathfrak{C}^m$.
    \end{enumerate} 
\end{definition}

\begin{definition}\label{Morphisms}
Let $\left(\mathfrak{C}, \{\phi_f\}_{f\in\C(\R^n)}\right)$ and
$\left(\mathfrak{D},\{\psi_f\}_{f\in\C(\R^n)}\right)$ be $\C$-rings. A \emph{morphism 
of $\C$-rings} is a map $\varphi :\mathfrak{C}\rightarrow  \mathfrak{D}$ such that the following diagram commutes for all $n\in\mathbb{N}$ and $f\in \C(\R^n)$:
\[ 
\begin{tikzcd}
\mathfrak{C}^n \arrow{r}{\varphi^n} \arrow[swap]{d}{\phi_{f}} & \mathfrak{D}^n\arrow{d}{\psi_f} \\
\mathfrak{C} \arrow{r}{\varphi}& \mathfrak{D}
\end{tikzcd}
\]
where $\varphi^n: \mathfrak{C}^n \rightarrow \mathfrak{D}^n$ is given by $\varphi^n(c_1, \cdots, c_n)=(\varphi(c_1), \cdots, \varphi(c_n))$. The set of morphisms of $\C$-rings from $\mathfrak{C}$ to $\mathfrak{D}$ is denoted by $Hom_{\C-{\bf Rings}}(\mathfrak{C}, \mathfrak{D})$.
\end{definition}

The collection of $\C$-rings and their associated morphisms constitutes a category, which we denote by {\bf C$^\infty$Rings}. Objects within this category are represented as $\left(\mathfrak{C}, \{\phi_f\}_{f\in\C(\R^n)}\right)$, or simply $\mathfrak{C}$ when no ambiguity arises.

\begin{remark}\label{rem:limits}
The category {\bf C$^\infty$Rings} possesses the following important properties:
\begin{itemize}
\item {\bf C$^\infty$Rings} is equivalent to the category of functors from Euclidean spaces to sets that preserve finite products. This characterization is crucial for establishing various properties of {\bf C$^\infty$Rings}.

\item All categorical limits and filtered colimits exist in {\bf C$^\infty$Rings} \cite[Prop. 2.5]{J}.

\item Every object in {\bf C$^\infty$Rings} admits a natural $\R$-algebra structure \cite[Lemmas 3 and 4]{Olarte}.

\item The $\R-$algebra quotient $\mathfrak{C}/I$ can be endowed with the structure of a $\C$-ring, which follows from Hadamard’s Lemma \cite[Definition 2.7, p.7]{J}.

\item We identify two important full subcategories: finitely generated, denoted {\bf C$^\infty$Rings}$^{fg}$, and finitely presented, denoted {\bf C$^\infty$Rings}$^{fp}$. The objects of {\bf C$^\infty$Rings}$^{fg}$ are $\C$-rings of the form $\C(\mathbb{R}^n)/I$, where $n\geq 0$ and $I$ is an ideal. The objects of {\bf C$^\infty$Rings}$^{fp}$ are finitely generated $\C$rings where the ideal $I$ is also finitely generated.
\end{itemize}
\end{remark}

\begin{definition}\label{localdef}
A $\C$-ring $\mathfrak{C}$ is called \emph{local} if it possesses a unique maximal ideal $\mathfrak{m}_{\mathfrak{C}}$. Furthermore, it is assumed that $\mathfrak{C}/\mathfrak{m}_{\mathfrak{C}}\cong \R$. A morphism $\varphi:\mathfrak{C}\rightarrow\mathfrak{D}$ between local $\C-$rings is a morphism of $\C-$rings satisfying $\varphi^{-1}(\mathfrak{m}_{\mathfrak{D}}) = \mathfrak{m}_{\mathfrak{C}}$.
\end{definition}

\begin{example}\label{ex:smooth}\rm
The set $\C(M)$ of smooth real-valued functions defined on a smooth manifold $M$ serves as the archetypal example of a $\C$-ring. The $n$-ary operations $\phi_{f}$ are defined by: $\phi_{f}(c_1, \ldots, c_n)(x)= f(c_1(x), \ldots, c_n(x))$, for every $x\in M$, 
$n\in\mathbb{N}$, $f\in \C(\R^{n})$ and $(c_1, \ldots, c_n)\in  \C(M)^n$. Note the following:
\begin{itemize}
\item $\C(\mathbb{R}^m)$ is the archetypal finitely generated $\C$-ring, with generators given by the projection functions  $p_1,\ldots, p_m$. This follows from the fact that $f=\phi_{f}(p_1,\ldots, p_m)$, for all $f\in \C(\mathbb{R}^m)$.

\item The set of germs of functions at a point $p\in M$, denoted $\C_p(M)$, is a typical example of a local $\C$-ring. The ideal $\mathfrak{m}_p=\{[(f,U)]\mid p \in U \; \text{and}\;  f(p)=0 \}$ (with $U \subset M$ open)  is the unique maximal ideal of $\C_p(M)$. The evaluation map at $p$, $\C_p(M) \rightarrow \R$, given by $[(f,U)]\mapsto f(p)$, is a surjective morphism of $\C$-rings with kernel $\mathfrak{m}_p$. Consequently, by \cite[Lemma 4]{Olarte}, $\C_p(X)/\mathfrak{m}_p\cong \R$.
\end{itemize}
\end{example}

In a manner analogous to the spectrum construction in classical algebraic geometry, which associates commutative rings to locally ringed spaces, we can define a \emph{Spectrum functor} that maps $\C$-rings to locally $\C$-ringed spaces. The subsequent definitions and results are designed to develop the theoretical framework necessary to achieve this objective.

\begin{definition}\label{LCRS}
A \textbf{$\C$-ringed space} is a pair $(X,\mathcal{O}_X)$, where $X$ is a topological space and $\mathcal{O}_X$ is a sheaf of $\C$-rings on $X$. The stalk of $X$ at $p \in X$ is defined as the $\C$-ring 
$$
\mathcal{O}_{X,p}:=\varinjlim_{p\in U}\mathcal{O}_X(U)
$$ 
where the direct limit is taken over all open sets $U\subseteq X$ containing $p$. 
\end{definition}

Explicitly, the elements of $\mathcal{O}_{X,p}$ are equivalence classes $[(f,U)]$ of pairs $(f,U)$, with $f\in\mathcal{O}_X(U)$. Two pairs $(f,U)$ and $(g,V)$ are equivalent if there exists an open set $ W\subseteq U\cap V$ such that $p\in W$ and $f|_W\equiv g|_W$. For any open neighborhood $U$ of $p$, we have the canonical morphism $\sigma_p(U):\mathcal{O}_X(U)\rightarrow \mathcal{O}_{X,p}$ given by $\sigma_p(s)=[(s,U)]$ for $s\in \mathcal{O}_X(U)$.\\

A $\C$-ringed space $(X,\mathcal{O}_X)$ is a \textbf{locally $\C$-ringed space} if all of its stalks $\mathcal{O}_{X,p}$ are local $\C$-rings for every $p\in X$.

\begin{example}\label{ex:rest}\rm
Let $(X,\mathcal{O}_{X})$ be a locally $\C$-ringed space and $U\subseteq X$ an open subset. The restriction of $\mathcal{O}_{X}$ to $U$, denoted $\mathcal{O}_{X}|_{U}$, is the sheaf defined on $U$ by $\mathcal{O}_{X}|_{U}(V)=\mathcal{O}_{X}(V)$ for any open subset $V$ of $U$. For any $p\in U$, the stalks $\mathcal{O}_{X,p}$ and $\mathcal{O}_{X}|_{U,p}$ are identical. Consequently, $(U,\mathcal{O}_X|_U)$ is also a locally $\C$-ringed space.
\end{example}

\begin{definition}
A morphism of $\C$-ringed spaces from $(X,\mathcal{O}_X)$ to $(Y,\mathcal{O}_Y)$ is a pair $(f,f^{\#})$, where: $f:X\rightarrow Y$ is a continuous map and $f^{\#}:\mathcal{O}_Y\rightarrow f_{\ast}\mathcal{O}_X$ is a morphism of sheaves of $\C$-rings on $Y$.
\end{definition}

\begin{remark}
In the case of locally $\C$-ringed spaces, the morphism $f^{\#}$ induces a morphism of local $\C$-rings on stalks $f^{\#}_{p}:\mathcal{O}_{Y,f(p)}\rightarrow \mathcal{O}_{X,p}$ for every point $p\in X$.
\end{remark}

We denote the category of $\C$-ringed spaces by $\bf{\C RS}$ and the category of locally $\C$-ringed spaces by $\bf{L\C RS}$. When the context is clear, we may abbreviate $(X,\mathcal{O}_X)$ to $X$ for a $\C$-ringed space and $(f,f^{\#})$ to $f$ for a morphism of $\C$-ringed spaces.

\begin{example}\rm
Let $M$ be a smooth manifold and $\mathcal{O}_{M}$ its sheaf of smooth functions, that is, $\mathcal{O}_M(U)=\C(U)$ for any open subset of $U\subseteq M$, with the usual restriction maps.  Then, as discussed in Example \ref{ex:smooth}, $\mathcal{O}_M(U)$ is a $\C$-ring, and the pair $(M,\mathcal{O}_{M})$ forms a locally $\C$-ringed space.  Furthermore, the stalk of $\mathcal{O}_M$ at any point $p\in M$ is the local $\C$-ring $\C_p(M)$ of germs of smooth functions at $p$ (see Example \ref{ex:smooth}).\\

Therefore, the category ${\bf Man}_{\mathbb{R}}$ of smooth manifolds can be viewed as a subcategory of $\bf{L\C RS}$.
\end{example}

\begin{definition}\label{def:Rpoints}
Let $\mathfrak{C}$ be a $\C-$ring. An $\R$-point of $\mathfrak{C}$ is a $\C$-ring morphism $x:\mathfrak{C}\rightarrow \R$. The set of all $\R$-points of $\mathfrak{C}$ is denoted by $X_{\mathfrak{C}}$.
\end{definition}
We equip $X_{\mathfrak{C}}$ with the topology generated by the basis of open sets of the form $U_{c}=\{x\in X_{\mathfrak{C}}\mid x(c)\neq 0 \}$ for all $c\in \mathfrak{C}$. $X_{\mathfrak{C}}$ is a regular space \cite[Lemma 4.15]{J} and hence a Hausdorff space.\\

Following \cite[Definition 2.13]{J}, let $\mathfrak{C}_x$ denote the \emph{localization} of $\mathfrak{C}$ at $x$, which is a local $\C$-ring constructed by localizing $\mathfrak{C}$ at the multiplicative set $S=\{c\in \mathfrak{C}:x(c)\neq 0\}$, for any $x\in X_\mathfrak{C}$. The {\it localization morphism} $\pi_x:\mathfrak{C}\rightarrow\mathfrak{C}_x$ is surjective \cite[Proposition 2.14]{J}. Moreover, any morphism of $\C$-rings $\varphi:\mathfrak{C}\rightarrow\mathfrak{D}$ induces a unique morphism of local $\C$-rings $\varphi_y:\mathfrak{C}_x \rightarrow\mathfrak{D}_y$, where $x=y\circ\varphi$, making the following diagram commute
\begin{equation*}
\begin{tikzcd}
\mathfrak{C} \arrow{r}{\varphi} \arrow[swap]{d}{\pi_x} 
& \mathfrak{D}\arrow{d}{\pi_y} \\
\mathfrak{C}_x\arrow{r}{\varphi_y}& \mathfrak{D}_y,
\end{tikzcd}
\end{equation*}

\begin{remark}\label{rem:localCinfty}
It is important to note that the localization procedure for $\R$-algebras differs significantly from that for $\C$-rings \cite[Section 3]{CM1}. For an $\R$-algebra $R$, the localization of $R$ at $t \in R$ is given by $R_t=\{\frac{r}{t^n}: r\in R,\, n\in \mathbb{N}\}$. However, in the category of $\C$-rings, the localization of a $\C$-ring $\mathfrak{C}$ at $f \in \mathfrak{C}$ is given by $\mathfrak{C}\{f^{-1}\}=\mathfrak{C}\{s\}/(fs-1)$, where $\mathfrak{C}\{s\} =\mathfrak{C} \otimes_\infty C^\infty (\R)$. This $\C$-ring localization can be considerably larger than its $\R$-algebra counterpart. For instance, consider the $\C$-ring $\C(\R^n)$ and a smooth function $f\in \C(\R^n)$. The localization $\C(\R^n)\{f^{-1}\}$ is isomorphic to $\C(U_f)$, where $U_f = \{ x  \in \R^n \mid f(x) \neq 0 \}$ is the principal open set defined by $f$. It is well known that not every smooth function on $U_f$ can be expressed in the form $h/f^n$ \cite[page 327]{MRI}.
\end{remark}

\begin{definition}\label{Fair}
A finitely generated $\C$-ring $\mathfrak{C}$ is called \emph{fair} if it satisfies the following condition: for any element $c\in\mathfrak{C}$, if the image of $c$ under the localization morphism $\pi_x: \mathfrak C \rightarrow \mathfrak C_x$ is zero ($\pi_x (c)=0$) for every $\R$-point $x$ of $\mathfrak C$, then $c$ must be zero.
\end{definition}

\begin{proposition}\label{prop:structural}
 Let $\mathfrak{C}$ be a $\C$-ring and $\mathfrak{C}_x$ its localization at $x \in X_\mathfrak{C}$. Then the assignment
\begin{equation*}
\mathcal{O}_{X_\mathfrak{C}}(U):=\left\{s:U\rightarrow\bigsqcup_{x\in U} 
\mathfrak{C}_x|\,\forall\, x\in U,\exists \; \text{an open}\; V_x\subseteq U\, \text{containing}\; x \, \text{and}  \, c\in 
\mathfrak{C}\,\text{such that}\; s(y)=\pi_y(c) \,\forall\,y\in V_x\right\}
\end{equation*}
for each open set $U\subseteq X_\mathfrak{C}$ defines a sheaf of $\C$-rings on $X_\mathfrak{C}$, where $\pi_y:\mathfrak{C}\rightarrow\mathfrak{C}_y$ is the localization morphism at $y$.
\end{proposition}
\begin{proof}
See \cite[Proposition 16]{Olarte}
\end{proof}

The pair $(X_\mathfrak{C}, \mathcal{O}_{X_\mathfrak{C}})$ thus defines a $\C$-ringed space. We refer to $\mathcal{O}_{X_\mathfrak{C}}$ as the structure sheaf of $X_\mathfrak{C}$.

\begin{proposition}\label{fnumeral}
Any morphism of $\C$-rings $\varphi:\mathfrak{C}\rightarrow\mathfrak{D}$ induces a morphism of locally $\C$-ringed spaces $(f_\varphi,f^{\#}_{\varphi} ):(X_\mathfrak{D},\mathcal{O}_{X_\mathfrak{D}})
\rightarrow(X_\mathfrak{C},\mathcal{O}_{X_\mathfrak{C}})$
\end{proposition}
\begin{proof}
See \cite[Proposition 17]{Olarte}
\end{proof}

\begin{lemma}\label{stalks}
For every $x  \in X_{\mathfrak{C}}$, the stalk $\mathcal{O}_{X_{\mathfrak{C}}, x}$ of the structure sheaf at $x$ is isomorphic to the localization $\mathfrak{C}_x$ of $\mathfrak{C}$ at $x$.
\end{lemma}
\begin{proof}
See \cite[Lemma 18]{Olarte}
\end{proof}

\begin{remark}
Lemma \ref{stalks} and Proposition \ref{fnumeral} together imply that for any morphism of $\C$-rings $\varphi:\mathfrak{C}\rightarrow\mathfrak{D}$, the induced morphism on stalks $\varphi_y:\mathfrak{C}_x\rightarrow\mathfrak{D}_y$, where $x=y\circ\varphi$, is precisely the morphism of local $\C$-rings $f^{\#}_{\varphi, x} :  \mathcal{O}_{X_{\mathfrak{D}},f_\varphi (x)} \rightarrow  \mathcal{O}_{X_{\mathfrak{C}},x}$.
\end{remark}

\begin{definition}\label{def:twofunctors}
We define the following two functors between the opposite category of {\bf C$^\infty$Rings} and $\bf{L\C RS}$:
\begin{itemize}
\item {\bf The Spectrum Functor:}  $Spec:  {\bf \C Rings}^{op} \rightarrow \bf{L\C RS}$ is defined on objects by $Spec(\mathfrak{C})=(X_\mathfrak{C},\mathcal{O}_{X_\mathfrak{C}})$, for any $\C$-ring $\mathfrak C$, and on morphisms by $Spec(\varphi)=(f_{\varphi},f^{\#}_{\varphi})$, for any $\C$-ring morphism $\varphi: \mathfrak C \rightarrow \mathfrak D$. 

\item {\bf The Global Sections Functor:} $\Gamma:\bf{L\C RS}\rightarrow {\bf \C Rings}^{op}$ is defined on objects by $\Gamma(X,\mathcal{O}_X)=\mathcal{O}_X(X)$, for any locally $\C$-ringed space $(X, \mathcal O_X)$, and on morphisms by $\Gamma(f)=f^{\#}(X):\mathcal{O}_Y(Y)\rightarrow\mathcal{O}_X(X)$, for any morphism of locally $\C$-ringed spaces $(f, f^{\#}) : (X, \mathcal O_X) \rightarrow (Y, \mathcal O_Y)$.
\end{itemize}
\end{definition}

Let $\mathfrak{C}$ be a $\C$-ring. $c\in\mathfrak{C}$ induces a global section $\Psi_{\mathfrak{C}}(c) \in \mathcal O_{X_{\mathfrak C}}(X_{\mathfrak C})$ on $X_\mathfrak{C}$, defined by:
\begin{align*}
\Psi_{\mathfrak{C}}(c)\colon   X_\mathfrak{C}  &\longrightarrow \bigsqcup_{x\in X_\mathfrak{C}}\mathfrak{C}_x\\
x &\longmapsto  \pi_x(c),
\end{align*} 
This defines a morphism of $\C$-rings $\Psi_{\mathfrak{C}}:\mathfrak{C}\rightarrow \Gamma(X_{\mathfrak{C}})$. Furthermore, for finitely generated $\C$-rings, we have the following:

\begin{proposition}
If $\mathfrak{C}$ is a finitely generated $\C$-ring, then the morphism of $\C$-rings $\Psi_\mathfrak{C}:\mathfrak{C}\rightarrow \Gamma\circ Spec(\mathfrak{C})$ is surjective.
\end{proposition}
\begin{proof}
See \cite[Proposition 21]{Olarte}.
\end{proof}

The morphism $\Psi_\mathfrak{C}$ represents the natural transformation arising from the right adjointness of the functors $Id_{\C\bf{-Rings}}$ and $\Gamma\circ Spec$. While the existence of a right adjoint to $\Gamma$ was proven in \cite{MR} and \cite{D1} without explicit construction,  \cite{J} provides a novel, explicit construction and explores its consequences. Due to its significance in $\C$-ring theory, we present this result without proof.  Further details can be found in \cite[Theorem 4.20]{J}.

\begin{theorem}\label{adjoint}
The functors $Spec$ and $\Gamma$  form an adjoint pair, with $Spec$ being the right adjoint and $\Gamma$ the left adjoint.  More precisely, for every $\C$-ring $\mathfrak C$ and every locally $\C$-ringed space $X$, there exist natural bijections $Hom_{\bf{\C Rings}}(\mathfrak{C},\Gamma(X))\xrightleftharpoons[R_{{\mathfrak{C}, X}}]{\,L_{{\mathfrak{C}, X}}\,} Hom_{\bf{L\C RS}}(X,Spec(\mathfrak{C}))$ 
\end{theorem}

Having established the necessary framework, we can now define \emph{affine $\C$-schemes} and \emph{$\C$-schemes}.

\begin{definition} A locally $\C$-ringed space $(X, \mathcal O_X)$ is called an affine $\C$-scheme if it is isomorphic to the spectrum of some $\C$-ring $\mathfrak C$. If $\mathfrak C$ is finitely generated, then $(X, \mathcal O_X)$ is called an affine finitely generated $\C$-scheme. A locally $\C$-ringed space $(X, \mathcal O_X)$ is called a $\C$-scheme if it is locally affine. That is, there exists an open cover $\{U_i\}$ of $X$ such that for each $i$, $(U_i,\mathcal{O}_X|_{U_i})$ is an affine $\C$-scheme.\\
\end{definition}
A morphism of affine $\C$-schemes or $\C$-schemes is simply a morphism of locally $\C$-ringed spaces.\\

We denote the categories of affine $\C$-schemes, affine finitely generated $\C$-schemes, and $\C$-schemes by $\bf{A\C Sch}$, $\bf{A\C Sch^{fg}}$ and $\bf{\C Sch}$, respectively.

\begin{example}\label{ff-functor}\rm
Since the $\C$-ring $\C(X)$ is finitely generated for any smooth manifold $X$, there exists a fully faithful functor $F:\bf{Man}_{\mathbb R} \rightarrow \bf{A\C Sch^{fg}}$.  $F$ acts on objects by $X\rightarrow \text{Spec}(\C(X))$ and on morphisms by $(f:X\rightarrow Y)\Longrightarrow (\text{Spec}(f^{*}))$, where $f^{*}$ is the pullback morphism induced by $f$.
\end{example}

A consequence of the adjunction between $Spec$ and $\Gamma$ (Theorem \ref{adjoint}) is that the functor $Spec$ transforms categorical limits in $\C$-Rings to colimits in $\bf{A\C Sch}$.  A direct calculation using universal properties shows that pushouts in $\C$-Rings correspond to fiber products in $\bf{A\C Sch}$. In particular, we have the following result:

\begin{corollary}\label{consequences of adjoint}
In the category of finitely generated $\C$-rings, the pushout  $\mathfrak{C}=\mathfrak{D}\bigsqcup_{\mathfrak{F}}\mathfrak{E}$ of morphisms $\phi:\mathfrak{F}\rightarrow \mathfrak{D}$ and $\psi:\mathfrak{F}\rightarrow \mathfrak{E}$ corresponds to the fiber product $\text{Spec}(\mathfrak{D})\times_{\text{Spec}(\mathfrak{F})}\text{Spec}(\mathfrak{E})$ in $\bf{A\C Sch}$.
\end{corollary}

Finally, we introduce the crucial concept of modules over a $\C$-superring and their corresponding geometrical objects over an affine $\C$-scheme. Our notation throughout this part follows \cite{J}.

\begin{definition}
A \emph{module $M$ over a} $\C$-ring  $\mathfrak{C}$ is simply a $\mathfrak{C}$-module when $\mathfrak{C}$ is regarded as a $\R$-algebra. Thus, for a $\C$-ringed space $(X,\mathcal{O}_X)$ a \emph{sheaf of modules} is simply an $\mathcal{O}_X$-module.
\end{definition}

\begin{definition}\label{Mspec}
Let $(X,\mathcal{O}_X)=\emph{Spec}(\mathfrak{C})$ be an affine $\C$-scheme, and $M$ be a module over $\mathfrak{C}$. We define the \emph{$\mathcal{O}_X$-module induced by} $M$, denoted by $\mathcal E = \emph{MSpec}(M)$, as follows \cite[Definition 5.16]{J}: For any open set $U\subset X$, we take $\mathcal E(U)$ to be the real vector space of all sections

$$
e:U\rightarrow \bigsqcup_{x\in U}M\otimes_{\mathfrak C} \mathfrak{C}_x 
$$ 
such that for any $x\in U$ there exists an open subset $W$ for which $e|_{W}(x)=m\otimes 1\in M\otimes \mathfrak{C}_x$, for some fixed $m\in M$. For open sets $V\subseteq U\subseteq X$ the restriction of $e$ from $U$ to $V$ is given by $e\rightarrow e|_{V}$. 

Now for a $\mathfrak{C}$-module morphism $\alpha:M\rightarrow N$, we write $\mathcal{E}=\emph{Mspec}(M)$ and $\mathcal{F}=\emph{Mspec}(N)$, therefore for any open set $U\subseteq X$ define $\emph{Mspec}(\alpha) (U)$ by $e\rightarrow \emph{Mspec}(\alpha)(U)(e):x\rightarrow (\alpha\otimes Id)(e(x))$, for all $x\in U$
\end{definition}

\begin{definition}\label{modglobalsections}
Let $(X,\sh_X)$ be a $\C$-ringed space. We have a global sections functor

$$
\Gamma: \sh_X-Mod\rightarrow \sh_X(X)-Mod 
$$ 
defined for any $\sh_X$-module $\xi$ as $\Gamma(\xi)=\xi(X)$. Additionally, for any morphism $\alpha:\xi\rightarrow \zeta$ of $\sh_X$-modules, we have $\Gamma(\alpha):\Gamma(\xi)\rightarrow\Gamma(\zeta)$ defined as $\alpha(X):\xi(X)\rightarrow\zeta(X)$ a morphism of $\sh_X(X)$-modules.
\end{definition}

\subsection{Superrings: a quick overview}
This subsection introduces the basic framework and examples of superrings, providing the necessary background for understanding $\C$-superrings in the subsequent section. For a more detailed account of superrings see \cite{Westra, Joel}.

\begin{definition}
Let $\mathfrak{R}$ be a ring and let $\Z_2= \{ \overline 0, \overline 1 \}$ be the group of integers modulo 2. We say that $\mathfrak R$ is $\Z_2$-graded if its underlying additive group is a direct sum of abelian groups $\mathfrak R\ev$ and $\mathfrak R\od$,  $\mathfrak{R}=\mathfrak{R}\ev \oplus \mathfrak{R}\od$, such that $\mathfrak{R}_i \mathfrak{R}_j \subseteq \mathfrak{R}_{i+j}$ (mod 2) for $i, j \in \Z_2$.
\end{definition}

Let $\mathfrak{R}=\mathfrak{R}\ev \oplus \mathfrak{R}\od$ be a $\Z_2$-graded ring. The set $h(\mathfrak{R}) := \mathfrak{R}\ev \cup \mathfrak{R}\od$ is called the set of homogeneous elements of $\mathfrak{R}$. A nonzero homogeneous element $a$ is called even if $a \in \mathfrak{R}\ev$ and odd if $a \in \mathfrak{R}\od$. The parity $|a|$ of a homogeneous element $a$ is defined as: $|a|=0$ if $a \in \mathfrak{R}\ev$ or $|a|=1$ if $a \in \mathfrak{R}\od$. \\

A $\Z_2$-graded ring $\mathfrak{R}=\mathfrak{R}\ev \oplus \mathfrak{R}\od$ is called supercommutative if for any two homogeneous elements $x, y \in h(R)$, we have:
$$xy= (-1)^{|x||y|} yx $$

\begin{definition}
A superring is a $\Z_2$-graded unitary supercommutative ring.
\end{definition}

\begin{remark}
The $\Z_2$-grading of the superring $\mathfrak{R}=\mathfrak{R}\ev \oplus \mathfrak{R}\od$ implies that the even part $\mathfrak{R}\ev$ is a commutative ring and the odd part $\mathfrak{R}\od$ is an $\mathfrak{R}\ev$-module.
\end{remark}

\begin{definition}
Let $\mathfrak{R}=\mathfrak{R}\ev \oplus \mathfrak{R}\od$ and $\mathfrak S=\mathfrak S\ev \oplus \mathfrak S\od$ be superrings. A morphism of superrings $\varphi: \mathfrak{R}\rightarrow \mathfrak S$ is a ring homomorphism that preserves the $\Z_2$-grading, that is, $\varphi(\mathfrak{R}_i)\subseteq \mathfrak S_i$ for all $i \in \Z_2$.
\end{definition}

A morphism $\varphi: \mathfrak{R}\rightarrow \mathfrak{S}$ is an isomorphism if it is bijective. In this case, we say that $\mathfrak{R}$ and $\mathfrak S$ are isomorphic superrings, denoted by $\mathfrak{R}\cong \mathfrak S$.

 \begin{definition}
Let $\mathfrak{R}=\mathfrak{R}\ev \oplus \mathfrak{R}\od$ be a superring. A superideal of $\mathfrak{R}$ is a $\Z_2$-graded ideal $I$ of $\mathfrak{R}$, that is, an ideal that admits a decomposition $I=I\ev\oplus I\od$, where $I\ev=I\cap\mathfrak{R}\ev$ and $I\od=I\cap \mathfrak{R}\od$.
\end{definition}

Given a superring $\mathfrak{R}=\mathfrak{R}\ev \oplus \mathfrak{R}\od$ and a superideal $I=I\ev\oplus I\od$, the quotient ring $\mathfrak{R}/I$ inherits a natural $\Z_2$-grading $\mathfrak{R}/I=(\mathfrak{R}_0/I_0)\oplus (\mathfrak{R}_1/I_1)$. This makes $\mathfrak{R}/I$ a superring.

\begin{definition}
Let $\mathfrak{R}=\mathfrak{R}\ev\oplus \mathfrak{R}\od$ be a superring.  The ideal $\mathfrak{J}_{\mathfrak{R}}$ generated by the odd part $\mathfrak{R}\od$, that is, $\mathfrak{J}_{\mathfrak{R}}=\mathfrak{R} \mathfrak{R}\od=\mathfrak{R}\od^2\oplus \mathfrak{R}\od$, is called the canonical superideal of $\mathfrak R$. The quotient ring $\overline{\mathfrak{R}}=\mathfrak{R}/\mathfrak{J}_\mathfrak{R}$ is called the superreduced ring of $\mathfrak R$.
\end{definition}

The superreduced ring $\overline{\mathfrak{R}}=\mathfrak{R}/\mathfrak{J}_\mathfrak{R}$ is a commutative ring isomorphic to $\mathfrak{R}\ev/\mathfrak{R}\od^2$. Furthermore, if $\mathfrak{J}_\mathfrak{R}$ is finitely generated, then $\mathfrak{J}_\mathfrak{R}/\mathfrak{J}_\mathfrak{R}^2$ is a finitely generated $\overline {\mathfrak R}$-module. The elements of $\overline {\mathfrak R}$ are called purely even.

\begin{example}\label{ex:supergrass}
Let $\mathbb{K}$ be a field. The polynomial superring over $\mathbb{K}$ with even indeterminates $X^1, \ldots, X^s$ and odd indeterminates $\theta^1, \ldots, \theta^d$ is:
\begin{equation}\label{eq:superpoly}
\mathfrak{R}=\mathbb{K}[X^1, \ldots, X^s \mid  \theta^1, \ldots, \theta^d]^{\pm}:=\mathbb{K}\langle Y^1, \ldots, Y^s, Z^1, \ldots, Z^d\rangle/( Y^i Y^j-Y^j Y^i,Z^i Z^j+Z^j Z^i , Y^i Z^j-Z^j Y^i)
\end{equation}
where $X^i$ corresponds to the image of $Y^i$ and $\theta^j$ corresponds to the image of $Z^j$ in the quotient, for $i=1, \ldots, s$ and $j=1, \ldots, d$.

The superreduced ring $\overline{\mathfrak{R}}$ is isomorphic to the polynomial ring $\mathbb{K}[X^1, \ldots, X^s]$. Every element of  $\mathfrak{R}$ can be uniquely expressed as:
\begin{equation}\label{eq:superpoly2}
   f=f_{i_0}(X^1, \ldots, X^s)+\sum_{J\,:\,\text{even}}f_{i_1\cdots i_J}(X^1, \ldots, X^s)\; \theta^{i_1}\cdots\theta^{i_J}+\sum_{J\,:\,\text{odd}}f_{i_1\cdots i_J}(X^1, \ldots, X^s)\; \theta^{i_1}\cdots\theta^{i_J},
 \end{equation}
where $f_{i_0}, f_{i_1\cdots i_J}\in\mathbb{K}[X^1, \ldots, X^s]$ for all $J$.
\end{example}

\begin{example}\label{ex:supergrass2}
Let $M$ be a module over a commutative ring $R$. The \emph{exterior algebra} of $M$ over $R$, denoted $\bigwedge_R(M)$, is defined as the quotient of the tensor algebra $T_R(M)$ by the two-sided ideal $N$ generated by all elements of the form $m \otimes m$, where $m \in M$.  Thus, the exterior algebra is the direct sum:
\[
 \bigwedge_R(M):=R \oplus M \oplus \bigwedge^2 M \oplus \bigwedge^3 M \oplus \cdots
\]
where, for $p \geq 2$, $ \bigwedge^p M=(M^{\otimes p})/N^p$ and $N^p= N \cap M^{\otimes p}$.\\

Let $M$ be a finitely generated $R$-module with a basis $\{ \theta^1,\cdots, \theta^d \}$. We define the superring:
\[
R[\theta^1,\cdots, \theta^d]^{\pm} := \bigwedge_R (M)
\]
which we call the polynomial superring in the odd variables $\theta^1,\cdots, \theta^d$ over $R$. Since $\bigwedge_R (M)$
 is also an $R$-module, we may refer to $R[\theta^1,\cdots, \theta^d]^{\pm}$ as an $R$-superalgebra.
\end{example}

Prime ideals play a crucial role in the development of algebraic geometry techniques, and this remains true in the context of superrings.

 \begin{definition}
A superideal $\mathfrak{p}$ of a superring $\mathfrak{R}$ is called prime if $\mathfrak{p}$ is a proper ideal of $\mathfrak{R}$, and for any two elements $p, q \in \mathfrak R$,  $p q \in \mathfrak{p}$  implies that $p \in \mathfrak{p}$ or $q \in \mathfrak{p}$.
\end{definition}

\begin{remark}
The definition of a prime ideal in a superring coincides with the usual definition for commutative rings. This is because every prime ideal in a superring is completely prime.
\end{remark}

\begin{proposition}\label{primeideal}
Let $\mathfrak{p}$ be a prime superideal of a superring $\mathfrak{R}=\mathfrak{R}\ev \oplus \mathfrak{R}\od$. Then $\mathfrak{p}=\mathfrak p\ev \oplus \mathfrak{R}\od$, where $\mathfrak p\ev$ is a prime ideal of the commutative ring $\mathfrak R\ev$.
\end{proposition}
\begin{proof}
See \cite[Lemma 4.1.9] {Westra}
\end{proof}

\begin{remark}
Proposition \ref{primeideal} demonstrates that the prime superideals of a superring $\mathfrak{R}=\mathfrak{R}\ev \oplus \mathfrak{R}\od$ are completely determined by the prime ideals of its even part $\mathfrak{R}\ev$. 
\end{remark}

\section{$\C$-Superrings}\label{cinftysuper}
This section defines and studies the main properties of $\C$-superrings. In essence, we combine the algebraic structures of superrings and $\C$-rings, described in the previous section, into a single unified structure and explore its properties.

We establish a category of $\C$-superrings that encompasses both $\C$-rings and a broader class of superrings. We also delve into specific examples, particularly focusing on split $\C$-superrings, and investigate their categorical properties, including the existence of limits and coproducts. 

\subsection{Definition and examples}
Our definition of $\C$-superrings follows that proposed by Liu and Yau in \cite[Definition 1.4.1]{Yau 1}. 

\begin{definition}\label{def:superringnew}
A $\C$-superring is a superring $\mathfrak{R}=\mathfrak{R}\ev\oplus \mathfrak{R}\od$ whose even part $\mathfrak{R}\ev$ is endowed with a $\C$-ring structure that is compatible with the underlying ring structure of $\mathfrak{R}\ev$ (induced by $\mathfrak{R}$).
\end{definition}

Since the odd part $\mathfrak{R}\od$ of a superring $\mathfrak{R}=\mathfrak{R}\ev \oplus \mathfrak{R}\od$ is an $\mathfrak{R}\ev$-module, the odd part $\mathfrak{R}\od$ of a $\C$-superring is naturally a module over the $\C$-ring $\mathfrak{R}\ev$.

\begin{definition}\label{def:smorph}
Let $\mathfrak{R}=\mathfrak{R}\ev\oplus \mathfrak{R}\od$ and $\mathfrak S=\mathfrak S\ev\oplus \mathfrak S\od$ be $\C$-superrings. A morphism of $\C$-superrings $\varphi:\mathfrak{R}\rightarrow \mathfrak S$ is a morphism of superrings that is, a grade-preserving map, such that its restriction to the even part, $\varphi|_{\mathfrak{R}\ev}:\mathfrak{R}\ev\rightarrow \mathfrak S\ev$, is a morphism of $\C$-rings.
\end{definition}
\begin{remark}
Let $\varphi:\mathfrak{R}\longrightarrow \mathfrak{S}$ be a morphism of $\C$-superrings. According to Definition \ref{def:smorph} the restriction of $\varphi$ to the even part of $\mathfrak{R}$ is a morphism in the category of $\C$-rings. Nevertheless, the restriction to the odd part, $\varphi\od:\mathfrak{R}\od\longrightarrow \mathfrak{S}\od$, is, in principle, merely a map between a $\mathfrak{R}\ev$-module and a $\mathfrak{S}\ev$-module. This could be problematic, as we require a morphism between $\mathfrak{R}\od$ and $\mathfrak{S}\od$ that preserves a specific structure. We resolve this inconvenience by applying \emph{extension of scalars} to the morphism, as shown the lemma below.
\end{remark}

\begin{lemma}\label{escalarext}
  Let $\varphi:\mathfrak{R}\rightarrow \mathfrak{S}$ be a (grade preserving) morphism of $\C$-superrings. Then $\mathfrak{S}\od$ is endowed with an $\mathfrak{R}\ev$-module structure, and consequently, the restriction of $\varphi$ to $\mathfrak{R}\od$ is indeed an $\mathfrak{R}\ev$-modules morphism.
\end{lemma}
\begin{proof}
First, observe that $\mathfrak{S}\od$ can be naturally endowed with a $\mathfrak{R}\ev$-module structure. For any $r\ev\in\mathfrak{R}\ev $ and $t\od\in\mathfrak{S}\od $, we define the action as 
$$
r\ev\cdot t\od=\varphi\ev(r\ev)t\od, 
$$ 
where the right side uses the action of $\mathfrak{S}\ev$ on $\mathfrak{S}\od$. Consequently, it is straightforward to show that under this action, the restriction  $\varphi\od:\mathfrak{R}\od\longrightarrow \mathfrak{S}\od$ becomes a $\mathfrak{R}\ev$-module morphism.
\end{proof}
\newpage
$\C$-superrings, together with their morphisms, form a category denoted $\C$-Superrings. This category contains $\C$-Rings as a subcategory. The situation is analogous to the relationship between the categories of Rings and Superrings. Furthermore, just as in the category of $\C$-rings, where modules are defined over the underlying ring structure, modules over $\C$-superrings are defined with respect to the underlying superring structure. It is easy to verify that the category of modules over $\C$-Superrings is an abelian category. Therefore, we may employ the standard machinery of modules over superrings in this context.

\begin{example}\label{ex:trivial} 
Any $\C$-ring $\mathfrak{C}$ can be naturally viewed as a $\C$-superring by setting its odd part to zero.  This defines a functor $S: {\bf \C Rings}  \rightarrow  {\bf \C Superrings}$, which acts on objects by $S(\mathfrak C) = \mathfrak C \oplus 0$ and on morphisms in the obvious way. This functor allows us to identify  {\bf C$^\infty$Rings} as a subcategory of  {\bf C$^\infty$Superrings}.
\end{example}

Inspired by the polynomial superring over a commutative ring  with odd indeterminates (Example \ref{ex:supergrass2}), we present the following example of a $\C$-superring \cite[Example 1.1.2]{Yau 2}.

\begin{example}\label{ex:first}
As a first illustrative example of a nontrivial $\C$-superring we consider $\C(\R^{1|2})$ the superring of smooth functions on the Euclidian superspace $\R^{1|2}$, that is, the superring of smooth functions extended by two anticommuting variables $\theta^1, \theta^2$. Its even part is the $\R$-algebra $\C(\R^{1|2})\ev$ whose elements are super polynomials of the form $F=f +g \theta^1\theta^2$, where $f,g\in\C(\R)$.

Given $F_i=f_i+g_i\theta^1\theta^2\in \C(\R^{1|2})\ev$ for $i=1,\ldots,n$ and a smooth function $h\in\C(\R^n)$, we define the associated $n$-ary operation $\phi_h$ by
\begin{equation}\label{eq:firstex}
\phi_h(F_1,\ldots,F_n)=h(f_1,\ldots,f_n)+\displaystyle{\sum_{j=1}^n\partial_jh(f_1,\ldots,f_n)g_j\theta^1\theta^2}. 
\end{equation}
By the chain rule, and because the derivatives of projection maps are constant, this formula defines a $\C$-ring structure on $\C(\R^{1|2})\ev$, that is compatible with its underlying multiplication. Indeed, we only verify the compatibility of the product, as the case for addition is straightforward. Then, taking $h(x,y)=xy$, to represent standard multiplication, we have $\partial_x h=y$, $\partial_y h=x$. Therefore, for $F_1=f_1 +g_1 \theta^1\theta^2$ and $F_2=f_2 +g_2 \theta^1\theta^2$, equation \eqref{eq:firstex} and the relation $(\theta^1\theta^2)^2 = 0$ yield
$$
\phi_h(F_1, F_2)=f_1f_2+f_2g_1\theta^1\theta^2+f_1g_2\theta^1\theta^2=(f_1+g_1\theta^1\theta^2)\cdot(f_2+g_2\theta^1\theta^2)=F_1\cdot F_2
$$
which proves the compatibility of the structures.
\end{example}

\begin{example} \label{Rpq}
Consider the space of smooth functions on the Euclidean superspace of dimension $p\vert{}q$ defined by
$$
\C(\R^{p|q})=\C({\R}^p)[\theta^1,\ldots, \theta^q]^{\pm} \;
	 :=\;   \frac{C^{\infty}({\R}^p)\langle Z^1,\ldots, Z^q \rangle}{\left( fZ^{i}-Z^{i}f\,,\;Z^{i}Z^{j}+Z^{j}Z^{i}
\mid f \in\C({\R}^p);\, i, j \in\{1,\ldots, q\}\right)}, 
$$
where $\theta^i$ represents the equivalence class of the indeterminates $Z^i$. Upon applying the natural generalization of the $n$-ary operations defined in Example \ref{ex:first} to $C^\infty(\mathbb{R}^{1\vert{}2})_0$, a notable discrepancy emerges in the formulation of the $k$-ary operations. Indeed, because $C^\infty(\mathbb{R}^{p\vert{}q})$ constitutes a special case of Proposition \ref{prop:superstruct}, the correct formulation of the $k$-ary operations given by formula \eqref{eq:taylorPhi1} incorporates the full higher-order Taylor expansion of the smooth function $h$. In contrast, the $k$-ary operations defined in Example \ref{ex:first} truncate this expansion, retaining only the constant and linear terms. This apparent contradiction is easily resolved: for $q \geq 4$, defining the operations on $\C(\R^{p|q})$ via the truncation in Example \ref{ex:first} fails to yield a $\C$-superring under Definition \ref{def:superringnew} because structural compatibility no longer holds.

To explain this claim in detail, recall that elements of this quotient decompose uniquely as
$$
F=f+\sum_{|I|=2n}f^I\theta^I+\sum_{|I|=2n+1}g^I\theta^I
$$
where $I\subseteq \{1,2,\ldots,q\}$ is a multi-index, $n\in \Z$, $f,f^I,g^I\in \C(\R^p)$ and $\theta^I=\theta^{i_1}\theta^{i_2}\ldots\theta^{i_{|I|}}$. Thus, the even part $\C({\R}^p)[\theta^1,\ldots, \theta^q]^{\pm}\ev$ consists of elements $F$ with $g^I = 0$ for all odd $|I|$. If we attempt to define the $\C$-ring structure on $\C({\R}^p)[\theta^1,\ldots, \theta^q]^{\pm}\ev$ using the truncated formula
\begin{equation}\label{eq:old-formula}
\Phi_h(F_1,\ldots,F_k)= \phi_h (f_1, \ldots, f_k)\,
	         +\, \sum_{i=1}^k   (\phi_{\partial_i h}) (f_1, \ldots, f_k)\,N_i,
	         \qquad F_j=f_j+N_j,
\end{equation}
with $f_j \in \C(\R^{p})$ and $N_j \in (\C(\R^{p|q})\od)^2$, for any smooth function $h:{\R}^k\rightarrow {\R}$ and $k\in {\Bbb Z}_{\ge 1}$, this construction fails to define a $\C$-ring structure on $\C(\R^p)[\theta^1,\ldots,\theta^q]^{\pm}\ev$ as soon as $q\geq 4$. In fact, \eqref{eq:old-formula} is incompatible with the pre-existing multiplication on $\C(\R^p)[\theta^1,\ldots,\theta^q]^{\pm}\ev$, a compatibility that is required of any $\C$-ring structure extending the underlying ring structure (as seen by testing \eqref{eq:old-formula} on $h(x,y)=xy$).

To demonstrate this failure explicitly, take $p=1$ and $q=4$, so that $\C(\R)[\theta^1,\theta^2,\theta^3,\theta^4]$. Consider the two even elements
$$
F_1=f_1+\theta^1\theta^2, \qquad F_2=f_2+\theta^3\theta^4 \, ,
$$
for arbitrary $f_1,f_2\in\C(\R)$. Both elements have the form $F_j = f_j + N_j$, where $N_j$ is an even nilpotent element.

Since $\mathfrak{C}[\theta^1,\ldots,\theta^q]^{\pm}\ev$ already carries the multiplication inherited from the exterior algebra $\bigwedge_{\mathfrak C}\xi$, compatibility requires that
$$
\Phi_{xy}(F_1,F_2)=F_1 F_2=f_1f_2+f_1\,\theta^3\theta^4+f_2\,\theta^1\theta^2+\theta^1\theta^2\theta^3\theta^4 \, ,
$$
where the final term is non-zero because $\theta^1, \theta^2, \theta^3, \theta^4$ are pairwise distinct odd generators, yielding $N_1 N_2 = (\theta^1\theta^2)(\theta^3\theta^4) = \theta^1\theta^2\theta^3\theta^4 \neq 0$.

On the other hand, applying the linear formula \eqref{eq:old-formula} with $h(x,y) = xy$ (where $\partial_x h = y$ and $\partial_y h = x$) yields
$$
\Phi_{xy}(F_1,F_2)= f_1f_2+ f_2\cdot N_1 + f_1\cdot N_2 = f_1f_2+f_2\,\theta^1\theta^2+f_1\,\theta^3\theta^4 \, .
$$
Comparing the two results shows that formula \eqref{eq:old-formula} entirely omits the product $N_1 N_2 = \theta^1\theta^2\theta^3\theta^4$:
$$
\Phi_{xy}(F_1,F_2)\,=\; F_1F_2 \;-\; \theta^1\theta^2\theta^3\theta^4 \;\neq\; F_1F_2 \, .
$$
Hence, \eqref{eq:old-formula} fails to reproduce the ring multiplication on $\C(\R)[\theta^1,\theta^2,\theta^3,\theta^4]$, and cannot yield a compatible $\C$-ring structure.

This discrepancy remains hidden when $q = 2$ because the even nilpotent ideal is square-zero: the only non-zero even nilpotent monomial is $\theta^1\theta^2$, and $(\theta^1\theta^2)^2 = 0$. Consequently, $N_1 N_2 = 0$ holds automatically whenever $N_1, N_2$ are multiples of $\theta^1\theta^2$, making the linear formula \eqref{eq:old-formula} to coincide with the full Taylor expansion. However, for $q \geq 4$, distinct even nilpotent monomials such as $\theta^1\theta^2$ and $\theta^3\theta^4$ have a non-zero product. Thus, the quadratic (and higher-order) terms in the Taylor expansion no longer vanish identically.
\end{example}

The $\C$-superring of Example \ref{Rpq} plays a fundamental role in the theory of real supermanifolds, where it serves as the structure sheaf (or coordinate superring) of the Euclidean superspace $\R^{p|q}$. Therefore, from now on, we will denote the coordinate superring $\C(\R^p)[\theta^1,\ldots, \theta^q]^{\pm}$ by $\C(\R^{p|q})$.

\begin{remark}
The $\C$-superring of Example \ref{Rpq}  is a special case of what we will later call a \emph{split $\C$-superring}. We will provide the $n$-ary operations for a split $\C$-superring in Proposition \ref{prop:superstruct}.
\end{remark}

\begin{example}\label{quotsr} 
Let $\mathfrak R=\mathfrak R\ev \oplus \mathfrak R\od$ be a $\C$-superring, and let $I = I\ev \oplus I\od$ be a superideal of $\mathfrak R$. Then the quotient $\mathfrak R/I$ inherits a $\C$-superring structure from $\mathfrak R$. Since $\mathfrak R/I = (\mathfrak R\ev/I\ev) \oplus (\mathfrak R\od/I\od)$ is a superring, we only need to show that the even part $\mathfrak R\ev/I\ev$ is a $\C$-ring. This follows from the fact that $I\ev$ is an ideal of the $\C$-ring $\mathfrak R\ev$, and the quotient of a $\C$-ring by an ideal is again a $\C$-ring (see Remark \ref{rem:limits}, item 4).
\end{example}

\begin{example}
Following Example \ref{quotsr}, the quotient of a $\C$-superring $\mathfrak R=\mathfrak R\ev \oplus \mathfrak R\od$ by its canonical superideal $\mathfrak{J}_{\mathfrak{R}}=\mathfrak{R} \mathfrak{R}\od=\mathfrak{R}\od^2\oplus \mathfrak{R}\od$ has the structure of a $\C$-superring. This quotient is the superreduced ring of $\mathfrak R$, given by $\overline{\mathfrak R}=\mathfrak{R}\ev/\mathfrak{R}\od^2 \oplus 0$, which can be identified with the $\C$-ring $\mathfrak{R}\ev/\mathfrak{R}\od^2$.
\end{example}

\begin{remark}
A key distinction of $\C$-superrings is that the even component $\mathfrak{R}\ev$, while itself a $\C$-ring, may contain nilpotent elements arising from the odd part of the superring due to the relation $\mathfrak{R}\ev\cap \mathfrak{J}_{\mathfrak{R}}= \mathfrak{R}\od^2$. This contrasts with the typical focus in the theory of $\C$-rings, where such nilpotent elements are not considered.  While $\C$-rings can possess nilpotent elements (e.g., Weil algebras \cite[Proposition 1.5]{D1} or \cite[Example 2.9]{J}), the nilpotent elements of primary interest in superring theory are those originating from the odd part.
\end{remark}

\begin{example}\label{reduced-quotient}
Consider the $\C$-superring $\mathfrak{R}=\C(\R^{p|q})$ from Example \ref{Rpq}. Each element in its canonical superideal $\mathfrak{J}_\mathfrak{R}=[\C(\R^{p|q})\od]^2\oplus \C(\R^{p|q})\od$ can be expressed in the form:
$$
F=\sum_{|I|=2n}f^I\theta^I+\sum_{|I|=2n+1}g^I\theta^I.
$$
This leads to the following decomposition of the superring:
$$
\C(\R^{p|q})=\C(\R^p)\oplus [\C(\R^{p|q})\od]^2\oplus \C(\R^{p|q})\od.
$$
\end{example}
Therefore, the superreduced ring is given by $\overline{\C(\R^{p|q})}=\C(\R^p)$.\\

More generally, if $\mathfrak{R}=\C(\R^{p|q})/I$, then $\mathfrak{J}_\mathfrak{R}=[\C(\R^{p|q})\od/I\od]^2\oplus [\C(\R^{p|q})\od/I\od] $ and its elements can be written in the form 
$$
\overline{F}=\sum_{|\alpha|=2n}\overline{f}^\alpha\theta^\alpha+\sum_{|\alpha|=2n+1}\overline{g}^\alpha\theta^\alpha.
$$
The morphism $\varphi:\C(\R^{p|q})/I\longrightarrow  \C(\R^{p})/[I\cap \C(\R^{p})]$ given by 
$$
\overline{f}+\sum_{|\alpha|=2n}\overline{f}^\alpha\theta^\alpha+\sum_{|\alpha|=2n+1}\overline{g}^\alpha\theta^\alpha\longmapsto \overline{f}
$$ 
is well defined, surjective and has kernel $\mathfrak{J}_\mathfrak{R}$. In conclusion, the reduced part is equal to
$$
\overline{\C(\R^{p|q})/I}=\C(\R^{p})/[I\cap \C(\R^{p})]
$$

\subsection{Split $\C$-superrings} 
A key property of real supermanifolds is that their structure sheaves are globally isomorphic to the exterior power of a locally free sheaf \cite{Batchelor}. Supermanifolds possessing this property are called split supermanifolds. This subsection introduces the algebraic framework necessary to formalize the concept of split $\C$-superrings.\\

In a forthcoming work \cite{ORTG}, we will present a detailed study of an extension of Batchelor's structure theorem \cite{Batchelor} to a class of superspaces that includes the category of supermanifolds.\\

Given a $\C$-ring $\mathfrak{C}$ and a finitely generated $\mathfrak C$-module $\xi$, with basis  $\{\theta^1,\ldots, \theta^q\}$, we can construct the polynomial superring $\mathfrak{C}[\theta^1, \ldots, \theta^q]^{\pm}:=\bigwedge_{\mathfrak C} \xi$  in the odd variables $\theta^1,\ldots, \theta^q$ over $\mathfrak C$, following Example \ref{ex:supergrass2}. The following proposition establishes that  $\mathfrak{C}[\theta^1, \ldots, \theta^q]^{\pm}$ admits the structure of a $\C$-superring.

\begin{proposition}\label{prop:superstruct}
The superring $\mathfrak{R}=\mathfrak{C}[\theta^1, \ldots, \theta^q]^{\pm}=\bigwedge_{\mathfrak C} \xi$ carries a natural $\C$-superring structure defined as follows. For $h\in\C(\R^k)$ and $F_j=f_j+N_j \in\mathfrak R\ev$, we define the $k$-ary operation
\begin{equation}\label{eq:taylorPhi1}
\Phi_h(F_1,\ldots,F_k):=\sum_{n\geq 0}\frac{1}{n!}\sum_{j_1,\ldots,j_n=1}^{k}\phi_{\partial_{j_1}\cdots\partial_{j_n}h}(f_1,\ldots,f_k)\;N_{j_1}\cdots N_{j_n} \,  \end{equation}
where $f_j \in \overline{\mathfrak R}$ and $N_j \in \mathfrak R\od^2$. This construction satisfies the following properties:
\begin{enumerate}
\item[(i)] The sum \eqref{eq:taylorPhi1} is finite (terminating at $n=\lfloor q/2\rfloor$), and the collection $\{\Phi_h\}_{h\in\C(\R^k),\,k\in\mathbb N}$ equips $\mathfrak R\ev$ with a $\C$-ring structure.
\item[(ii)] This $\C$-ring structure is compatible with the underlying addition, ring and scalar multiplication on $\mathfrak R\ev$ induced by the exterior algebra $\bigwedge_{\mathfrak C}\xi$, that is, $\Phi_{x+y}=+$, $\Phi_{xy}=\,\cdot\,$ and $\Phi_{\lambda x}=\lambda\,(-)$ for every $\lambda\in\R$.
\end{enumerate}
In particular, $\mathfrak R=\mathfrak{C}[\theta^1,\ldots,\theta^q]^{\pm}$, equipped with the $k$-ary operations \eqref{eq:taylorPhi1}, is a $\C$-superring.  
\end{proposition}
\begin{proof}
Since each $N_j\in\mathfrak{R}\od^2$ is even, the factors in \eqref{eq:taylorPhi1} of $N_{j_1}\cdots N_{j_n}$ commute. Thus, the product depends only on the multiset of indices $\{j_1,\ldots,j_n\}$. Likewise, the operator $\partial_{j_1}\cdots \partial_{j_n}h$ is symmetric in its indices because mixed partial derivatives of smooth functions commute. Consequently, expression \eqref{eq:taylorPhi1} is well-defined. Moreover, the ideal $\mathfrak{R}\od^2$ is nilpotent with $\left(\mathfrak{R}\od^2\right)^{\lfloor q/2\rfloor+1}=0$, so any product $N_{j_1}\cdots N_{j_n}$ with $n>\lfloor q/2\rfloor$ vanishes. This establishes that the sum in \eqref{eq:taylorPhi1} is finite, proving the finiteness assertion in $(i)$. We next verify that \eqref{eq:taylorPhi1} defines a valid $k$-ary operation.

\begin{itemize}
\item{\bf Projections are preserved:} Let $p_i:\R^k\to\R$ be the $i$-th coordinate projection map. Then $\partial_j p_i=\delta_{ij}$ is constant, so all higher-order derivatives of $p_i$ (of order $\geq 2$) vanish. Consequently, only the terms $n=0,1$ contribute to \eqref{eq:taylorPhi1}:
$$
\Phi_{p_i}(F_1,\ldots,F_k)=\phi_{p_i}(f_1,\ldots,f_k)+\sum_{j=1}^k \phi_{\delta_{ij}}(f_1,\ldots,f_k)\,N_j = f_i+N_i=F_i \, ,
$$
where we used the fact that $\mathfrak C$, being a $\C$-ring, preserves coordinate projections ($\phi_{p_i}(f_1, \ldots, f_k) = f_i$) and that $\phi_c(\cdot) = c$ for any constant $c \in \mathbb{R}$.

\item\textbf{Compositions are preserved:} Let $h\in\C(\R^m)$, $g_1,\ldots,g_m\in\C(\R^k)$. We must show that
$$
\Phi_{h(g_1,\ldots,g_m)}(F_1,\ldots,F_k)=\Phi_h\bigl(\Phi_{g_1}(F_1,\ldots,F_k),\ldots,\Phi_{g_m}(F_1,\ldots,F_k)\bigr).
$$
Write $G_i:=\Phi_{g_i}(F_1,\ldots,F_k)$. Applying \eqref{eq:taylorPhi1} to $g_i$, we obtain explicit expressions for its superreduced and nilpotent parts:
$$
\overline{G_i}=\phi_{g_i}(f_1,\ldots,f_k), \qquad
\mathrm{nil}(G_i)=\sum_{n\geq 1}\frac{1}{n!}\sum_{j_1,\ldots,j_n=1}^k \phi_{\partial_{j_1}\cdots\partial_{j_n}g_i}(f_1,\ldots,f_k)\,N_{j_1}\cdots N_{j_n} \, .
$$
Substituting these expressions into \eqref{eq:taylorPhi1} for $\Phi_h(G_1,\ldots,G_m)$ and collecting terms according to their total order $n = n_1 + \dots + n_r$ in the $N_j$'s, the coefficient of $N_{j_1}\cdots N_{j_n}$ (summed over the index distributions among the $r$ factors $\mathrm{nil}(G_{i_1}),\ldots,\mathrm{nil}(G_{i_r})$ with appropriate factorials) equals
$$
\sum_{r\geq 0}\frac{1}{r!}\sum_{i_1,\ldots,i_r=1}^m \phi_{\partial_{i_1}\cdots\partial_{i_r}h}\bigl(\phi_{g_1}(f_1,\ldots,f_k),\ldots,\phi_{g_m}(f_1,\ldots,f_k)\bigr)\cdot\Bigl[\text{coefficient of } N_{j_1}\cdots N_{j_n} \text{ in } \textstyle\prod_{l=1}^r \mathrm{nil}(G_{i_l})\Bigr].
$$
By the higher-order chain rule (Faà di Bruno's formula), this coincides, order by order in $n$, with $\phi_{\partial_{j_1}\cdots\partial_{j_n}[h(g_1,\ldots,g_m)]}(f_1,\ldots,f_k)$. Indeed, $\partial_{j_1}\cdots\partial_{j_n}[h(g_1,\ldots,g_m)]$ is a universal polynomial expression in the partial derivatives $\partial^\bullet h$ evaluated at $(g_1,\ldots,g_m)$ and those of $g_i$ evaluated at $(x_1,\ldots,x_k)$. Since $\mathfrak C$ is a $\C$-ring, applying $\phi$ to such a composite smooth function agrees with evaluating the corresponding polynomial in the images under $\phi$, by definition.  Therefore
$$
\Phi_h(G_1,\ldots,G_m)=\sum_{n\geq 0}\frac{1}{n!}\sum_{j_1,\ldots,j_n=1}^k \phi_{\partial_{j_1}\cdots\partial_{j_n}[h(g_1,\ldots,g_m)]}(f_1,\ldots,f_k)\,N_{j_1}\cdots N_{j_n}=\Phi_{h(g_1,\ldots,g_m)}(F_1,\ldots,F_k) \, ,
$$
as required.
\end{itemize}
This completes the proof of item $(i)$.

To prove item $(ii)$, namely, the compatibility with the ring structure, we first consider addition, given by $h(x,y)=x+y$. All partial derivatives of order $\geq 2$ vanish, so \eqref{eq:taylorPhi1} truncates at $n=1$:
$$
\Phi_{x+y}(F_1,F_2)=f_1+f_2+N_1+N_2=F_1+F_2 \, .
$$
Next consider scalar multiplication by $\lambda\in\R$, given by $h(x)=\lambda x$. Likewise, only the terms for $n=0,1$ contribute:
$$
\Phi_{\lambda x}(F)=\lambda f+\lambda N=\lambda F \, .
$$
Finally, let $h(x,y)=xy$ represent ring multiplication. Then, $\partial_x h=y$, $\partial_y h=x$, $\partial_x\partial_y h=\partial_y\partial_x h=1$, and all partial derivatives of order $\geq 3$ vanish. Thus, \eqref{eq:taylorPhi1} has exactly three contributions:
\begin{itemize}
\item The term $n=0$ yields $\phi_{xy}(f_1,f_2)=f_1f_2$. 
\item The term $n=1$ yields $\phi_{\partial_x h}(f_1,f_2)\,N_1+\phi_{\partial_y h}(f_1,f_2)\,N_2=f_2N_1+f_1N_2$.
\item The term $n=2$ summing over the two orderings $(j_1,j_2) = (1,2)$ and $(2,1)$ (each contributing $\phi_{\partial_x\partial_y h}(f_1,f_2) = 1$), yields
$$
\frac{1}{2!}\Bigl(\phi_{\partial_1\partial_2 h}(f_1,f_2)\,N_1N_2+\phi_{\partial_2\partial_1 h}(f_1,f_2)\,N_2N_1\Bigr)=N_1N_2 \, .
$$
\end{itemize}
Combining these terms gives
$$
\Phi_{xy}(F_1,F_2)=f_1f_2+f_2N_1+f_1N_2+N_1N_2=(f_1+N_1)(f_2+N_2)=F_1F_2 \, ,
$$
which coincides with the pre-existing product on $\mathfrak R\ev$. This completes the proof of item $(ii)$.

Bringing together items $(i)$ and $(ii)$, we conclude that the collection $\Phi=\{\Phi_h\}$ defines a $\C$-ring structure on $\mathfrak R\ev$ compatible with its underlying ring structure. By Definition \ref{def:superringnew}, this completes the proof that $\mathfrak R=\mathfrak C[\theta^1,\ldots,\theta^q]^{\pm}$ is a $\C$-superring.
\end{proof}

\begin{definition}\label{splitsuperring}
Let $\mathfrak C$ be a $\C$-ring and $\xi$ a finitely generated $\mathfrak C$-module with basis $\{ \theta^1, \cdots, \theta^q \}$. A $\C$-superring $\mathfrak R$ is said to be \emph{split} if it is isomorphic to a $\C$superring of the form $\mathfrak{C}[\theta^1, \ldots, \theta^q]^{\pm}=\bigwedge_{\mathfrak C} \xi$ for some $q\geq 0 \in \Z$.
\end{definition}

\begin{remark}
Let $\mathfrak R= \mathfrak R\ev \oplus \mathfrak R\od$ be a superring such that $\mathfrak R\od$ is a finitely generated $\mathfrak R\ev$-module. The superring $\mathfrak R$ admits a $\mathfrak J_{\mathfrak R}$-adic filtration of length $q$:
\[
\mathfrak R=:\mathfrak J_{\mathfrak R}^0 \supset \mathfrak J_{\mathfrak R} \supset \mathfrak J_{\mathfrak R}^2 \supset \cdots \supset \mathfrak J_{\mathfrak R}^q \supset \mathfrak J_{\mathfrak R}^{q+1}=0.
\]
Define the $\mathbb Z$-graded superring $Gr\; \mathfrak R$ associated to $\mathfrak R$ by:
\[
Gr\; \mathfrak R := \frac{\mathfrak R}{\mathfrak J_{\mathfrak R}} \oplus \frac{\mathfrak J_{\mathfrak R}}{\mathfrak J_{\mathfrak R}^2 }\oplus \cdots \oplus \frac{\mathfrak J_{\mathfrak R}^{q-1}}{\mathfrak J_{\mathfrak R}^q} \oplus \mathfrak J_{\mathfrak R}^q.
\]

Denote the commutative ring $\mathfrak R/\mathfrak J_{\mathfrak R}$  by $R$ and the $(\mathfrak R/\mathfrak J_{\mathfrak R})$-module $\mathfrak J_{\mathfrak R}/\mathfrak J_{\mathfrak R}^2$ by $\xi$. We have that $\xi$ is a free $R$-module with finite rank and $Gr\; \mathfrak R$  is isomorphic to $\bigwedge_R \xi$. Therefore, $Gr\; \mathfrak R$ is split.\\

Note that $\mathfrak J_{\mathfrak R}^n/\mathfrak J_{\mathfrak R}^{n+1} \cong \mathfrak R\od^{n}/\mathfrak R\od^{n+2} $ because $\mathfrak J_{\mathfrak R}^n=\begin{cases}
			\mathfrak R\od^{n+1} \oplus \mathfrak R\od^{n},  & \text{if $n$ odd}\\
            \mathfrak R\od^{n} \oplus \mathfrak R\od^{n+1}, & \text{if $n$ even}
		 \end{cases}$. \\
Then, $Gr\;  \mathfrak R$ can be rewritten to explicitly show the $\Z_2$-grading. For $q$ odd:
\[
Gr \; \mathfrak R \cong \left( \frac{\mathfrak R\ev} {\mathfrak R\od^2}  \oplus \frac{\mathfrak R\od^2} {\mathfrak R\od^4} \cdots \oplus \frac{\mathfrak R\od^{q-1}}{\mathfrak R\od^{q+1}} \oplus \mathfrak R\od^{q+1} \right) 
\oplus \left(  \frac{\mathfrak R\od^1}{\mathfrak R\od^3}  \oplus \frac{\mathfrak R\od^3}{\mathfrak R\od^5} \cdots \oplus \frac{\mathfrak R\od^{q-2}}{\mathfrak R\od^{q}} \oplus \mathfrak R\od^{q} \right) \, ,
\]
For $q$ even:
\[
Gr \; \mathfrak R \cong \left( \frac{\mathfrak R\ev} {\mathfrak R\od^2}  \oplus \frac{\mathfrak R\od^2} {\mathfrak R\od^4} \cdots \oplus \frac{\mathfrak R\od^{q-2}}{\mathfrak R\od^{q}} \oplus \mathfrak R\od^{q} \right) 
\oplus \left(  \frac{\mathfrak R\od^1}{\mathfrak R\od^3}  \oplus \frac{\mathfrak R\od^3}{\mathfrak R\ev^5} \cdots \oplus \frac{\mathfrak R\od^{q-1}}{\mathfrak R\od^{q+1}} \oplus \mathfrak R\od^{q+1} \right) \, .\\
\]
\end{remark}

\begin{remark}
The superreduced ring of a split $\C$-superring $\mathfrak{R} \cong \mathfrak{C}[\theta^1, \ldots, \theta^n]^{\pm}$ is isomorphic to the $\C$-ring $\mathfrak C$.
\end{remark}

The split $\C$-superrings form a full subcategory of the category of $\C$-superrings, which we denote by {\bf SplC$^\infty$Superrings}.

\begin{definition}\label{def:superfunct}
For any fixed $q\in \mathbb{Z}^+$, we define a functor $T:{\bf \C Rings}  \rightarrow  {\bf \C Superrings}$ as follows:
\begin{itemize}
\item {\bf On objects:} For a $\C$-ring $\mathfrak{C}$, $T(\mathfrak C)=\mathfrak{C}[\theta^1,\ldots,\theta^q]^{\pm}$.

\item {\bf On morphisms:} For a morphism of $\C$-rings $\varphi:\mathfrak{C}\rightarrow\mathfrak{D}$, we define the morphism of $\C$-superrings $T(\varphi):\mathfrak{C}[\theta^1,\ldots,\theta^q]^{\pm}\longrightarrow \mathfrak{D}[\eta^1,\ldots,\eta^q]^{\pm}$ by:
$$
T(\varphi)\left( c_j+\displaystyle{\sum_{|I_j|=2n}c_{I_j}\theta^{I_j}}+\displaystyle{\sum_{|I_j|=2n+1}d_{I_j}\theta^{I_j}} \right)= \varphi(c_j)+\displaystyle{\sum_{|I_j|=2n}\varphi(c_{I_j}) \, \eta^{I_j}}+\displaystyle{\sum_{|I_j|=2n+1}\varphi(d_{I_j}) \, \eta^{I_j}}
$$
\end{itemize}
\end{definition} 

\begin{remark}
Definition \ref{def:superfunct} illustrates a specific instance of a more general construction known as the superization functor, introduced by Carchedi in \cite{Carchedi} for Fermat Theories. This powerful tool provides a systematic method for extending algebraic theories to their superalgebraic counterparts, incorporating both commutative and anticommutative elements.
\end{remark}

\begin{remark}\label{rem:split}  
Given a superring $\mathfrak{R}=\mathfrak{R}\ev\oplus\mathfrak{R}\od$, the following \emph{structural} short exact sequence of $\mathfrak{R}$-modules:
\begin{equation}\label{eq:split}  
0\longrightarrow \mathfrak{J}_{\mathfrak{R}} \stackrel{\iota} \longrightarrow \mathfrak{R}\stackrel{\pi}{\longrightarrow} \overline{\mathfrak{R}}\longrightarrow 0
\end{equation}
satisfies the splitting lemma, meaning the following three statements are equivalent:
\begin{itemize}
\item $\mathfrak{R}\cong \overline{\mathfrak{R}}\oplus \mathfrak{J}_{\mathfrak{R}}= (\overline{\mathfrak{R}} \oplus \mathfrak{R}\od^2)\oplus \mathfrak{R}\od$, where $\iota$ is the inclusion of $\mathfrak{J}_{\mathfrak{R}}$ into the second summand, and $\pi$ is the projection onto the first.

\item $\pi$ has a right inverse, i.e., there exists a homomorphism $\sigma: \overline{\mathfrak{R}} \rightarrow \mathfrak{R}$ such that $\pi\circ\sigma=Id_{\overline{\mathfrak{R}}}$.

\item $\iota$ has a left inverse, i.e., there exists a homomorphism $\kappa : \mathfrak R \rightarrow \mathfrak{J}_{\mathfrak{R}}$ such that $\kappa \circ \iota=Id_{\mathfrak J_{\mathfrak R}}$
\end{itemize}

If any of these statements hold, the sequence is called a \emph{split exact sequence}, or simply said to \emph{split}.
\end{remark}   

The following proposition justifies the use of the term ``split" for both a superring and its structural exact sequence.

\begin{lemma} A $\C$-superring $\mathfrak{R}$ with finitely generated odd part is split if and only if the canonical short exact sequence (\ref{eq:split}) is split.
\end{lemma}

\begin{proof}
Suppose that $\mathfrak{R}$ is a $\C$-superring split, that is, $\mathfrak{R}\cong\mathfrak{C}[\theta^1, \ldots, \theta^n]^{\pm}$, for some $\C$-ring $\mathfrak{C}$, then, as we pointed in Proposition \ref{prop:superstruct}, the elements of $\mathfrak{R}$ are of the form 
$$
C=c_0+\sum_{|I|=2n}c_I \, \theta^I+\sum_{|K|=2n+1}d_K\, \theta^K \, ,
$$
Where $c_0,c_I,d_K$ are elements of $\mathfrak{C}$ and $I,K$ are both multi-indexes of even and odd size respectively. 

Thus, with this description of the elements it is easy to see that $\overline{\mathfrak{R}}$ can be identified with $\mathfrak{C}$, the elements of $\mathfrak{R}\ev$ coming from the odd variables are of the form $\sum_{|I|=2n}c_I\, \theta^I$ and the elements of $\mathfrak{R}\od$ are of the form $\sum_{|K|=2n+1}d_K\, \theta^K \,$

This allows us to conclude that $\mathfrak{R}\cong (\overline{\mathfrak{R}} \oplus \mathfrak{R}\od^2)\oplus \mathfrak{R}\od$, which is equivalent to saying that the canonical exact sequence splits.

Conversely, suppose that $\mathfrak{R}\od$ is generated by the odd variables $\xi^1\ldots\xi^q$ and the sequence (\ref{eq:split}) is split, that is $\mathfrak{R}$ splits into a sum $\overline{\mathfrak{R}}\oplus\mathfrak{J}_\mathfrak{R}$. We denote by $\theta^i$ the image of $\xi^i$ under the quotient projection $\mathfrak{R}\rightarrow \overline{\mathfrak{R}}$. Therefore, for any multi-index $(i_1\cdots i_k)$ with $1\leq k\leq q$ the product $\xi^{i_1}\ldots\xi^{i_k}$ can be identified with the product $\theta^{i_1}\ldots\theta^{i_k}$ 
 in $\frac{\mathfrak J_{\mathfrak R}}{\mathfrak J_{\mathfrak R}^2 }\oplus \cdots \oplus \frac{\mathfrak J_{\mathfrak R}^{q-1}}{\mathfrak J_{\mathfrak R}^q} \oplus \mathfrak J_{\mathfrak R}^q$ which induces a $\overline{\mathfrak{R}}$-module isomorphism between ${\mathfrak J_{\mathfrak R}} $ and $\frac{\mathfrak J_{\mathfrak R}}{\mathfrak J_{\mathfrak R}^2 }\oplus \cdots \oplus \frac{\mathfrak J_{\mathfrak R}^{q-1}}{\mathfrak J_{\mathfrak R}^q} \oplus \mathfrak J_{\mathfrak R}^q$. Therefore $\mathfrak R$ is isomorphic to its associated graded superring which is split. 
\end{proof}

\begin{lemma}
 Let $\mathfrak{R}$ be a split superring and $\mathfrak{I}$ a superideal of $\mathfrak{R}$ such that $\mathfrak I \subseteq \mathfrak{J}_{\mathfrak{R}}$. Then the quotient superring $\mathfrak{R}/\mathfrak{I}$ is also split.
\end{lemma}
\begin{proof}
Since  $\mathfrak{J}_{\mathfrak{R}}=\mathfrak{R}\od^2 \oplus \mathfrak{R}\od$ and $\mathfrak{I} \subseteq \mathfrak{J}_{\mathfrak{R}}$, it follows that $\mathfrak{I}\ev\subseteq \mathfrak{R}\od^2$ and $\mathfrak{I}\od\subseteq \mathfrak{R}\od$. Therefore, the superreduced ring of $\mathfrak{R}$ is isomorphic to:
\begin{equation}\label{eq:quotreduc}
\overline{\mathfrak{R}}
\cong \mathfrak{R}\ev/\mathfrak{R}\od^2\cong\frac{\mathfrak{R}\ev/\mathfrak{I}\ev}{\mathfrak{R}\od^2/\mathfrak{I}\ev}
\cong\frac{\mathfrak{R}\ev/\mathfrak{I}\ev\oplus \mathfrak{R}\od/\mathfrak{I}\od}{\mathfrak{R}\od^2/\mathfrak{I}\ev\oplus \mathfrak{R}\od/\mathfrak{I}\od}
\cong\frac{\mathfrak{R}/\mathfrak{I}}{\mathfrak{J}_{\mathfrak{R}}/\mathfrak{I}}   
\end{equation}
Additionally, the ideal $\mathfrak{J}_{\mathfrak{R}}/\mathfrak{I}$ in the quotient ring $\mathfrak R/ \mathfrak I$ is equal to the canonical superideal $\mathfrak{J}_{\mathfrak{R}/\mathfrak{I}}$ of $\mathfrak R/ \mathfrak I$. Thus,
$$
\overline{\mathfrak{R}}\cong\frac{\mathfrak{R}/\mathfrak{I}}{\mathfrak{J}_{\mathfrak{R}}/\mathfrak{I}}=\overline{\mathfrak{R}/\mathfrak{I}}.
$$
Finally, since $\mathfrak R$ is split, we have $\mathfrak R\ev=\overline{\mathfrak R} \oplus \mathfrak R\od^2$. As $\mathfrak I\ev \subseteq \mathfrak R\od^2$, it follows that:
$$
\mathfrak{R}/\mathfrak{I}=\mathfrak{R}\ev/\mathfrak{I}\ev\oplus\mathfrak{R}\od/\mathfrak{I}\od\cong\overline{\mathfrak{R}}\oplus\mathfrak{R}\od^2/\mathfrak{I}\ev\oplus\mathfrak{R}\od/\mathfrak{I}\od\cong\overline{\mathfrak{R}/\mathfrak{I}}\oplus\mathfrak{R}\od^2/\mathfrak{I}\ev\oplus\mathfrak{R}\od/\mathfrak{I}\od.
$$
We have $\mathfrak{R}/\mathfrak{I}\cong\overline{\mathfrak{R}/\mathfrak{I}}\oplus\mathfrak{J}_{\mathfrak{R}/\mathfrak{I}}$. Therefore, $\mathfrak R/ \mathfrak I$ is a split superring.
\end{proof}

\begin{lemma}
Let $\mathfrak{R}$ be a split superring with splitting morphism $\sigma: \overline{\mathfrak{R}}\rightarrow \mathfrak{R}$. Let $\mathfrak{R}\rightarrow\mathfrak{T}$ be a surjective morphism of superrings with kernel $\mathfrak{I}$, and denote $\overline{\mathfrak{I}}$ the image of $\mathfrak{I}$ in $\overline{\mathfrak{R}}$ . If $\sigma(\overline{\mathfrak{I}}) \subseteq \mathfrak{I}$, then $\mathfrak{T}$ is a split superring
\end{lemma}
\begin{proof}
See \cite[Proposition 3.5.4]{Westra}
\end{proof}

\begin{remark}
The property of being split is not necessarily preserved under quotients of $\C$-superrings, as demonstrated by Example \ref{nonsplit} below. This observation highlights the complexities inherent in constructing quotients of $\C$-superrings and motivates a more thorough investigation into the conditions under which the quotient remains split. Characterizing these conditions represents a promising avenue for future research, as a comprehensive understanding of the behavior of quotients could lead to a more refined classification of superspace structures.
\end{remark}

The following example generalizes one found in \cite[page 22]{Westra}, illustrating a case where the quotient of a split superring is not split.

\begin{example}\label{nonsplit}
Let $\mathfrak{I}$ be the ideal of $\C(\R^{1|2})$ generated by $x^2+ \theta^1\theta^2$. Consider the superring $\mathfrak{R}/\mathfrak{I}$. We will show that $\mathfrak{R}$ is not split.\\

Recall that elements of $\mathfrak{R}=\C(\R^{1|2})$ are of the form 
$f_0(x)+f_{12}(x)\; \theta^1\theta^2 +f_1(x)\;\theta^1+f_2(x)\;\theta^2$, where $f_0, f_1, f_2, f_{12} \in\C({\R})$. Note that the odd parts $(\C(\R^{1|2}) )\od$ and $\mathfrak{R}\od$ can be identified because in the quotient $\theta^1 \theta^2$ is equivalent to $-x^2$. With this identification, $x^2$ becomes nilpotent in the quotient $\mathfrak{R}$.  Specifically, $\mathfrak{R}\ev$ is isomorphic to $\C(\R)/(x^4)$.\\
Observe that the superreduced ring  $\overline{\mathfrak R}$ is isomorphic to $\C(\R)/(x^4)$.  Consequently, we have a surjective morphism:
$$\C(\R)/(x^4)\rightarrow\frac{\C(\R)/(x^4)}{(x^2)/(x^4)})=\frac{\C(\R)}{(x^2)} \, .$$
This morphism corresponds to the canonical projection $\mathfrak{R}\ev\rightarrow \overline{\mathfrak{R}}$. However, there is no injective morphism:
$$\C(\R)/(x^2)\rightarrow \C(\R)/(x^4) \, .$$
Thus, a right inverse to the canonical projection $\mathfrak R \rightarrow \overline{\mathfrak R}$ cannot be defined. Consequently, the canonical short exact sequence of $\mathfrak R$ does not split in this case, meaning $\mathfrak R$ is not a split superring
\end{example}

\subsection{Categorical properties}
There is a natural functor, defined in \cite[page 14]{Westra} for general superrings, that can be extended to a functor $\mathcal{F}:{\bf \C Superrings} \rightarrow {\bf \C Rings}$. On objects, this functor is defined by $\mathfrak{R}\longmapsto\overline{\mathfrak{R}}$. On morphisms $\varphi:\mathfrak{R}\rightarrow\mathfrak{S}$, the functor assigns the morphism $\overline{\varphi}:\overline{\mathfrak{R}}\rightarrow \overline{\mathfrak{S}}$ given by $\overline{\varphi}(\overline{r})=\overline{\varphi(r)}$, which is well-defined because $\varphi(\mathfrak{J}_{\mathfrak{R}})\subseteq \mathfrak{J}_{\mathfrak{S}}$. This induces the following commutative diagram:
\[
\begin{tikzcd}
\mathfrak R\arrow{r}{q_{\mathfrak R}} \arrow[swap]{d}{\varphi} & {\overline{\mathfrak R}=\mathcal F(\mathfrak R)} \arrow{d}{\overline{\varphi} } \\
 \mathfrak S \arrow[swap]{r}{ q_{\mathfrak S}}  & {\overline{\mathfrak S}=\mathcal F(\mathfrak S)}
\end{tikzcd}
\]

Note that the morphism $\overline \varphi$ is indeed a morphism of $\C$-rings. For $\overline{c_1},\ldots, \overline{c_n} \in \overline{\mathfrak{R}}\cong\mathfrak{R}\ev/\mathfrak{R}\od^2$, $h\in \C(\R^n)$, $\phi_h: \mathfrak R\ev^n \rightarrow \mathfrak R\ev$, and $\psi_h: \mathfrak S\ev^n \rightarrow \mathfrak S\ev$, we have
$$
\overline{\varphi}(\phi_h(\overline{c_1},\ldots, \overline{c_n}))=\overline{\varphi}(\overline{\phi_h(c_1\ldots, c_n)})=\overline{\varphi(\phi_h(c_1\ldots, c_n))}=\overline{\psi_h(\varphi(c_0),\ldots, \varphi(c_n))}=\psi_h(\overline{\varphi}(\overline c_0),\ldots \overline{\varphi}(\overline c_0)).
$$
Thus, we obtain the commutative diagram
\[
\begin{tikzcd}
\mathfrak R\ev^n\arrow{r}{\phi_h} \arrow[swap]{d}{\varphi^n}  &\mathfrak R\ev\arrow{r}{q_{\mathfrak R}} \arrow[swap]{d}{\varphi} & \overline{\mathfrak R} \arrow{d}{\overline{\varphi} } & \overline{\mathfrak R}^n\arrow[swap]{l}{\phi_h}  \arrow{d}{\overline \varphi^n}\\
 \mathfrak S\ev^n \arrow[swap]{r}{ \psi_h}  & \mathfrak S\ev \arrow[swap]{r}{ q_{\mathfrak S}}  & \overline{\mathfrak S} & \overline{\mathfrak S}^n\arrow{l}{\phi_h}
\end{tikzcd}
\]

In contrast, we have a functor $\mathcal{G}:{\bf \C Rings}  \rightarrow  {\bf \C Superrings}$ defined by adding a zero odd part to both objects and morphisms. The following proposition is related to \cite[Proposition 3.1.10]{Westra}.

\begin{proposition}
The functors $\mathcal{F}$ and $\mathcal{G}$ form an adjoint pair $(\mathcal{F},\mathcal{G})$.
\end{proposition}

\begin{proof}
 Let $\mathfrak{R}$ be a $\C$-superring and $\mathfrak{C}$ a $\C$-ring. Observe that for any $\varphi \in \text{Hom}_{\bf \C Superrings}(\mathfrak{R}, \mathcal{G}(\mathfrak{C}))$, we have $\varphi (\mathfrak{J}_{\mathfrak{R}})=0$, since $(\mathcal{G}(\mathfrak{C}))\od=0$. Therefore, $\overline{\varphi}$ is the unique morphism that makes the following diagram commute:
\[
\begin{tikzcd}
\mathfrak R\arrow{r}{q_{\mathfrak R}} \arrow[swap]{d}{\varphi} & {\overline{\mathfrak R}=\mathcal F(\mathfrak R)} \arrow{d}{\overline{\varphi} } \\
 \mathcal G(\mathfrak{C})= \mathfrak{C} \oplus 0  \arrow[swap]{r}{\cong}  & \mathfrak C
\end{tikzcd}
\]
Hence, we define a bijection $\mu: \text{Hom}_{\bf \C Superrings}( \mathfrak{R},\mathcal G (\mathfrak{C})) \rightarrow \text{Hom}_{\bf \C Rings}(\mathcal F(\mathfrak{R}), \mathfrak{C})$ as $\varphi \mapsto \overline{\varphi}$. The inverse of $\mu$ is given by the composition $\upsilon \circ q_{\mathfrak R}$
  for $\upsilon \in \text{Hom}_{\bf \C Rings}(\mathcal{F}(\mathfrak{R}),\mathfrak{C})$.\\

Now, let $\chi:\mathfrak{S}\rightarrow \mathfrak{R}$ be a morphism of $\C$-superrings and $\psi:\mathfrak{C}\rightarrow \mathfrak{D} $ a morphism of $\C$-rings. The following diagram commutes for both directions of the horizontal arrows, establishing the naturality of the adjunction:

\begin{center}
\begin{tikzcd}[row sep=3cm,column sep=2.5cm]
   \text{Hom}_{\bf \C Superrings}( \mathfrak{R},\mathcal G (\mathfrak{C}))
     \arrow[transform canvas={yshift=.8ex}]{r}{\varphi_{\mathfrak R} \mapsto \overline{\varphi_{\mathfrak R}}}
     \arrow[transform canvas={yshift=-.8ex},leftarrow]{r}[swap]{\upsilon_{\overline{\mathfrak R}} \circ q_{\mathfrak R} \leftmapsto \upsilon_{\overline{\mathfrak R}} }
     \arrow{d}[swap]{\varphi_{\mathfrak R} \mapsto \varphi_{\mathfrak S}}
 &\text{Hom}_{\bf \C Rings}( \mathcal F(\mathfrak{R}), \mathfrak{C})
     \arrow{d}{\upsilon_{\overline{\mathfrak R}} \mapsto \upsilon_{\overline{\mathfrak S}}}
      \\
     \text{Hom}_{\bf \C Superrings}( \mathfrak{S},\mathcal G (\mathfrak{D}))
     \arrow[transform canvas={yshift=.8ex}]{r}{\varphi_{\mathfrak S} \mapsto \overline{\varphi_{\mathfrak S}}}
     \arrow[transform canvas={yshift=-.8ex},leftarrow]{r}[swap]{\upsilon_{\overline{\mathfrak R}} \circ q_{\mathfrak S} \leftmapsto \upsilon_{\overline{\mathfrak R}} }
 &  \text{Hom}_{\bf \C Rings}( \mathcal F(\mathfrak{S}), \mathfrak{D})
\end{tikzcd}
\end{center}

where the vertical arrows are given by pre-composition with $\chi$ and post-composition with $\psi$, respectively, as shown in the following commutative diagrams:

\begin{figure}[htb]
\centering
\begin{minipage}{0.48 \textwidth}
\centering
\begin{tikzcd}
\mathfrak{R} \arrow{r}{\varphi_\mathfrak{R}}  & {\mathcal G (\mathfrak{C})=\mathfrak C \oplus 0} \arrow{d}{\psi} \\
\mathfrak{S} \arrow{u}{\chi} \arrow{r}{\varphi_{\mathfrak S}}& {\mathcal G (\mathfrak{D})=\mathfrak D \oplus 0}
\end{tikzcd}
\end{minipage}
\centering  
\begin{minipage}{0.48 \textwidth}
\centering 
\begin{tikzcd}
{\overline{\mathfrak{R}}=\mathcal F(\mathfrak R)} \arrow{r}{\upsilon_{\overline{\mathfrak R}}}  & \mathfrak{C}\arrow{d}{\psi} \\
{\overline{\mathfrak{S}}=\mathcal F(\mathfrak S)} \arrow{u}{\overline{\chi}} \arrow{r}{\upsilon_{\overline{\mathfrak S}}}& \mathfrak{D}
\end{tikzcd}
\end{minipage}
\end{figure}
\end{proof}

\begin{lemma}\label{limit}
The category ${\bf \C Superrings}$ is complete and finitely co-complete. 
\end{lemma}

\begin{proof} 

We first prove completeness. Consider $\{\mathfrak{R}_i,f_{ij} \}_{i,j\in I}$ an inverse system of $\C$-superrings over the partially 
ordered set $I$. Each $\mathfrak{R}_i$ has the form $\mathfrak{R}_{i \bar 0}\oplus \mathfrak{R}_{i \bar 1}$, where the even part $\mathfrak{R}_{i \bar 0}$ is a $\C$-ring, the odd part $\mathfrak{R}_{i \bar 1}$ is a $\mathfrak{R}_{i \bar 0}$-module, and both parts are $\R$-modules. Since all of the $\mathfrak{R}_{i \bar 0}$ are $\C$-rings, their inverse limit exists (see \cite[Proposition 2.5]{J}) and we denote it as $\mathfrak{R}{\ev}$. The inverse limit $\mathfrak{R}\ev$ is equipped with natural projections $\pi_{i \bar 0}: \mathfrak{R}\ev\rightarrow\mathfrak{R}_{i \bar 0}$.\\

Note that the $\mathfrak{R}_{i \bar 1}$'s are $\mathfrak R_{\bar 0}$-modules, with an action defined as $a{\ev} a_{i\bar 0}:=\left( \pi_{i \bar 0}(a{\ev})  \right) a_{i\bar 1}$, for all $i\in I$, $a{\ev}\in\mathfrak{R}{\ev}$, and $a_{i\bar 1}\in \mathfrak{R}_{i \bar 1}$ . Consequently, we have an inverse system of $\mathfrak{R}{\ev}$-modules $\{\mathfrak{R}_{i \bar 1}\}$, whose inverse limit in the category of $\mathfrak{R}\ev$-modules exists, and we denote it by $\mathfrak{R}{\od}$.\\

Therefore, we have the superring $\mathfrak{R}=\mathfrak{R}\ev\oplus \mathfrak{R}\od$, where the even part is a $\C$-ring, that is, a $\C$-superring. Furthermore, we obtain the following commutative diagram

\[
\begin{tikzcd}
& \mathfrak{R}\ev\oplus \mathfrak{R}\od \arrow[swap]{dl}{\pi_{j \bar 0} \oplus \pi_{j \bar 1}} \arrow{dr}{\pi_{i \bar 0} \oplus \pi_{i \bar 1}} \\
\mathfrak R_{j \bar 0} \oplus \mathfrak R_{j \bar 1}  \arrow[swap]{rr}{f_{ij \bar 0} \oplus f_{ij \bar 1}} & & \mathfrak R_{i \bar 0} \oplus \mathfrak R_{i \bar 1}
\end{tikzcd}
\]

where the pair $(\mathfrak R, \pi_i)$ is universal. Thus, $\mathfrak{R}$ is the limit of  $\{\mathfrak{R}_i,f_{ij} \}$ in the category of ${\bf \C Superrings}$.\\

In the case of finite colimits, the category of $\C$-rings is known to be finitely co-complete. Consequently, the proof for the existence of finite colimits in the category of $\C$-superrings closely parallels the approach used for proving the existence of limits, allowing us to apply similar reasoning.
\end{proof}

\begin{proposition}\label{super-coproduct}
Let $\mathfrak{C}$ and $\mathfrak{D}$ be $\C$-rings. The coproduct of the split $\C$-superrings $\mathfrak{R}=\mathfrak{C}[\theta^1,\ldots, \theta^q]^{\pm}$ and $\mathfrak{S}=\mathfrak{D}[\xi^1,\ldots,\xi^n]^{\pm}$ (or pushout $\mathfrak R \sqcup_{\mathbb R} \mathfrak S$) in the category ${\bf \C Superrings}$ is given by  
$$
\mathfrak{R}\otimes_{\infty s} \mathfrak{S}:=(\mathfrak{C}\otimes_\infty\mathfrak{D})[\theta^1,\ldots, \theta^q,\xi^1,\ldots,\xi^n]^{\pm},
$$
where $\mathfrak{C}\otimes_\infty\mathfrak{D}$ denotes the coproduct in the category ${\bf \C Rings}$.
\end{proposition}

\begin{proof}
Consider the commutative diagram of coproducts of $\C$-rings:
\begin{equation}\label{eq:localmor}
\begin{tikzcd}
\R \arrow[r,hook] \arrow[d,hook]
& \mathfrak{D}\arrow{d}{\gamma} \\
\mathfrak{C}\arrow[swap]{r}{\delta}& \mathfrak{C}\otimes_\infty\mathfrak{D}
\end{tikzcd}
\end{equation}

We define the canonical inclusion morphisms into $\mathfrak{R}\otimes_{\infty s} \mathfrak{S}:=(\mathfrak{C}\otimes_\infty\mathfrak{D})[\theta^1,\ldots, \theta^q,\xi^1,\ldots,\xi^n]^{\pm}
$ as follows: 
\begin{equation*}
  \alpha:  \mathfrak{R} \longrightarrow  \mathfrak{R}\otimes_\infty \mathfrak{S} \, , \;\;\;   \left(f+\displaystyle{\sum_{I}f_I\theta^I} \right) \longmapsto  \delta(f) +\displaystyle{\sum_{I}\delta(f_I)\theta^I} \, ;
\end{equation*}
\begin{equation*}
\beta:   \mathfrak{S} \longrightarrow  \mathfrak{R}\otimes_\infty \mathfrak{S}  \, , \;\;\;   \left( f+\displaystyle{\sum_{I}f_I\xi^I} \right) \longmapsto  \gamma(f) +\displaystyle{\sum_{I}\gamma(f_I)\xi^I} \, ;
\end{equation*}

where $\delta$ and $\gamma$ are defined in (\ref{eq:localmor}). Consequently, the following diagram commutes 
\begin{equation*}
\begin{tikzcd}
\R \arrow[r,hook] \arrow[d,hook] 
& \mathfrak{R}\arrow{d}{\alpha} \\
 \mathfrak{S}\arrow{r}{\beta}& (\mathfrak{C}\otimes_\infty\mathfrak{D})[\theta^1,\ldots, \theta^q,\xi^1,\ldots,\xi^n]^{\pm}
\end{tikzcd}
\end{equation*}
 
 Since $\mathfrak{C}\otimes_{\infty s}\mathfrak{D}$ is a coproduct, we conclude that $(\mathfrak{C}\otimes_\infty\mathfrak{D})[\theta^1,\ldots, \theta^q,\xi^1,\ldots,\xi^n]^{\pm}$ satisfies the universal property of coproducts.\\
\end{proof}
\begin{example}
Consider the task of adding a new odd variable to a given split $\C$-superring $\mathfrak{R}=\mathfrak{C}[\theta^1,\ldots, \theta^q]^{\pm}$. Let $\mathfrak{S}=\R[\theta]^{\pm}$, where $\R$ is regarded as a $\C$-ring. The coproduct $\mathfrak{R}\otimes_\infty \mathfrak{S}$  is then given by:
$$
\mathfrak{R}\otimes_{\infty s} \mathfrak{S}=\mathfrak{C}[\theta^1,\ldots, \theta^q]^{\pm}\otimes_{\infty s}\R[\theta]^\pm=(\mathfrak{C} \otimes_\infty \R)[\theta^1,\ldots, \theta^q,\theta]^\pm \cong \mathfrak{C}[\theta^1,\ldots, \theta^q,\theta]^\pm \, .
$$

This shows that adding a new odd variable to $\mathfrak{C}[\theta^1,\ldots, \theta^q]^{\pm}$ corresponds to taking the coproduct with $\R[\theta]^{\pm}$.\\
\end{example}

Finitely generated $\C$-rings, which are isomorphic to quotients of the form $\C(\R^n)/I$ (where $I$ is an ideal), are thus characterized in terms of the free $\C$-ring on $n$ generators, $\C(\R^n)$. We aim to establish an analogous concept for the super setting, generalizing the notion of finite generation from the non-graded context. To this end, we introduce the following definition.

\begin{definition}
A $\C$-superring $\mathfrak{R}$ is said to be finitely generated if there exists a surjective morphism $\C(\R^{m|n}):=\C(\R^m)[\theta^1,\ldots, \theta^n]^{\pm}\longrightarrow \mathfrak{R}$ for some non-negative integers $m$ and $n$.    
\end{definition}
Thus, any finitely generated $\C$-superring $\mathfrak{R}$ is isomorphic to a quotient $\C(\R^{m|n})/I$, where $I$ is a $\Z_2$-graded ideal of $\C(\R^{m|n})$. The $\Z_2$-graded structure of this quotient is given by: 
$$\frac{\C(\R^{m|n})}{I}=\frac{\C(\R^{m|n})\ev}{I\ev}\oplus \frac{\C(\R^{m|n})\od}{I\od} \, .$$

\begin{remark}
While this work primarily focuses on finitely generated $\C$-superrings, it is important to emphasize that they form a proper full subcategory of the category ${\bf \C Superrings}$. It is expected that non-finitely generated $\C$-superrings could describe the structure sheaves of infinite-dimensional supermanifolds.
\end{remark}

\begin{example}
A finitely generated $\C$-superring $\mathfrak R$ necessarily possesses a finitely generated superreduced ring  $\overline{\mathfrak R}$. Furthermore, it is plausible that for any non-finitely generated $\C$-rings $\mathfrak{C}$ and any natural number $n$, its \emph{split extension} $\mathfrak{C}[\theta^1,\ldots, \theta^n]$ remains non-finitely generated as a $\C$-superring.    
\end{example}

The following example generalizes the concept of free $\C$-rings determined by a set (see \cite[Proposition 3.3]{CM3}).

\begin{example}
Consider arbitrary sets $U$ and $V$. These sets generate a free $\C$-superring, denoted by $\C \left(\R^{U|V} \right)$. It is defined as the colimit:
$$
\C \left( \R^{U|V} \right) = \varinjlim_{\substack{U_n \subseteq U \text{ finite} \\ V_m \subseteq V \text{ finite}}} \C \left( \R^{U_n|V_m} \right) \, ,
$$

where the colimit is taken over finite subsets $U_n\subseteq U$ and $V_m\subseteq V$. The superreduced ring of $\C(\R^{U|V})$ is $\C(\R^U)$.\\
\end{example}

\subsection{Local $\C$-superrings}

In the existing literature on $\C$-rings, two distinct but closely related definitions of a local $\C$-ring are found. The first states that a $\C$-ring $\mathfrak{C}$ is local if it possesses a unique maximal ideal $\mathfrak{m}$. This notion appears in the works of Moerdijk and Reyes \cite{MR} and, more recently, in those of Berni and Mariano \cite{CM1}, which primarily focus on algebraic and model-theoretic aspects of $\C$-rings.

On the other hand, in more geometrically oriented works, such as those by Dubuc \cite{D2} and Joyce \cite{J}, the definition is refined. They additionally require that the residue field, $\mathfrak C/ \mathfrak m$, be isomorphic to $\R$ as a $\C$-ring. Before Joyce’s work, $\C$-rings satisfying this additional condition were referred to as Archimedean $\C$-rings (see \cite{D1}), but he simply called them local $\C$-rings. See \cite[Example 2.12]{J} for an example of a $\C$-ring that has only one maximal ideal, but whose residue field is not isomorphic to $\R$.

Although this work is primarily concerned with the algebraic aspects of $\C$-superrings, our main goal is to establish the foundations of algebraic supergeometry over $\C$-superrings. This theory is intended to extend the existing notion of smooth supermanifolds and allow for the ``superization'' of other classes of differentiable spaces already present in the literature.
Consequently, we adopt Joyce’s terminology and refer to local $\C$-superrings, even though, strictly speaking, they could be called Archimedean $\C$-superrings. Nonetheless, in the next section, we will analyze the differences and similarities that arise when defining a spectrum functor based on these two notions of local ring.

\begin{definition}\label{def:localSring}
A $\C$-superring $\mathfrak R$ is called local if its superreduced ring $\overline{\mathfrak R}$ is a local $\C$-ring in the sense of Definition \ref{localdef}. That is, $\overline{\mathfrak R}$ must possess a unique maximal ideal $\overline{\mathfrak m}$, and its residue field, $\overline{\mathfrak R}/\overline{\mathfrak m}$, must be isomorphic to $\R$.
\end{definition}

\begin{remark}
While in the theory of superrings, a superring is defined as local if it possesses a unique maximal ideal \cite[Definition 4.1.21]{Westra}, it can be proven that a superring $\mathfrak R$ is local with maximal ideal $\mathfrak m$ if and only if its superreduced ring $\overline{\mathfrak R}$ is a local ring with maximal ideal $\overline{\mathfrak m}$ \cite[Proposition 4.1.23]{Westra}.
\end{remark}

Weil superalgebras provide a powerful algebraic framework for studying infinitesimal neighborhoods within supermanifolds \cite{Alldridge}. They capture infinitesimal  deformations in both even and odd directions, effectively generalizing concepts like tangent spaces and jet spaces to the superspace setting. By incorporating both commuting and anticommuting elements, Weil superalgebras are thus fundamental to the study of supergeometry.

A \emph{Weil $\R$-superalgebra} \cite[Definition 3.26]{Alldridge} is a finite-dimensional $\R$-superalgebra $\mathfrak{R}$ possessing a nilpotent superideal $\mathfrak{m}$ such that $\mathfrak{R} \cong \R\oplus\mathfrak{m}$. Every Weil $\R$-superalgebra is a local superring with $\mathfrak{m}$ as its unique maximal superideal.  Furthermore, $\mathfrak{m}$ consists precisely of all nilpotent elements of $\mathfrak{R}$. Since the unique maximal superideal $\mathfrak m$ is of the form $\mathfrak{m}\ev\oplus \mathfrak{R}\od$, with $\mathfrak{m}\ev$ maximal in $\mathfrak{R}\ev$, we have the decomposition $\mathfrak{R}\cong \R\oplus\mathfrak{m}\ev\oplus\mathfrak{R}\od$. Consequently, the even part $\mathfrak{R}\ev \cong \R\oplus\mathfrak{m}\ev$ itself constitute a Weil algebra (in the non-graded sense).\\

\begin{example}\label{ex:formalseries}
Let $\R[\![X_1,\ldots,X_p]\!]$ denote the ring of formal power series in $p$ 
variables. By Borel's theorem, the Taylor expansion at the origin
$$
j_0^{\infty}:\C(\R^p)\longrightarrow \R[\![X_1,\ldots,X_p]\!],
\qquad
j_0^\infty(g)=\sum_{\alpha\in\mathbb N^p}\frac{D^\alpha g(0)}{\alpha!}\,X^\alpha,
$$
is a surjective morphism of $\R$-algebras with kernel the ideal of functions 
that are flat at the origin. Consequently $\R[\![X]\!]\cong \C(\R^p)/\ker j_0^\infty$ 
inherits, by Remark \ref{rem:limits}, the structure of a $\C$-ring. Writing 
$g_i=g_i(\mathbf 0)+\tilde g_i$ with $\tilde g_i\in\mathfrak M_0:=(X_1,\ldots,X_p)$, 
its $m$-ary operations are given by Taylor expansion at the constant terms 
(see \cite[Theorem 1]{Ish} and \cite[1.3 Borel's Theorem]{MR}):
\begin{equation}\label{eq:seriesphi}
\phi_h(g_1,\ldots,g_m)=\sum_{\gamma\in\mathbb N^m}
\frac{\partial^{\gamma}h\bigl(g_1(\mathbf 0),\ldots,g_m(\mathbf 0)\bigr)}{\gamma!}\;
\tilde g_1^{\,\gamma_1}\cdots\tilde g_m^{\,\gamma_m},
\end{equation}
Applying Proposition \ref{prop:superstruct} to the $\C$-ring $\R[\![X]\!]$ and to 
the free module $\xi$ with basis $\{\theta^1,\ldots,\theta^q\}$, we obtain the split 
$\C$-superring of \emph{formal power series in $p$ even and $q$ odd variables}
$$
\R[\![X^{1},\ldots,X^{p}\mid\theta^1,\ldots,\theta^q]\!]
:=\R[\![X_1,\ldots,X_p]\!][\theta^1,\ldots,\theta^q]^{\pm}
=\bigwedge\nolimits_{\R[\![X]\!]}\xi ,
$$
whose $k$-ary operations $\Phi_h$ are those of \eqref{eq:taylorPhi1}, computed with 
the coefficientwise operations \eqref{eq:seriesphi}.  Explicitly, writing an even 
element as $F_j=a_j+\mu_j$ with $a_j\in\R$ its constant term and 
$\mu_j=\tilde f_j+N_j$ its part in the maximal superideal 
$\mathfrak M:=(X_1,\ldots,X_p,\theta^1,\ldots,\theta^q)$, one obtains
\begin{equation}\label{eq:seriessuper}
\Phi_h(F_1,\ldots,F_k)=\sum_{\alpha\in\mathbb N^k}
\frac{\partial^{\alpha}h(a_1,\ldots,a_k)}{\alpha!}\;
\mu_1^{\alpha_1}\cdots\mu_k^{\alpha_k}.
\end{equation}
Since $\mu^{\alpha}\in\mathfrak M^{|\alpha|}$ and 
$\bigcap_{n\geq0}\mathfrak M^{n}=0$, the sum \eqref{eq:seriessuper} is again 
    $\mathfrak M$-adically convergent. The superreduced ring is $\R[\![X]\!]$, which is 
a local $\C$-ring with maximal ideal $\mathfrak M_0$ and residue field $\R$; hence, 
by Definition \ref{def:localSring}, $\R[\![X\mid\theta]\!]$ is a local 
$\C$-superring. It is not, however, a Weil $\R$-superalgebra, being 
infinite-dimensional over $\R$.
Weil $\R$-superalgebras are precisely the finite-dimensional quotients of this 
construction. Let $\mathfrak I$ be a superideal with 
$\mathfrak M^{\,r+1}\subseteq\mathfrak I\subseteq\mathfrak M$ for some $r\geq0$, and set
$$
\mathfrak W:=\R[\![X_1,\ldots,X_p\mid\theta^1,\ldots,\theta^q]\!]\big/\mathfrak I .
$$
By Example \ref{quotsr}, $\mathfrak W$ is a $\C$-superring; it is local with maximal 
superideal $\mathfrak m:=\mathfrak M/\mathfrak I$, its residue ring is $\R$, and 
$\mathfrak m^{\,r+1}=0$, so that $\mathfrak W\cong\R\oplus\mathfrak m$ is 
finite-dimensional over $\R$. Its $\C$-ring operations are those induced by 
\eqref{eq:seriessuper} on the quotient: denoting by 
$\varepsilon:\mathfrak W\rightarrow\R$ the augmentation and writing 
$w=\varepsilon(w)+\tilde w$ with $\tilde w\in\mathfrak m$, we obtain, for every 
$h\in\C(\R^k)$ and $w_1,\ldots,w_k\in\mathfrak W\ev$,
\begin{equation}\label{eq:weilphi}
\Phi_h(w_1,\ldots,w_k)=\sum_{|\alpha|\leq r}
\frac{\partial^{\alpha}h\bigl(\varepsilon(w_1),\ldots,\varepsilon(w_k)\bigr)}{\alpha!}\;
\tilde w_1^{\,\alpha_1}\cdots\tilde w_k^{\,\alpha_k},
\end{equation}
now a \emph{finite} sum, since every term with $|\alpha|>r$ lies in 
$\mathfrak m^{\,r+1}=0$. 
We stress that the higher-order terms of \eqref{eq:weilphi} cannot be discarded. 
Truncating the expansion at first order yields a family of operations compatible 
with the multiplication of $\mathfrak W$ if and only if 
$\left(\mathfrak m\right)\ev^{2}=0$, which already fails for the purely even Weil 
algebra $\mathfrak W=\R[x]/(x^{3})$: taking $h(u,v)=uv$ and $w_1=w_2=x$, the 
truncated formula returns $0$, whereas $w_1w_2=x^{2}\neq0$. This is the same 
phenomenon observed in Example \ref{Rpq} for $q\geq4$, and it is the reason why 
\eqref{eq:taylorPhi1}, rather than its linearization, is the correct definition of 
the $k$-ary operations throughout this paper.
\end{example}

\begin{example}
Consider the $\C$-superring $\mathfrak R=\mathbb R[\theta^1, \theta^2, \theta^3]^\pm$, and let $\mathfrak J_{\mathfrak R}$ be its canonical superideal. Since $(\mathfrak R/\mathfrak J^2_{\mathfrak R})/(\mathfrak J_{\mathfrak R}/\mathfrak J^2_{\mathfrak R})\cong \mathfrak R/\mathfrak J_{\mathfrak R}  \cong \mathbb R$, the ideal $\mathfrak m=\mathfrak J_{\mathfrak R}/\mathfrak J^2_{\mathfrak R}$ of $\mathfrak R/\mathfrak J^2_{\mathfrak R}$ is maximal.

Thus, the quotient ring $\mathfrak R/\mathfrak J^2_{\mathfrak R}$  is a Weil $\R$-superalgebra with maximal ideal $\mathfrak m=\mathfrak J_{\mathfrak R}/\mathfrak J^2_{\mathfrak R}$. Both its even part and its superreduced ring are isomorphic to $\R$. More generally, for any $\mathfrak S=\mathbb R[\theta^1, \cdots, \theta^n]^\pm$, the $\C$-superring of the form $\mathfrak S/\mathfrak J^2_{\mathfrak S}$ is a Weil $\R$-superalgebra.
\end{example}

\begin{example}
Let $\mathfrak{C}$ be a Weil algebra and consider $\mathfrak{R}=\mathfrak{C}[\theta^ 1,\ldots,\theta^n]^\pm$ a split $\C$-superring. Since $\mathfrak{C}=\R\oplus\mathfrak{m}$ with $\mathfrak{m}$ a maximal nilpotent ideal of $\mathfrak{C}$, we have that $I\ev=\mathfrak{m}\oplus \mathfrak{R}\od^2$ is a maximal nilpotent ideal of $\mathfrak{R}\ev$ because $\frac{\mathfrak{R}\ev}{\mathfrak{m}\ev}\cong \frac{\overline{\mathfrak{R}}\oplus \mathfrak{R}\od^2}{\mathfrak{m}\oplus \mathfrak{R}\od^2}\cong \frac{\mathfrak{C}\oplus\mathfrak{R}\od^2}{m\oplus \mathfrak{R}\od^2}\cong \frac{\mathfrak{C}}{\mathfrak{m}}\cong \R$ which is a field. Therefore, $\mathfrak{m}_{\mathfrak{R}}=\mathfrak{m}\ev\oplus \mathfrak{R}\od $ is a maximal graded  ideal of $\mathfrak{R}$ . Thus, since $\mathfrak{R}=\R\oplus \mathfrak{m}\oplus \mathfrak{R}\od^2\oplus \mathfrak{R}\od=\R\oplus \mathfrak{m}_{\mathfrak{R}}$, we have that  $\mathfrak{R}$ is a Weil $\R$-superalgebra.    
\end{example}

\begin{lemma}
Consider $\mathfrak{W}$ a Weil $\R$-superalgebra and $\mathfrak{R}$ a $\C$-superring. If $\varphi:\mathfrak{W}\rightarrow \mathfrak{R}$ is a superalgebra morphism, then $\varphi$ is a morphism of $\C$-superrings. 
\end{lemma}
\begin{proof}

According to corollary 3.9 of \cite{MR} this result holds for the ungraded case, that is, if $\alpha: \mathfrak{C}\rightarrow W$. is an $\R$-algebra morphism between a Weil algebra and a $\C$-ring, then $\alpha$ is a morphism of $\C$-rings. For the graded case we have that $\varphi:\mathfrak{W}\rightarrow \mathfrak{R}$ is a morphism of $\R$ superalgebras whose restriction to the event part $\varphi\ev:\mathfrak{W}\ev\rightarrow \mathfrak{R}\ev$ is a morphism of $\R$-algebras. Therefore, $\varphi$ is a morphism of $\C$-superrings since it satisfies all the conditions of Definition \ref{def:smorph} 
\end{proof}

\subsection{Localization} 

To construct geometric objects from $\C$-superrings, localization serves as a fundamental tool, thus requiring the introduction of local $\C$-superrings. Similar to $\C$-rings (and unlike commutative rings), there exist prime superideals at which the localization of a $\C$-superring is not necessarily local (see Example \ref{ej-loc-nonlocal}). This fact renders the traditional prime spectrum unsuitable for geometric applications.\\

This section begins with a review of basic properties of prime ideals, followed by a discussion on localization. We will demonstrate the existence of localization for $\C$-superrings and explore its connection to the localization of general superrings. Moreover, we will extend to the $\C$-superring context an important concept from the theory of $\C$-rings widely studied in  \cite{CM1} , the notion of the $\C$-radical of a superideal. This extension is pivotal. It introduces radical prime superideals, which provide the appropriate notion of prime for geometric work within $\C$-superrings (just as in the $\C$-ring case). This is because the localization of a $\C$-superring at a $\C$-radical prime superideal is always a local $\C$-superring.

\begin{definition}\label{def:suplocal}
 Let $\mathfrak{R}$ be a $\C$-superring and $S$ a multiplicative closed set in $\mathfrak {R}\ev$. The localization of $\mathfrak{R}$ at $S$ is a $\C$-superring, denoted $\mathfrak{R}\{S^{-1}\}$, together with a $\C$-superring morphism $\mathcal{L}:\mathfrak{R}\rightarrow \mathfrak{R}\{S^{-1}\}$, such that the following two conditions are satisfied:
\begin{itemize}
\item Every element in $\mathcal{L}(S)$ is invertible in $\mathfrak{R}\{S^{-1}\}$.

\item Universal Property: For any $\C$-superring $\mathfrak S$ and any $\C$-superring morphism $\varphi:\mathfrak{R}\rightarrow \mathfrak{S}$ such that every element in $\varphi(S)$ is invertible in $\mathfrak S$, there exists a unique $\C$-superring morphism $\psi: \mathfrak{R}\{S^{-1}\}\rightarrow \mathfrak{S}$ that makes the following diagram commute:
\end{itemize}

\begin{center}
\begin{tikzcd}
\mathfrak R \arrow{d}[swap]{\varphi} \arrow{r}{\mathcal L} & \mathfrak R \{ S^{-1} \} \arrow{dl}{\psi} \\
 \mathfrak S
\end{tikzcd}
\end{center}

\end{definition}

\begin{proposition}
Let $\mathfrak{R}$ be a $\C$-superring and $S$ a multiplicative closed set in $\mathfrak{R}\ev$. The localization of $\mathfrak{R}$ at $S$ is given by the $\C$-superring $\mathfrak{R}\{S^{-1}\} := \mathfrak{R}\ev\{S^{-1}\} \oplus \mathfrak{R}\od \{S^{-1}\}$, where $\mathfrak{R}\ev\{S^{-1}\}$ is the localization of the $\C$-ring $\mathfrak{R}\ev$  and $\mathfrak{R}\od \{S^{-1}\}$ is the localization of the $\mathfrak{R}\ev$-module $\mathfrak{R}\od $, both with respect to $S$.    
\end{proposition}

\begin{proof}
 We first show that $ \mathfrak{R}\ev\{S^{-1}\} \oplus \mathfrak{R}\od \{S^{-1}\}$ has the structure of a $\C$-superring. By the properties of $\C$-rings, $\mathfrak{R}\ev\{S^{-1}\}$ is a $\C$-ring and $\mathfrak{R}\od \{S^{-1}\}$ 
is a $\mathfrak{R}\ev\{S^{-1}\}$-module. Furthermore, $\mathfrak{R}\ev \{S^{-1}\} \mathfrak{R}\ev \{S^{-1}\} \subseteq \mathfrak{R}\ev \{S^{-1}\}$, $\mathfrak{R}\ev \{S^{-1}\} \mathfrak{R}\od \{S^{-1}\} \subseteq \mathfrak{R}\od \{S^{-1}\}$, and $\mathfrak{R}\od \{S^{-1}\} \mathfrak{R}\od \{S^{-1}\} \subseteq \mathfrak{R}\ev \{S^{-1}\}$, which completes the superring structure.\\

Now, we define the localization morphism $\mathcal{L}: \mathfrak{R} \rightarrow \mathfrak{R}\{S^{-1}\}$. This is a $\C$-superrings morphism whose even and odd components are  the localization morphisms   $\mathcal{L}\ev: \mathfrak{R}\ev \rightarrow \mathfrak{R}\ev\{S^{-1}\}$ and $\mathcal{L}\od:\mathfrak{R}\od  \rightarrow \mathfrak{R}\od \{S^{-1}\}$ respectively. Moreover, it satisfies the universal property of localization since it follows from the universality of localization for the ungraded case.
\end{proof}

\begin{remark}\label{rmk:localnil}
 It is not hard to verify from Definition \ref{def:suplocal} that if $S$ contains nilpotent elements then $\mathfrak{R}\{S^{-1}\}$ is the zero superring.   
\end{remark}

\begin{example}
Let $\mathfrak{R}$ be a $\C$-superring and $f \in \overline{\mathfrak R}$. Consider the multiplicative closed set $S=\{f^k;\; k \in \mathbb N \cup 0 \}$. The localization of $\mathfrak{R}$ at $S$, denoted $\mathfrak{R}\{f^{-1}\}$, is generally larger than the corresponding localization of $\mathfrak{R}$ when $\mathfrak{R}$ is considered solely as an $\mathbb R$-superalgebra.

This difference arises because the $\C$-superring localization $\mathfrak{R}\{f^{-1}\}$ is defined as $\mathfrak{R}\{f^{-1}\}:=(\mathfrak{R}\ev)\{f^{-1}\} \oplus (\mathfrak{R}\od)\{f^{-1}\}$. Here, $(\mathfrak{R}\ev)\{f^{-1}\}$ is the localization of the $\C$-ring $\mathfrak{R}\ev $ (as detailed in Remark \ref{rem:localCinfty}), and $(\mathfrak{R}\ev)\{f^{-1}\}$ is the localization of the $\mathfrak{R}\ev$-module $\mathfrak{R}\od$ (with $\mathfrak{R}\ev$ acting as an $\mathbb R$-algebra). The specific nature of $\C$-localization, particularly for $\mathfrak{R}\ev$, can lead to a larger structure compared to purely algebraic localization.
\end{example}

\begin{lemma}
Let $\mathfrak{R}$ and $\mathfrak{S}$ be $\C$-superrings. For any homogeneous $f \in \mathfrak R\ev$, let $S=\{f^k;\; k \in \mathbb N \cup 0 \}$ be the multiplicative set generated by $f$, and denote $\mathfrak R\{f^{-1}\}:=\mathfrak R \{ S^{-1} \}$ as the localization of $\mathfrak R$ at $S$. Given a $\C$-superring morphism $\varphi:\mathfrak{R}\rightarrow \mathfrak{S}$, there exists a unique $\C$-superring morphism $\varphi_f: \mathfrak{R}\{f^{-1}\}\rightarrow \mathfrak{S} \{ \varphi(f)^{-1}\}$ such that the following diagram commutes:
       \[ 
\begin{tikzcd}
\mathfrak{R} \arrow{r}{\mathcal L_{\mathfrak R, f}} \arrow[swap]{d}{\varphi} & \mathfrak{R}\{f^{-1}\}\arrow{d}{\varphi_f} \\
\mathfrak{S} \arrow{r}{\mathcal L_{\mathfrak S, \varphi(f)} }& \mathfrak{S} \{\varphi(f)^{-1} \}
\end{tikzcd}
\]
\end{lemma}

\begin{proof}
We consider two cases.
\begin{itemize}
\item Case 1: $\varphi(f)$ is nilpotent. In this case, by Remark \ref{rmk:localnil}, the localization $\mathfrak S \{\varphi(f) ^{-1}\}$ is the zero superring. Consequently, the unique morphism $\varphi_f$ must be the zero morphism.

\item Case 2: $\varphi(f)$ is not nilpotent. Then the canonical localization map $\mathcal{L}_{\mathfrak S, \varphi(f)}$ ensures that $\mathcal{L}_{\mathfrak S, \varphi(f)}(\varphi(f))$ is an invertible element on $\mathfrak{S} \{ \varphi(f)^{-1} \}$. By the universal property of localization for $\mathfrak R$, applied to the morphism $\mathfrak R \rightarrow \mathfrak S \{\varphi(f)^{-1} \}$, there exists a unique morphism $\varphi_f: \mathfrak{R} \{f^{-1}\} \rightarrow \mathfrak{S} \{\varphi(f)^{-1} \}$ that makes the diagram commute.
\end{itemize}
\end{proof}

\begin{remark}
Since the canonical ideal $\mathfrak{J}_{\mathfrak{R}}$ is generated by the odd part of a $\C$-superring $\mathfrak{R}$, and such elements are nilpotent, it follows that if a multiplicative set $S$ contains elements from $\mathfrak{J}_{\mathfrak{R}}$, the localization $\mathfrak R \{ S^{-1}\}$  must be the zero superring (see Remark \ref{rmk:localnil}).
\end{remark}

\begin{remark}
Similar to the situation in commutative algebra, it is not true in general that the localization of a $\C$-ring with respect to an arbitrary multiplicatively closed set is a local $\C$-ring. This point is clearly illustrated in \cite[Remark 3.31]{CM1}, where the authors construct a continuum of maximal ideals for the $\C$-ring $\mathfrak{C}=\C(\R)\{f^{-1}\}$, where $f$ is a function that is not identically zero. Thus, by constructing a split $\C$-superring whose superreduced part is $\mathfrak C$, we obtain an example of a non-local $\C$-superring.
\end{remark}

\begin{lemma}
Let $\mathfrak R$ be a $\C$-superring and $S \subseteq \mathfrak R\ev$ be a multiplicative closed set. Let $\overline S$ be the direct image of $S$ under the canonical projection $q_{\mathfrak R} :\mathfrak R \rightarrow \overline{\mathfrak R}$. Then the following isomorphism holds: $\overline{\mathfrak{R}\{S^{-1}\}}\cong \overline{\mathfrak{R}}\{\overline{S} ^{-1}\}$
\end{lemma}
\begin{proof} If $S$ contains nilpotent elements then both $\overline{\mathfrak{R}\{S^{-1}\}}$ and $\overline{\mathfrak{R}}\{\overline S^{-1}\}$ are zero. Now, if $S$ contains no nilpotent elements, it is enough to show that $\overline{\mathfrak{R}\{S^{-1}\}}$ satisfies the universal property of localization in the category of $\C$-rings. Let $\mathcal{L}: \mathfrak{R} \rightarrow \mathfrak{R}\{S^{-1} \} $ be the localization morphism. If $\overline{a}\in \overline{S}$, then $a\in S$ and $\mathcal{L}(a)$ is invertible in $\mathfrak{R}\{ S^{-1} \}$. Therefore it is invertible in $\overline{\mathfrak{R}\{S^{-1} \}}$. Thus, if $\varphi:\overline{\mathfrak{R}}\rightarrow \mathfrak{C}$ is a $\C$-ring morphism such that any element of $\overline{S}$ is invertible in $\mathfrak{C}$, then by adding zero odd part to $\mathfrak{C}$, we have a morphism of $\C$-superrings from $\mathfrak{R}$ to $\mathfrak{C}$ such that any element of $S$ is invertible in $\mathfrak{C}$. In consequence, by the universal property of localization applied to $\mathfrak{R}$ and $S$, we have a morphism $\psi:\overline{\mathfrak{R}\{S^{-1}\}}\rightarrow \mathfrak{C}$ making the diagram commutative.

\begin{center}
\begin{tikzcd}
\mathfrak R \arrow{d}[swap]{\varphi} \arrow{r}{\mathcal L} & \mathfrak R \{ S^{-1} \} \arrow{dl}{\psi} \\
 \mathfrak C
\end{tikzcd}
\end{center}

Thus, taking reduced part to all elements of the diagram led us to the commutativity of the following diagram

\begin{center}
\begin{tikzcd}
\overline{\mathfrak R} \arrow{d}[swap]{\varphi} \arrow{r}{\overline{\mathcal L}} & \overline{\mathfrak R \{ S^{-1} \}} \arrow{dl}{\psi} \\
 \mathfrak C
\end{tikzcd}
\end{center}

In conclusion, we have $\overline{\mathfrak{R}\{\mathcal{D}\}^{-1}}\simeq\overline{\mathfrak{R}}\{\overline{\mathcal{D}}\}^{-1}$ by the uniqueness of the universal property of localization.
\end{proof}

The concept of $\C$-radical of a superideal extends the definition provided for $\C$-rings \cite[Definition 3.47]{CM1}, see also \cite{MRI, MRII}.

\begin{definition}\label{defrad} Let $\mathfrak{R}=\mathfrak R\ev \oplus \mathfrak R\od$ be a $\C$-superring, and let $\mathfrak{I}=\mathfrak I\ev \oplus \mathfrak I\od$ be a superideal of $\mathfrak R$. The $\C$-radical of its $\mathfrak I$, denoted as $\sqrt[\infty]{\mathfrak{I}}$, is the superideal defined as:

$$\sqrt[\infty]{\mathfrak{I}}=\sqrt[\infty]{\mathfrak{I}\ev}\oplus\mathfrak{R}\od \, .$$

Here, $\sqrt[\infty]{\mathfrak{I}\ev}$ represents the $\C$-radical of the even part of $\mathfrak I$, and is given by:
$$
\sqrt[\infty]{\mathfrak{I}\ev}:= \left\{ a \in \mathfrak{R}\ev \mid \left(\dfrac{\mathfrak{R}\ev}{\mathfrak{I}\ev}\right)\{(a+\mathfrak{I}\ev)^{-1}\} \cong 0 \right\}.
$$
\end{definition}

\begin{remark}
\label{cor:samerad}
In commutative algebra, the radical of an ideal $I$ of a commutative ring $R$, denoted $\sqrt I$, is equivalently defined as the set $
 \left \{ r \in R \mid \left( R/I\right) [ (r+I)^{-1} ] \cong 0 \right \}$. Here, the square brackets denote algebraic localization (see, e.g., \cite[page 49]{CM1}). This definition, which employs localization, is better understood by recalling that $\sqrt I$ is the preimage of the ideal of nilpotent elements of the quotient ring $R/I$ via the canonical projection map $\pi: R \to R/I$, that is, $\sqrt I=\pi^{-1}\left(\sqrt{(0)} \right)$. Drawing inspiration from this definition, the $\C$-radical of an ideal $I$ of a $\C$-ring $\mathfrak C$ is defined analogously, but employing $\C$-localization instead of algebraic localization. Given that the algebraic localization $\mathfrak C[S^{-1}]$ of a $\C$-ring $\mathfrak C$ is generally a subring of its $\C$-localization $\mathfrak C\{S^{-1}\}$, it has been shown that $\sqrt I \subseteq \sqrt[\infty]{I}$ \cite[Proposition 1]{BK}. The notion of $\C$-radical for $\C$-rings, as well as other types of radicals of an ideal that are not generally equivalent, are subjects of ongoing study within the field \cite{BK} (see also \cite{CM1}).
\end{remark}

\begin{definition}
A superideal $\mathfrak{I}$ of a $\C$-superring is called \emph{$C^\infty$-radical} if $\sqrt[\infty]{\mathfrak{I}}=\mathfrak{I}$. 
\end{definition}

\begin{remark}
Note that $\mathfrak{I}=\sqrt[\infty]{\mathfrak{I}}$ if and only if  $\sqrt[\infty]{\mathfrak{I}\ev}=\mathfrak{I}\ev$ and $\mathfrak{I}\od=\mathfrak{R}\od$
\end{remark}

\begin{remark}\label{rem:radical-prime}
Recall that a prime superideal $\mathfrak{p}$ of a superring $\mathfrak R$ has the form $\mathfrak{p}=\mathfrak{p}\ev\oplus\mathfrak{R}\od$. Therefore, a prime superideal  $\mathfrak{p}$ of a $\C$-superring $\mathfrak R$  is $\C$-radical if and only if its even part $\mathfrak{p}\ev$ is a $\C$-radical ideal of $\mathfrak{R}\ev$. Following the notation of \cite{CM1}, we write $Spec^{\infty}(\mathfrak{C})$ for the set of all $\C$-radical prime ideals of a $\C$-ring  $\mathfrak{C}$. This set is called the $\C$-spectrum of the $\C$-ring and plays an important role in defining locally $\C$-ringed spaces. In the next section, we will generalize the notion of $\C$-spectrum to the case of $\C$-superrings.
\end{remark}

From \cite[Lemma 2.3]{MRI} we know that the $\C$-radical of an ideal is itself an ideal. In particular, maximal ideals of a $\C$-ring are $\C$-radical. Therefore, a direct consequence of Remark \ref{cor:samerad} is that the $\C$ radical of a graded ideal on a $\C$-superring is a graded ideal and, in particular, maximal graded ideals of a $\C$-superring are $\C$-radical.\\

As in the case of $\C$-rings, we observe that $\sqrt{\mathfrak{I}}\subseteq\sqrt[\infty]{\mathfrak{I}}$, where $\sqrt{\mathfrak{I}}$ denotes the usual radical  of the superideal $\mathfrak{I}$. However, this inclusion can be a strict one. The following example is a simple extension of an example presented in \cite[Section 1]{BK}

\begin{example}
 Consider  $\C(\R)$ as a $\C$-superring with a zero odd part, and let $\mathfrak{I}$ be the ideal generated by $\text{exp}({-1/x^2})$. We claim that $x\in \sqrt[\infty]{\mathfrak{I}}$ but $x\notin \sqrt{\mathfrak{I}}$. it is clear that $x^n\notin \mathfrak{I}$ for all $n \in \mathbb N$. Let $b=\text{exp}(-1/x^2)\in \mathfrak{I}$. With the localization morphism $\mathcal{L}_x:\C(\R) \to \C(\R)\{x^{-1}\}$,  $\mathcal L_x(b)$ is invertible in $\C(\R)\{x^-1\} \cong \C(\R \setminus {0})$. Furthermore, its inverse is given by $\mathcal L_x(\text{exp}(1/x^2))$. This is equivalent to saying that $x$ is an element of $\sqrt[\infty]{\mathfrak{I}}$ (see \cite[Proposition 3.48]{CM1}).
\end{example}

At this point, one might question the necessity of introducing the $\C$-radical of a superideal. The answer lies in our primary intended use of $\C$-superrings: the study of locally ringed superspaces. This application requires prime ideals to exhibit desirable behavior under localization. Specifically, we expect that the localization at a prime ideal results in a local $\C$-superring. However, this property may not be guaranteed for general prime superideals, as the following examples show.

\begin{example}\label{ej-loc-nonlocal}
    Let $\mathfrak{m}_0 \subset \C(\R)$ be the ideal of flat functions  at $0$. According to example 1.2 in \cite{MRII}, $\mathfrak{m}_0$ is a non-$\C$-radical ideal, and the localization of $\C(\R)$ at $\mathfrak{m}_0$ is not a local ring (meaning it has more than one maximal ideal). Therefore, it is not hard to show that for the $\C$-superring $\mathfrak{R}=\C(\R)[\theta^1,\theta^2]^{\pm}$, the superideal $\mathfrak{p}=\mathfrak{m}_0\oplus\mathfrak{R}\od$ is non-$\C$-radical, and the localization of $\mathfrak{R}$ at $S=\mathfrak R\ev-\mathfrak m_0$ cannot be a local ring. 
    
\end{example}

Let $\mathfrak{R}$ be a $\C$-superring. We denote the localization of $\mathfrak{R}$ at the multiplicatively closed  set  $S= \mathfrak R\ev-\mathfrak p\ev$ as $\mathfrak{R}_{\mathfrak p\ev}$.

\begin{theorem}\label{rad-local}
Let $\mathfrak{p}=\mathfrak{p}\ev\oplus\mathfrak{R}\od$ be a $\C$-radical prime superideal of a $\C$-superring $\mathfrak{R}=\mathfrak{R}\ev\oplus\mathfrak{R}\od$. Then the $\C$-superring $\mathfrak{R}_\mathfrak{p\ev}$ is local.

\end{theorem}
\begin{proof}
Since $\mathfrak{p}$ is $\C$-radical, its even part $\mathfrak{p}\ev$ is a $\C$-radical prime ideal of the $\C$-ring $\mathfrak{R}\ev$, Therefore, according to \cite[Lemma 1.1]{MRII}, the localization $(\mathfrak R_{\mathfrak p\ev})\ev=(\mathfrak{R}\ev)_{\mathfrak p\ev}$ is a local $\C$-ring, with a unique maximal ideal which we denote by $\mathfrak{m}\ev$. Consequently, the ideal $\mathfrak{m}=\mathfrak{m}\ev\oplus(\mathfrak{R}\od)_{\mathfrak p\ev}$ is the unique maximal ideal of $\mathfrak{R}_{\mathfrak p\ev}$.  
\end{proof}

\begin{example}
Let $M$ be a smooth manifold. As established in  \cite[Theorem 1.3]{MRII}, any countably generated prime ideal of $\C(M)$ is $\C$-radical. This result directly generalizes to superideals of $\C$-superrings. In particular, any countably generated prime superideal of the structure sheaf of a smooth supermanifold is $\C$-radical. 
\end{example}

The following result is an adaptation to the $\C$-superring case of the result mentioned in page 84 of \cite{CM1}. It is a consequence of Theorem \ref{rad-local}. We present it for completeness.

\begin{lemma}
Every $\C$-radical prime $\mathfrak{p}$ ideal of a $\C$-superring is the kernel of a morphism of $\C$-superrings. 
\end{lemma}
\begin{proof}
Let $\mathfrak{p}$ be a $\C$-radical prime ideal, of a $\C$-ring $\mathfrak{C}$, the localization $\mathfrak{C}_{\mathfrak{p}}$ has only one maximal ideal $\mathfrak{m}$ in $\mathfrak{C}_{\mathfrak{p}}$. let $\mathcal{L}$ and $q$ denote the localization morphism and the quotient morphism onto $\mathfrak{C}_{\mathfrak{p}}/\mathfrak{m}$ respectively. We claim that $\mathfrak{p}$ is the kernel of the composition $\varphi=q\circ\mathcal{L}$. Indeed, if $a \in \mathfrak{p}$ then $\mathcal{L}(a)$ is not invertible in $\mathfrak{C}_{\mathfrak{p}}$. Hence, it cannot be outside $\mathfrak{m}$, that is, $\varphi(a)=0$. By definition this means that $a\in ker\varphi$. On the other hand, if $\varphi(a)=0$ in the fraction field $\mathfrak{C}_{\mathfrak{p}}/\mathfrak{m}$, then $\mathcal{L}(a)\in \mathfrak{m}$ which means that this is not an invertible element in the localization hence $a\in \mathfrak{p}$.

Now for the graded case, let $\mathfrak{R}=\mathfrak{R}\ev\oplus\mathfrak{R}\od$ be a $\C$-superring and denote by $\mathfrak{m}$ the only maximal ideal of the local $\C$-ring $(\mathfrak{R}_{\mathfrak{p}}){\ev}$. Then, $\mathfrak{p}\ev$ is the kernel of a morphism $\varphi\ev:\mathfrak{R}\ev\rightarrow (\mathfrak{R}_{\mathfrak{p}}){\ev} $
 Therefore, 
$\mathfrak{p}\ev\oplus\mathfrak{R}\od$ is a $\C$-radical prime ideal and is the kernel of the morphism $\varphi$  that results when we add zero odd part to  $\varphi\ev$. That is, we extend $\varphi\ev$ to $\mathfrak{R}$ by taking $\varphi(r\od)=0$ for any $r\od\in \mathfrak{R}\od$, then $\mathfrak{p}$ is clearly the kernel of $\varphi:\mathfrak{R}\rightarrow (\mathfrak{R}_{\mathfrak{p}}){\ev}  $
\end{proof}

The $\C$-radical of an ideal is a useful notion in smooth commutative algebra. It allows us to define the notion of $\C$-nilpotent elements and characterize a $\C$-ring as $\C$-reduced (see, for example, \cite{CM1, BK}).\\
We now explore some properties of $\C$-radicals of superideals in $\C$-superrings, extending some of the results established of $\C$-rings. 

\begin{proposition}\label{prop:sradical}
Let $\mathfrak{I}, \mathfrak{H}$ be superideals of a $\C$-superring $\mathfrak{R}=\mathfrak R\ev \oplus \mathfrak R\od$. The following properties hold: 
    \begin{itemize}
\item [a)] If $\mathfrak{I}\subseteq \mathfrak{H}$ then $\sqrt[\infty]{\mathfrak{I}}\subseteq \sqrt[\infty]{\mathfrak{H}} $.
\item [b)]$\sqrt[\infty]{\mathfrak{I}}$ is the intersection of all $\C$-prime radical superideals that contain $\mathfrak{I}$
\item[c)] $\sqrt[\infty]{\mathfrak{I}\mathfrak{H}}=\sqrt[\infty]{\mathfrak{I}\cap \mathfrak{H}}$
\end{itemize}
\end{proposition}
\begin{proof}
\item [a)] We have that $\mathfrak{I}\ev$ and $\mathfrak{H}\ev$ are ideals of the $\C$-ring $\mathfrak R\ev$. Since 
$\mathfrak{I}\ev\subseteq \mathfrak{H}\ev$, it follows $\sqrt[\infty]{\mathfrak{I}\ev}\subseteq \sqrt[\infty]{\mathfrak{H}\ev}$ by \cite[Theorem 4.29]{CM1}. Therefore, $\sqrt[\infty]{\mathfrak{I}\ev}\oplus \mathfrak{R}\od \subseteq \sqrt[\infty]{\mathfrak{H}\ev}\oplus \mathfrak{R}\od $, which means $\sqrt[\infty]{\mathfrak{I}}\subseteq \sqrt[\infty]{\mathfrak{H}}$.

\item[b)] Let $\mathfrak p$ denote a $\C$-radical prime superideal. Then we have
$$
\bigcap\limits_{\mathfrak{p}\supset\mathfrak{I}} \mathfrak{p}=\bigcap\limits_{\mathfrak{p}\ev\oplus \mathfrak{R}\od\supseteq \mathfrak{I}} \left( \mathfrak{p}\ev\oplus \mathfrak{R}\od  \right)= \left(\bigcap\limits_{\mathfrak{p}\ev\supseteq \mathfrak{I}\ev}\mathfrak{p}\ev \right) \oplus\mathfrak{R}\od=\sqrt[\infty]{\mathfrak{I}\ev}\oplus \mathfrak{R}\od=\sqrt[\infty]{\mathfrak{I}}\, .
$$
Here, we have used the property $\sqrt[\infty]{\mathfrak{I}\ev}=\bigcap\limits_{\mathfrak{p}\ev\supseteq \mathfrak{I}\ev}\mathfrak{p}\ev$ (from \cite[Theorem 4.42 b)]{CM1}), where  $\mathfrak p\ev$ represents a $\C$-radical prime ideal of the $\C$-ring $\mathfrak R\ev$.

\item [c)] Since $\mathfrak{I}\mathfrak{H}\subseteq \mathfrak{I}\cap\mathfrak{H}$,any prime ideal containing $\mathfrak{I}\cap\mathfrak{H}$ also contains $\mathfrak{I}\mathfrak{H}$. Now, we suppose that there exists some prime $\mathfrak{p}$ such that $\mathfrak{p}\supseteq \mathfrak{I}\mathfrak{H} $ but there exists $x\in \mathfrak{I}\cap\mathfrak{H} $ such that $x\notin \mathfrak{p}$. However, $x^2\in \mathfrak{I}\mathfrak{H}\subseteq \mathfrak{p}$ and since $\mathfrak{p}$ is prime, we have $x\in\mathfrak{p}$ which is a contradiction.
Thus, the intersection of all even primes that contain $\mathfrak{I}\mathfrak{H}$ is the same as the intersection of all even primes that contain $\mathfrak{I}\cap\mathfrak{H}$. Thus, the conclusion is a consequence of the item $b)$ above.
\end{proof}

\begin{lemma}\label{lem:inftyrad} Let $\mathfrak{R}=\mathfrak R\ev \oplus \mathfrak R\od$ be a $\C$-superring, and let $\mathfrak{J}_\mathfrak{R}$ be the canonical ideal generated by the odd part $\mathfrak R\od$. Then, the following equalities hold:
\begin{itemize}
\item [a)]  $\sqrt[\infty]{\mathfrak{J}_{\mathfrak{R}}}=\sqrt[\infty]{(0)_{\overline{\mathfrak{R}}}}\oplus \mathfrak{J}_{\mathfrak{R}}$.
\item [b)]  $\sqrt[\infty]{(0)_{\mathfrak{R}\ev}}=\sqrt[\infty]{(0)_{\overline{\mathfrak{R}}}}\oplus \mathfrak{R}\od^2$.
\item [c)] $\sqrt[\infty]{(0)_{\mathfrak{R}}}=\sqrt[\infty]{(0)_{\mathfrak{R}\ev}}\oplus \mathfrak{R}\od=\sqrt[\infty]{(0)_{\overline{\mathfrak{R}}}}\oplus \mathfrak{J}_{\mathfrak{R}}$.
\end{itemize}

To simplify the notation, $\sqrt[\infty]{0_{\overline{\mathfrak{R}}}}$ is used to represent both the $\C$-radical of the zero ideal in $\overline{\mathfrak{R}}$ and its inverse image under the quotient projection $\mathfrak{R}\longrightarrow \overline{\mathfrak{R}}$.
\begin{proof}
\begin{itemize}
\item [a)]
Let $a\in \sqrt[\infty]{\mathfrak{J}_{\mathfrak{R}}}$ by the definition of $\C$-radical we see that the localization $\overline{\mathfrak{R}}\{\overline{a}^{-1}\}$ is the zero $\C$-ring. Therefore, $\overline{a}$ is $\C$-nilpotent in $\overline{\mathfrak{R}}$. Here we have two cases: 
First $a\in \mathfrak{J}_{\mathfrak{R}}$, therefore $\overline{a}=0_{\overline{\mathfrak{R}}}\in \overline{\mathfrak{R}}$, hence $\overline{a}\in \sqrt[\infty]{0_{\overline{\mathfrak{R}}}} $ and clearly $a\in \sqrt[\infty]{0_{\overline{\mathfrak{R}}}}\oplus \mathfrak{J}_{\mathfrak{R}}$. On the other hand, if $a \notin \mathfrak{J}_{\mathfrak{R}}$ we have that $\overline{a}$ is a nonzero $\C$-nilpotent element of $\overline{\mathfrak{R}}$, therefore $\overline{a}\in \sqrt[\infty]{0_{\overline{\mathfrak{R}}}} $ and consequently $a\in \sqrt[\infty]{0_{\overline{\mathfrak{R}}}}\oplus \mathfrak{J}_{\mathfrak{R}}$. Thus, we have $\sqrt[\infty]{\mathfrak{J}_{\mathfrak{R}}}\subseteq \sqrt[\infty]{0_{\overline{\mathfrak{R}}}}\oplus \mathfrak{J}_{\mathfrak{R}}$ The other inclusion is obvious since both $\sqrt[\infty]{0_{\overline{\mathfrak{R}}}}$ and $\mathfrak{J}_{\mathfrak{R}}$ are contained in $\sqrt[\infty]{\mathfrak{J}_{\mathfrak{R}}}$.\\ 

\item [b)] In this case, we are working with the $\C$-radical of an ideal in a $\C$-ring. By definition, $\sqrt[\infty]{(0)_{\mathfrak{R}\ev}}$ is the set of all $\C$-nilpotent elements of $\mathfrak{R}\ev$. It is not hard to verify that all nilpotents of $\mathfrak{R}\ev$ are of two types: those that are even elements originating from  $\mathfrak{R}\od^2$, or those whose class in $\overline{\mathfrak{R}}$ is  $\C$-nilpotent (i.e, elements of $\sqrt[\infty]{(0)_{{\overline{\mathfrak{R}}}}}$). Thus, we have $\sqrt[\infty]{(0)_{\mathfrak{R}\ev}}\subseteq\sqrt[\infty]{(0)_{\overline{\mathfrak{R}}}}\oplus \mathfrak{R}\od^2$. The reverse inclusion is clear.

\item[c)] The first equality follows directly from the definition of $\C$-radical applied to superideal $(0)_\mathfrak{R}$. The second is a direct application of item $b)$.
\end{itemize} 
\end{proof}
\end{lemma}

\begin{corollary}\label{prime_intersection}
For any $\C$-superring $\mathfrak{R}$, the ideal $\sqrt[\infty]{(0)_\mathfrak{R}}$ is equal to the intersection of all $\C$-radical prime superideals of $\mathfrak{R}$.
\end{corollary}
\begin{proof}
Since all radical prime superideals of $\mathfrak R$ are of the form $\mathfrak{p}\ev\oplus\mathfrak{R}\od$, where $\mathfrak{p}\ev$ is a radical prime ideal of $\mathfrak{R}\ev$ (Remark \ref{rem:radical-prime}), their intersection is given by $\left(\bigcap_{\mathfrak{p}\ev\in Spec^{\infty}(\mathfrak{R}\ev)}\mathfrak{p}\ev \right)\oplus\mathfrak{R}\od$. Moreover, for the non-graded case (see, \cite[Theorem 4.22]{CM1}), we have $\bigcap_{\mathfrak{p}\ev\in Spec^{\infty}(\mathfrak{R}\ev)}\mathfrak{p}\ev=\sqrt[\infty]{(0)_{\mathfrak{R}\ev}}$. Consequently, using Lemma \ref{lem:inftyrad}$\text{c)}$, which states that $\sqrt[\infty]{(0)_{\mathfrak{R}}}=\sqrt[\infty]{(0)_{\mathfrak{R}\ev}}\oplus \mathfrak{R}\od$, the desired result is obtained.
\end{proof}

According to Definition 4.5 of \cite{CM1}, a $\C$-ring $\mathfrak{C}$ is defined as $\C$-reduced if and only if $0$ is its unique $\C$-nilpotent element, meaning $\sqrt[\infty]{(0)_\mathfrak{C}}=(0)_\mathfrak{C}$. This definition isn't directly applicable to a $\C$-superring $\mathfrak{R}$ because all elements within $\mathfrak{J}_\mathfrak{R}$ are $\C$-nilpotent. Consequently, considering Lemma \ref{lem:inftyrad}, we propose the following definition for the super case:

\begin{definition}\label{def:superred}
We say that a $\C$-superring $\mathfrak{R}$ is $\C$-superreduced if and only if $\sqrt[\infty]{(0)_\mathfrak{R}}=\mathfrak{J}_\mathfrak{R}$.
\end{definition}

\begin{remark}\label{rmk:superred}
From Definition \ref{def:superred}, if $\mathfrak{R}$ is $\C$-superreduced, then the $\C$-ring  $\mathfrak{R}/\sqrt[\infty]{(0)_\mathfrak{R}}$ coincides with $\overline{\mathfrak{R}}=\mathfrak R/\mathfrak J_{\mathfrak R}$. Furthermore, by Lemma \ref{lem:inftyrad} $c)$, we know that $\sqrt[\infty]{(0)_{\overline{\mathfrak{R}}}}=(0)_{\overline{\mathfrak{R}}}$. Consequently, $\mathfrak{R}$ being $\C$-superreduced is equivalent to $\overline{\mathfrak{R}}$ being $\C$-reduced, which in turn is equivalent to $\mathfrak{J}_\mathfrak{R}$ being $\C$-radical.
\end{remark}

\begin{lemma}\label{lem:redmorph}
Let $\varphi:\mathfrak{R}\rightarrow\mathfrak{S}$ be an injective morphism of $\C$-superrings. If $\mathfrak{S}$ is $\C$-superreduced, then $\mathfrak{R}$ is $\C$-superreduced.
\end{lemma}
\begin{proof}  
Given that $\varphi:\mathfrak{R}\rightarrow\mathfrak{S}$ is an injective morphism of $\C$-superrings, the induced morphism of $\C$-rings $\overline{\varphi}:\overline{\mathfrak{R}}\rightarrow\overline{\mathfrak{S}}$, defined by $\overline \varphi(\overline a)=\overline{\varphi(a)}$ for all $a \in \mathfrak R$, is also injective. From  \cite[Proposition 4.33]{CM1} and Remark \ref{rmk:superred}, it then follows that $\overline{\mathfrak{R}}$ is $\C$-reduced, which is equivalent to $\mathfrak{R}$ is $\C$-superreduced.
\end{proof}

\begin{lemma} For any morphism of $\C$-superrings $\varphi: \mathfrak{R}\rightarrow \mathfrak{S}$, the following properties hold:
\begin{itemize}
\item[a)]  If $\mathfrak{I}\subseteq \mathfrak{S}$ is a superideal, then $  \sqrt[\infty]{\varphi^{-1}(\mathfrak{I})}\subseteq \varphi^{-1}(\sqrt[\infty]{\mathfrak{I}})$.
\item[b)] If $\varphi$ is injective and $\mathfrak{I}$ is a $\C$-radical superideal of $\mathfrak{S}$, then $\varphi^{-1}(\mathfrak{I})$ is a $\C$-radical superideal of $\mathfrak R$.  
\end{itemize}
\end{lemma}
\begin{proof}
\begin{itemize}
\item[a)] Let $K=\varphi^{-1}(\mathfrak{I})$. Since $\mathfrak{I}=\mathfrak{I}\ev\oplus \mathfrak{I}\od$ is a graded ideal of $\mathfrak S$, it follows that $K=K\ev\oplus K\od$ is a superideal of $\mathfrak R$. By definition, we have $\sqrt[\infty]{K}=\sqrt[\infty]{K\ev}\oplus \mathfrak{R}\od$.\\

Now, we claim that $\varphi^{-1}(\sqrt[\infty]{\mathfrak{I}})=\varphi\ev^{-1}(\sqrt[\infty]{\mathfrak{I}\ev})\oplus\mathfrak{R}\od$, where $\varphi\ev$ is the restriction of $\varphi$ to $\mathfrak{R}\ev$. To show this, consider an element $a\oplus b \in \varphi\ev^{-1}(\sqrt[\infty]{\mathfrak{I}\ev})\oplus\mathfrak{R}\od$. Since $\varphi$ is grading preserving, $\varphi(a\oplus b)=\varphi\ev(a)\oplus\varphi(b)$. It is clear that $\varphi(a\oplus b)$ belongs to $\varphi^{-1}(\sqrt[\infty]{\mathfrak{I}})=\varphi^{-1}(\sqrt[\infty]{\mathfrak{I}\ev}\oplus \mathfrak{S}\od)$.\\
Conversely, let $a\oplus b\in \varphi^{-1}(\sqrt[\infty]{\mathfrak{I}})$. This implies that $\varphi(a)\oplus\varphi(b)\in  \sqrt[\infty]{\mathfrak{I}}=\sqrt[\infty]{\mathfrak{I}\ev}\oplus \mathfrak{S}\od$. Therefore, $a\in \varphi\ev^{-1}(\sqrt[\infty]{\mathfrak{I}\ev})$ and $b \in \varphi^{-1}(\mathfrak{S}\od)\subseteq\mathfrak{R}\od$.\\
Now, because $\sqrt[\infty]{K}=\sqrt[\infty]{\varphi\ev^{-1}(\mathfrak{I}\ev)} \subseteq \varphi\ev^{-1}(\sqrt[\infty]{\mathfrak{I}\ev})$ (by, \cite[Proposition 4.39]{CM1}), we can conclude that
$$
\sqrt[\infty]{K}=\sqrt[\infty]{K\ev}\oplus\mathfrak{R}\od\subseteq \varphi\ev^{-1}(\sqrt[\infty]{\mathfrak{I}\ev})\oplus \mathfrak{R}\od.
$$

\item[b)] Let $\mathfrak{I}$ be a $\C$-radical superideal of $\mathfrak{S}$ and let $\mathfrak{K}=\varphi^{-1}(\mathfrak{I})$ be its inverse image. We know that $\mathfrak{K}\subseteq \sqrt[\infty]{\mathfrak{K}}$. Consequently, $\varphi$ induces a well-defined and injective morphism $\widehat{\varphi}:\mathfrak{R}/\mathfrak{K}\rightarrow \mathfrak{S}/\mathfrak{I} $. Since $\mathfrak{S}/\mathfrak{I}$ is $\C$-superreduced (because $\mathfrak{I}$ is $\C$-radical), Lemma \ref{lem:redmorph} implies that $\mathfrak{R}/\mathfrak{K}$ is also $\C$-superreduced. This, in turn, means $\mathfrak{K}$ is $\C$-radical.
\end{itemize}
\end{proof}

\begin{remark}
For any $\C$-ring $\mathfrak{C}$ and any nonzero ideal $I$ of $\mathfrak{C}$, we have $\sqrt[\infty]{0_{\mathfrak{C}}}=\sqrt[\infty]{I}/I$ (see \cite[Prop. 4.30]{CM1}). However, this result does not generally extend to our $\mathbb{Z}_2$-graded setting. Indeed, for a nonzero superideal $\mathfrak{I}$ of a $\C$-superring $\mathfrak{R}$, we have:  
$$
\frac{\sqrt[\infty]{\mathfrak{I}}}{\mathfrak{I}}=\frac{\sqrt[\infty]{\mathfrak{I}\ev}\oplus\mathfrak{R}\od}{\mathfrak{I}\ev\oplus \mathfrak{I}\od}=\frac{\sqrt[\infty]{\mathfrak{I}\ev}}{\mathfrak{I}\ev}\oplus\frac{\mathfrak{R}\od}{\mathfrak{I}\od}=\sqrt[\infty]{0_{\mathfrak{R}\ev}}\oplus \frac{\mathfrak{R}\od}{\mathfrak{I}\od}.
$$ 
This expression is equal to $\sqrt[\infty]{0_{\mathfrak{R}\ev}}\oplus \mathfrak{R}\od$ if and only if the odd part $\mathfrak{I}\od$ is zero.
\end{remark}

 \section{The spectrum of a $\C$-Superring}\label{superspec}
To construct a geometric framework in the spirit of algebraic geometry via ringed spaces, the primary step is assigning a topological space to a given $C^\infty$-superring. This topological space must be endowed with a sheaf of $C^\infty$-superrings featuring local stalks. Formally, this construction amounts to defining a functor from the category of $C^\infty$-superrings to the category of locally $C^\infty$-ringed superspaces.\\

We have identified two distinct methods for constructing this spectrum:\\

\begin{itemize}
\item \textbf{The $\C$-Prime Spectrum}. The first approach, detailed in the following subsection, defines the $\C$-prime spectrum of a $\C$-superring.  This spectrum is simply the set of all $\C$-radical prime ideals,  equipped with the Zariski topology. This method aligns with the algebraic techniques found in works by Moerdijk and Reyes \cite{MRI} or Cerqueira and Mariano \cite{CM1}.\\

\item \textbf{The Maximal Spectrum}. The second approach, presented in the subsequent section, defines the spectrum of a 
$\C$-ring ``\'a la Joyce" (see \cite{J}). This involves considering morphisms from the $\C$-superring to the real numbers, endowed with the Gelfand topology. Alternatively, this can be understood as the set of maximal ideals of the $\C$-ring. While this set is generally a proper subset of the $\C$-prime spectrum, they coincide in specific cases.
\end{itemize}

\subsection{The prime spectrum}
The most straightforward method for defining the spectrum of a $\C$-ring involve its prime spectrum. However, localization at a prime ideal isn't always local unless the ideal is $\C$-radical \cite[Proposition 4.3]{CM1}. This same issue applies to $\C$-superring, which means prime superideals aren't suitable for constructing a local structure sheaf. Therefore, we use $\C$-radical prime superideals, since their locality is established by Theorem \ref{rad-local}. The following definition, adapted for a $\C$-superring, draws inspiration from the concepts presented in \cite{CM1}.

\begin{definition} Let $\mathfrak{R}$ be a $\C-$ superring. We define the \emph{$\C$-spectrum of $\mathfrak{R}$}, denoted as ${\rm sSpec}^{\infty}(\mathfrak{R})$, to be the set of all $\C$-radical prime superideals of $\mathfrak R$:
$$
{\rm sSpec}^{\infty}(\mathfrak{R}) = \{ \mathfrak{p} \in {\rm sSpec}(\mathfrak{R}) \mid \mathfrak{p} \, \mbox{is} \,\, \C\mbox{-radical} \} \, .
$$
This set is endowed with the smooth Zariski topology, which is generated by a base of principal open sets of the form:
$$
D^{\infty}(a) = \{ \mathfrak{p} \in {\rm sSpec}^{\infty}\,(\mathfrak{R}) \mid a \notin \mathfrak{p} \}. 
$$ 
Here, $a$ is an element of $\mathfrak R$. The closed sets in this topology are given by:
$$
Z^{\infty}(\mathfrak{I})=\{\mathfrak{p}\in  {\rm sSpec}^{\infty}(\mathfrak{R})\mid \mathfrak{p}   \supseteq \mathfrak{I}\} \, ,
$$
\end{definition}
where $\mathfrak{I}$ is any superideal of $\mathfrak R$.

\begin{lemma}\label{lem:superrad}
Let $\mathfrak{R}$ be a $\C$-superring. If $\mathfrak{I}$ is a superideal of $\mathfrak{R}$, then we have:
$$
Z^\infty(\mathfrak{I})=Z^\infty(\sqrt[\infty]{\mathfrak{I}})
$$
\end{lemma}
\begin{proof}
Every prime superideal $\mathfrak{p}$ of $\mathfrak{R}$ can be expressed in the form $\mathfrak{p}=\mathfrak{p}_0\oplus \mathfrak{R}\od$. The inclusion $\mathfrak{p}\supseteq \mathfrak{I}$ directly implies $\mathfrak{p}\ev\supseteq \mathfrak{I}\ev$. Given $\mathfrak{p}\in{\rm sSpec}^{\infty}(\mathfrak{R})$, Proposition \ref{prop:sradical} b) ensures that $\mathfrak{p}\ev\supseteq \sqrt[\infty]{\mathfrak{I}\ev}$. This further means $\mathfrak{p}\supseteq \sqrt[\infty]{\mathfrak{I}\ev}\oplus \mathfrak{R}\od$, which establish the inclusion $Z^\infty(\mathfrak{I})\supseteq Z^\infty(\sqrt[\infty]{\mathfrak{I}})$.\\

Conversely, suppose $\mathfrak{p}\supseteq \sqrt[\infty]{\mathfrak{I}}$. Then it follows that $\mathfrak{p}_0\supseteq \sqrt[\infty]{\mathfrak{I}\ev}$. Consequently, we have $\mathfrak{p}\supseteq \sqrt[\infty]{\mathfrak{I}\ev}\oplus \mathfrak{R}\od \supseteq \mathfrak{I}\ev \oplus \mathfrak{R}\od= \mathfrak I$, since $\sqrt[\infty]{\mathfrak{I}\ev} \subseteq \mathfrak{I}\ev$. This completes the proof.
\end{proof}

\begin{remark}
Following the reasoning in the proof of Lemma \ref{lem:superrad}, we can readily construct a bijection between sets $D^{\infty}(a) = \{ \mathfrak{p} \in {\rm sSpec}^{\infty}\,(\mathfrak{R})\mid a \notin \mathfrak{p} \}$ and $D^{\infty}(a\ev) = \{ \mathfrak{p}\ev \in {\rm Spec}^{\infty}\,(\mathfrak{R}\ev) \mid a\ev \notin \mathfrak{p}\ev \}$. Consequently, this establishes a homeomorphism between $\text{sSpec}^\infty(\mathfrak{R})$ and $Spec^\infty(\mathfrak{R}\ev)$.
\end{remark}

As this remark shows, the topological space associated with a $\C$-superring is entirely determined by its even part. This implies that $Z^\infty (\mathfrak{I})=Z^\infty(\mathfrak{I}\ev)$, for any superideal $\mathfrak{I}$ in $\mathfrak{R}$. These relationships lead the following basic properties, whose proofs are omitted because they are straightforward ``super'' extension of Propositions 5.3, 5.4, and 5.5 of \cite{CM1}.

\begin{proposition} Let $\mathfrak{R}$ be a $\C$-superring, and let $a,b \in \mathfrak{R}$. The following conditions are equivalent:
    \begin{itemize}
  \item[(i)] $D^{\infty}(a) \subseteq D^{\infty}(b)$.
  \item[(ii)] $Z^{\infty}(a) \supseteq Z^{\infty}(b)$.
  \item[(iii)] $a \in \sqrt[\infty]{(b)}$.
  \item[(iv)] $\sqrt[\infty]{(a)} \subseteq \sqrt[\infty]{(b)}$.
\end{itemize}  
\end{proposition}

\begin{proposition}  Let $\mathfrak{R}$ be a $\C$-superring. The following properties hold:
\begin{itemize}
\item[(i)]  For any superideals $\mathfrak{a}$ and $\mathfrak{b}$ of $\mathfrak{R}$, we have $\sqrt[\infty]{\mathfrak{a}} \subseteq \sqrt[\infty]{\mathfrak{b}}$ if and only if $Z^{\infty}(\mathfrak{b}) \subseteq Z^{\infty}(\mathfrak{a})$.
\item[(ii)] $Z^{\infty}\,((0)) = {\rm sSpec}^{\infty}\,(\mathfrak{R})$ and $Z^{\infty}((1_{\mathfrak{R}})) = \emptyset$.
\item[(iii)] For any superideals $\mathfrak{a}$ and $\mathfrak{b}$ of $\mathfrak{R}$, we have $Z^{\infty}(\mathfrak{a}) \cup Z^{\infty}(\mathfrak{b}) = Z^{\infty}(\mathfrak{a}\cdot \mathfrak{b}) = Z^{\infty}(\mathfrak{a} \cap \mathfrak{b})$;
\item[(iv)] $\bigcap_{i \in I}Z^{\infty}(\mathfrak{a}_i) = Z^{\infty}\left( \sum_{i \in I} \mathfrak{a}_i\right)$, where $\sum_{i \in I} \mathfrak{a}_i$ denotes the sum superideal generated by the family $\{ \mathfrak{a}_i\}_{i \in I}$ of superideals of $\mathfrak{R}$.
\end{itemize}
\end{proposition}
\begin{definition}
Let $\mathfrak{R}=\mathfrak R\ev \oplus \mathfrak R\od$ be a $\C$-superring and let $X={\rm Spec}^{\infty}(\mathfrak{R}\ev)$. We define its structure sheaf $\mathcal{O}_X$ as a sheaf of $\C$-superrings on $X$ satisfying the following conditions:
\begin{itemize}
\item  For any $a\in \mathfrak{R}\ev$, the $\C$-superring of sections over the principal open set $D^{\infty}(a)$, denoted $\mathcal{O}_X(D^{\infty}(a))$, is given by $\mathfrak{R}\{ a^{-1} \}$.

\item For any $\mathfrak{p}\ev\in X$, the stalk at $\mathfrak{p}\ev$, $\mathcal{O}_{X, \mathfrak p\ev}$, is  the local $\C$-superring $\mathfrak{R}_{\mathfrak{p}\ev}$.
\end{itemize}
   
\end{definition}

This definition extends the construction given in \cite[section 1]{MRII} directly to the super context.

\subsection{The space of $\R$-points}
The second approach to associating  a topological space with a $\C$-superring, developed by Dubuc \cite{D1} and Joyce \cite{J}, defines the spectrum of a $\C$-ring $\mathfrak{C}$ as its set of  $\R$-points. This set consists of all $\C$-ring morphisms from $\mathfrak{C}$ to $\R$, or, equivalently, the set of maximal ideals of $\mathfrak C$. This alternative construction immediately prompts two key questions: What is the relationship, if any, between the set of $\R$-points and the $\C$-radical spectrum? And what criteria guide the selection of one approach over the other? 

The first question has a straightforward answer: the space of $\R$-points is generally a subspace of the $\C$-radical spectrum. This is because $\R$-points correspond to maximal ideals, which are themselves $\C$-radical prime ideals. To address the second question, we must consider the specific problem we aim to solve or develop within the realm of $\C$-rings. For example, Joyce's work is primarily oriented towards real differential geometry; thus, it is more advantageous in that context to have residue fields isomorphic to $\R$.

In this subsection, we define the super-version of this maximal spectrum. We then investigate natural extensions of key concepts, such as \textit{fair} or \textit{germ-determined}, which are widely studied for $\C$-rings. Finally, we prove that our superspectrum is adjoint to the global sections functor and present some applications of this result.

\begin{lemma}
For any $\C$-superring $\mathfrak{R}$, the spaces of $\R$-points of the $\C$-rings $\mathfrak{R}\ev $ and $\overline{\mathfrak{R}}$ are homeomorphic.
\end{lemma}
\begin{proof}
Let $\overline{x}:\overline{\mathfrak{R}}\rightarrow\R$ be a $\R$-point of $\overline{\mathfrak{R}}$. Consider the canonical projection $\pi:\mathfrak{R}\ev\rightarrow\overline{\mathfrak{R}}$. We can define a map $x:\mathfrak{R}\ev\rightarrow\R$, by setting $x(r)=\overline{x}(\pi(r))$. This map $x$ is clearly a well-defined morphism of $\C$-rings. 

Conversely, given a morphism $x:\mathfrak{R}\ev\rightarrow\R$, we can define  $\overline{x}:=\pi\circ x$. This $\overline x$ is a well-defined morphism of $\C$-rings, making it an $\R$-point of $\overline{\mathfrak{R}}$. 

Consequently, we have established a bijection between the sets of $\R$-points of $\mathfrak R\ev$ and $\overline{\mathfrak{R}}$. Let $X_{\mathfrak R\ev}$ and $X_{\overline{\mathfrak R}}$ denote the spaces of $\R$-points of $\mathfrak R\ev$ and $\overline{\mathfrak R}$, respectively. This bijection is a homeomorphism because it maps basic open sets from $X_{\mathfrak R\ev}$ to  basic open sets of $X_{\overline{\mathfrak R}}$. Specifically, if $r\in\mathfrak{R}\ev$, then the image of the basic open $U_r=\{x\in X_{\mathfrak{R}{\ev}}\mid x(r)=0\}$ in $X_{\overline{\mathfrak R}}$ is precisely the basic open set $U_{\overline{r}}=\{\overline{x}\in X_{\overline{\mathfrak{R}}{\ev}}\mid \overline{x}(\overline{r})=0\}$ in $X_{\overline{\mathfrak R}}$.
\end{proof}

\begin{definition}
For a $\C$-superring $\mathfrak{R}$, its space of $\R$-points, denoted $X_\mathfrak{R}$, is defined as the space of $\R$-points of its even part $\mathfrak{R}\ev$. Thus,  $X_\mathfrak{R}:=X_{\mathfrak{R}{\ev}}$.   
\end{definition}

One can readily verify that a bijection exist between the set of $\R$-points of $\mathfrak{C}$ and the set of its maximal ideals $\mathfrak{m}$ for which the quotient ring  $\mathfrak{C}/\mathfrak{m}$ is isomorphic to $\R$. This set is variably referred to as the \textit{real spectrum of} $\mathfrak{C}$ by some authors (e.g., \cite{GS}) and the \textit{Archimedean spectrum} by others  (e.g., \cite{MRII}). However, it should not be confused with the real spectrum as defined in \cite{CM1}. Since maximal ideals are always $\C$-radical, the $\R$-points form a subspace of the $\C$-prime spectrum.\\

\begin{example}
Let $E$ be a set and let $\mathfrak{R}=\C(\R^E)[\theta^1,\ldots,\theta^n]/\mathfrak{I}$ be a $\C$-superring. The reduced part of $\mathfrak{R}$ is a $\C$-ring of the form $\mathfrak{\overline{R}}=\C(\R^{E})/\overline{\mathfrak{I}}$. It is easy to see that the $\R$-points of $\mathfrak{R}$ are precisely the evaluation morphisms at points in the zero set of ideal $\overline{\mathfrak{I}}$. Moreover, this zero set is a subspace of $\R^E$ with the topology inherited from the product-topology on $\R^E$.   
\end{example}

We have defined $X_{\mathfrak R}$ as the topological space of $\R$-points associated with a $\C$-superrings $\mathfrak R$. Our next step is to define a sheaf of $\C$-superrings on $X_{\mathfrak R}$, which will provide the necessary ingredients for defining the superspectrum. Prior to that, however, we will introduce the notion of localization of a $\C$-superring at an $\R$-point, and then extend some familiar concepts and result from the realm of $\C$-rings.

\begin{definition}\label{LocR-point}
Let $x$ be an $\R$-point of a $\C$-superring $\mathfrak{R}=\mathfrak R\ev \oplus \mathfrak R\od$. The localization of $\mathfrak{R}$ at the $\R$-point $x$, denoted $\mathfrak{R}_x$, is the local $\C$-superring $\mathfrak{R}\{ S^{-1} \}$ where $S=\{r\in \mathfrak{R}\ev \mid x(r)\neq 0\}$. 
\end{definition}
\begin{lemma}\label{lem:superloc}
 Consider a $\C$-superring $\mathfrak{R}=\mathfrak{R}\ev\oplus\mathfrak{R}\od$ and an $\R$-point $x\in X_{\mathfrak{R}\ev}$. Let $S=\{r\in \mathfrak{R}\ev \mid x(r)\neq 0\}$. Then, $\mathfrak{R}_x$ decomposes as $\mathfrak{R}_x=(\mathfrak{R}\ev)_{x}\oplus(\mathfrak{R}\od)_{x}$. Here, $(\mathfrak{R}\ev)_{x}$ represents the localization of the $\C$-ring $\mathfrak{R}\ev$ with respect to $S$, and $(\mathfrak{R}\od)_{x}$ is the resulting $(\mathfrak{R}\ev)_{x}$-module formed by localizing the $\mathfrak{R}\ev$-module $\mathfrak{R}\od$ at S.
\end{lemma}

\begin{proof}
The localization morphism $\mathcal{L}_x:\mathfrak{R}\rightarrow(\mathfrak{R}\ev)_{x}\oplus(\mathfrak{R}\od)_{x}$ is defined as $\pi_x\oplus \rho_x$, where $\pi_x$ and $\rho_x$ are the respective localization morphisms of $\mathfrak{R}\ev$ and $\mathfrak{R}\od$ at $x$. The universal properties of localization for $\pi_x$ and $\rho_x$ imply the universal property of localization for $\mathcal{L}_x$, thus completing the proof.
\end{proof}

\begin{lemma}
 Let $\mathfrak{R}=\mathfrak{R}\ev\oplus\mathfrak{R}\od$ be a $\C$-superring. Let $x\in X_{\mathfrak{R}\ev}$ and $\overline{x} \in X_{\overline{\mathfrak{R}}}$ be $\R$-points of $\mathfrak{R}$ and $\overline{\mathfrak{R}}$, respectively. Then, we have the following isomorphism of $\C$-rings: $\overline{\mathfrak{R}_x}\cong \overline{\mathfrak{R}}_{\overline{x}}$.
\end{lemma}

\begin{proof}
Note that, due to the functoriality of the projection $q:\mathfrak{R}\rightarrow\overline{\mathfrak{R}}$, the following diagram commutes:
\[ 
\begin{tikzcd}
\mathfrak{R}\arrow{r}{\mathcal{L}_x} \arrow[swap]{d}{q} & \mathfrak{R}_x\arrow{d}{q_x} \\
\overline{\mathfrak{R}} \arrow{r}{\overline{\mathcal{L}_{x}} }& \overline{\mathfrak{R}_{x}}
\end{tikzcd}
\]
By definition, $\overline{\mathfrak{R}}_{\overline{x}}$ is the localization of $\overline{\mathfrak{R}}$ at the set $T=\{\overline{r} \in \overline{\mathfrak{R}}\mid \overline{x}(\overline{r})\neq 0\}$. Consider any $r\in\mathfrak{R}\ev$ such that $\overline{r}\in T$. From the definition of the $\R$-point $\overline{x}$ and given that $\overline{x}(\overline{r})\neq 0$, it follows that $x(r)\neq 0$. This implies that $\mathcal{L}_x(r)$ is invertible in $\mathfrak{R}_x$. Consequently, its image  $\overline{\mathcal{L}_x(r)}=q_x(\mathcal{L}_x(r))$, under the projection $q_x$, is an invertible element in $\overline{\mathfrak{R}_x}$. By the commutativity of the diagram, we have $\overline{\mathcal{L}_{x}}(\overline{r})=\overline{\mathcal{L}_{x}(r)}$. It follows that $\overline{\mathcal{L}_{x}}(\overline{r})$ is invertible  for any $\overline{r}\in T$. 

Therefore, by the universal property of localization for $\C$-rings, we conclude that $\overline{\mathfrak{R}_x}\cong \overline{\mathfrak{R}}_{\overline{x}}$.
\end{proof}
  
The localization morphism for $\C$-rings is generally not surjective. However, when localizing at an $\R$-point, the localization morphism is known to be surjective \cite[Proposition 2.14]{J}. This property extends to $\C$-superrings, as the following lemma shows.

 \begin{lemma}\label{locsobre}
Let $\mathfrak{R}$ be a $\C$-superring and let $x \in X_{\mathfrak{R}}$ be an $\R$-point. The localization morphism $\mathcal{L}_x: \mathfrak R \to \mathfrak R\ev$ is surjective.
 \end{lemma}

\begin{proof}
 Since the even part of the ring $\mathfrak R\ev$ is a $\C$-ring, the localization morphism $(\mathcal{L}_x)\ev:\mathfrak{R}\ev\longrightarrow(\mathfrak{R}{\ev})_{x}$ is surjective. Furthermore, the odd elements of $\mathfrak{R}_{x}$ are the form $ar$ with $r\in \mathfrak{R}_1$ and $a\in (\mathfrak{R}{\ev})_x$. Thus, if $b$ is a preimage of $a$ under $(\mathcal{L}_x)\ev$, then $br$ is a preimage of $ar$ in $(R\od)_x$. Consequently, $\mathcal{L}_x$ is surjective. 
\end{proof}

\begin{remark}\label{locx-cociente}

According to \cite[Proposition 2.14]{J}, taking $x$ as $\R$-point of $\mathfrak{R}\ev$ the kernel of the localization morphism $(\mathcal{L}_x)\ev: \mathfrak R\ev \rightarrow (\mathfrak R\ev)_x$ is given by $$(I_x)\ev=\{a\in\mathfrak{R}\ev \mid \exists\; {b\in\mathfrak{R}{\ev}}\; \text{with}\; x(b)\neq 0 \; \text{in}\; \R \; \text{and} \;  ab=0 \; \text{in }\; \mathfrak R\ev \}$$ 

Therefore, the kernel of the localization morphism $\mathcal{L}_x$ is given by
$$I_x=(I_x)\ev\mathfrak{R}$$ which is an ideal of $\mathfrak{R}$. Thus, the localization of $\mathfrak{R}$ on $x$ is isomorphic to the quotient $\mathfrak{R}/I_x$
\end{remark}

The proposition presented below highlights the functorial nature of localizing a $\C$-superring at an $\R$-point.

\begin{proposition}\label{Rpointmor}
Let $\varphi:\mathfrak{R}\rightarrow \mathfrak{S}$ be a morphism of finitely generated $\C$-superrings. For an $\R$-point $x \in X_{\mathfrak S}$, define $y=x\circ \varphi\ev \in X_{\mathfrak R}$. Then, there exists an induced morphism $\varphi_x:\mathfrak{R}_y\rightarrow\mathfrak{S}_x$  making the following diagram commute:
\[ 
\begin{tikzcd}
\mathfrak{R}\arrow{r}{\varphi} \arrow[swap]{d}{\mathcal{L}_y} & \mathfrak{S}\arrow{d}{\mathcal{L}_x} \\
\mathfrak{R}_y \arrow{r}{\varphi_x }& \mathfrak{S}_x .
\end{tikzcd}
\] 
\end{proposition}

\begin{proof}
According to Remark \ref{locx-cociente}, we have $\mathfrak{R}_y\cong \mathfrak{R}/I_y$ and $\mathfrak{S}_x\cong \mathfrak{S}/I_x$. Thus we only have to check that there is a well defined morphism f $\C$-superrings $\tilde{\varphi}:\mathfrak{R}/I_y\rightarrow  \mathfrak{S}/I_x$ such that the diagram \[ 
\begin{tikzcd}
\mathfrak{R}\arrow{r}{\varphi} \arrow[swap]{d}{} & \mathfrak{S}\arrow{d}{} \\
\mathfrak{R}/I_y \arrow{r}{\tilde{\varphi} }& \mathfrak{S}/I_x .
\end{tikzcd}
\] 
where the vertical arrows are quotient projection. \\
Indeed for any $\title{r}\in\mathfrak{R}/I_y $ we define $\tilde{\varphi}(\tilde{r}):= \widetilde{\varphi(r)}$. Without loss of generality, take $r=ab$ with $a\in (I_y)\ev$ and $b$ any element of $\mathfrak{R}$. Since $a\in (I_y)\ev$, there exists $c$ in $\mathfrak{R}\ev$ such that $y(c)\neq 0$ but $ac=0$. Thus, $\varphi(ac)=\varphi(a)\varphi(c)=0$ and, by the definition of $y$, we have $x(\varphi(c))\neq 0$ which proves that $\varphi(a)\in (I_x)\ev$ and consequently $ab\in I_x$. In conclusion, $\tilde{\varphi}$ is well defined.
\end{proof}

\begin{example}
Let $\mathfrak{C}$ be a $\C$-ring and consider the $\C$-superring $\mathfrak{R}=\mathfrak{C}[\theta^1,\ldots,\theta^n]^\pm$. Let $x$ be an $\R$-point of $\mathfrak R$. The reduced part of $\mathfrak R$ corresponds to $\mathfrak{C}$, and its localization at $x$ is given by $\mathfrak{R}_x=\mathfrak{C}_x[\theta^1,\ldots,\theta^n]^\pm$. To see this, consider a general element $C\in \mathfrak R$:
$$
C=c_0+\sum_{|I|=2n} c_I \, \theta^I+\sum_{|K|=2n+1} d_K\, \theta^K \, 
$$
where $c_0, \,c_I,\, d_K \in \mathfrak{C}$. The localization morphism $\mathcal{L}_x$ maps this element to:
$$
\mathcal{L}_x(C)=\pi_x(c_0)+\sum_{|I|=2n}\pi_x(c_I) \, \theta^I+\sum_{|K|=2n+1}\pi_x(d_K)\, \theta^K . 
$$
\end{example}

\begin{example}
Let $M=(|M|, \mathcal{O}_M)$ be a supermanifold of dimension $p\mid q$. The sheaf $\mathcal{O}_M$ is defined for any open set $U\subseteq | M|$ as $\mathcal{O}_M(U)=\C(U)\otimes\bigwedge_{\R} (\theta^1, \cdots, \theta^q)$ \cite[Definition 4.1.2]{CCF}, where $\C(U)$ is the $\C$-ring  of smooth functions on $U$. Therefore, for any $p\in |M|$, the localization of $\mathcal{O}_{M}$ at $p$ is given by $\C_p(|M|)\otimes\bigwedge_{\R} (\theta^1, \cdots, \theta^q)$, where $\C_p(|M|)$ denotes the $\C$-ring of germs of smooth functions at $p$. 
\end{example}

\subsection{Fair $\C$-superrings}

Having successfully extended the notion of localization at an $\R$-point to $\C$-superrings, our next step is to investigate the relationship between a $\C$-superring and its localizations. Specifically, we seek to identify conditions under which a $\C$-superring  is determined by its germs.

Recall that in the context of $\C$-rings, a fair C$^\infty$-ring is one that is finitely generated and germ-determined \cite[subsection 2.4]{J}. When extending the concept of fairness to $\C$-superrings, we have identified two possible definitions, which are not equivalent. The first option is a straightforward extension of Definition \ref{Fair} to  $\C$-superrings, simply replacing the term ``ring" with ``superring" in that definition. The second option is to define a $\C$-superring as fair if its reduced part is a fair $\C$-ring. The following example demonstrates why these two definitions are not equivalent.

\begin{example}\label{nonfair}
Let $f:\R\rightarrow\R$ be a smooth function such that $f(x)>0$ for all $x\in (0,1)$ and $f(x)=0$ elsewhere. Consider $K$ as the ideal of $\C(\R)$ generated by finite sums of the form $\sum_{a\in A\subset \mathbb{Z}}\; g_a(x)f(x-a)$, where $g \in \C(\R)$  and $A$ is a finite subset of $\mathbb{Z}$. Example 2.21 in \cite{J} shows that $K$ is a non-fair ideal of $\C(\R)$, as defined by Joyce in \cite[Definition 2.16]{J}. Consequently, the quotient $\mathfrak{C}= \C(\R)/K$ is also not fair \cite[Definition 2.16]{J}. 

Now, consider the ideal $\mathfrak{K}=\{k\, \theta^1 \theta^2 \mid k\in  K\}$ of $\C(\R^{1|2})=\C(\R)[\theta^1, \theta^2]^\pm$ and define the $\C$-superring $\mathfrak{R}=\C(\R^{1|2})/\mathfrak{K}$. To calculate the reduced part of $\mathfrak{R}$ we proceed as in Example \ref{reduced-quotient}, thus $\overline{\C(\R^{1|2})/\mathfrak{K}}=\C(\R)/[\mathfrak{K}\cap \C(\R)]=\C(\R)/(0)=\C(\R)$ which is clearly fair as it is finitely presented. Let $h\in \C(\R)/K$ be a non-zero function such that $\pi_p(h)=0$ for all $p\in Z(K)$. By the definition of $\mathfrak{K}$ we have that   $\mathcal{L}_p(h\theta^{12})=\pi_p(h)\theta^{12}=0\in \mathfrak{R}_p$ for all $p$. However, since $h\theta^{12}\neq 0$, it follows that $\mathfrak{K}$ is not a fair ideal, hence $\mathfrak{R}$ is not fair even though its reduced part $\overline{\mathfrak{R}}$ is.
\end{example}

We have decided to define a fair $\C$-superring by directly extending the concept from $\C$-rings. This approach directly builds on the definition provided by Joyce in \cite[Definition 2.16]{J}.

\begin{definition}\label{SuperFair}
Consider the split $\C$-superring $\C(\R^{p\mid q})=\C(\R)[\theta^1,\cdots ,\theta^q]^\pm$, and let $I$ be a superideal of it. For any $\R$-point $x$ of $\C(\R^{p\mid q})$, we denote by $I_x$ the image of $I$ under the localization morphism $\mathcal{L}_x:\C(\R^{p\mid q}) \to \C(\R^{p\mid q})_x $.

We say that $I$ is a \emph{fair superideal} if, for any $r\in \C(\R^{p\mid q})$,  $\mathcal{L}_x(r)\in I_x$ for all $x \in X_{\C(\R^{p\mid q})}$ if and only if $r\in I$. A finitely generated $\C$-superring $\mathfrak{R}=\C(\R^{p\mid q})/I$ is \emph{fair} if $I$ is a fair superideal of $\C(\R^{p\mid q})$.

Equivalently, a finitely generated $\C$-superring $\mathfrak{R}$ is fair if, for all $a \in \mathfrak R$ and $x \in X_{\mathfrak R}$, $\mathcal{L}_x(a)=0$ if and only if $a=0$.
\end{definition}

\begin{proposition}\label{prop:equivfair}
Let $\mathfrak{R}=\mathfrak{C}[\theta^1\ldots\theta^q]^\pm$ be a split $\C$-superring, where $\mathfrak{C}$ is a finitely generated $\C$-ring. Then $\mathfrak{C}$ is fair (in the sense of Definition \ref{Fair}) if and only if $\mathfrak{R}$ is fair (in the sense of Definition \ref{SuperFair}).
\end{proposition}

\begin{proof}
Recall that $\overline{\mathfrak R}=\mathfrak C$ and $X_{\mathfrak R\ev}\cong X_{\mathfrak C}$.\\

($\Rightarrow$) Assume $\mathfrak C$ is fair. This means that for all $x \in X_\mathfrak C$, if $\pi_x(c)=0$ for some $c \in \mathfrak C$, then $c=0$ (Definition \ref{Fair}). Now, an arbitrary element $C \in \mathfrak{R}$ can be uniquely written as:
$$
C=c_0+\sum_{|I|=2n}c_I\,  \theta^I+\sum_{|K|=2n+1}d_K\,  \theta^K \, ,  
$$ 
where $c_0, c_I, d_K \in \mathfrak C $. Suppose that $\mathcal{L}_x(C)= 0$ for all $x \in X_{\mathfrak R\ev}$. We have
$$
\mathcal{L}_x(C)=\pi_x(c_0)+\sum_{|I|=2n}\pi_x(c_I) \, \theta^I+\sum_{|K|=2n+1}\pi_x(d_K)\, \theta^K.
$$ 
Since $\mathcal{L}_x(C)= 0$, it must be that $\pi_x(c_0)=0$,  $\pi_x(c_I)=0$, and  $\pi_x(d_K)=0$. This holds for all $x$. Since $\mathfrak{C}$ is fair, it follows that $c_0=0=c_I=d_K$ and therefore $C=0$. Thus, $\mathfrak{R}$ is a fair $\C$-superring.\\

($\Leftarrow$) Conversely, suppose $\mathfrak{R}$ is fair. This means that for all $x \in X_{\mathfrak R\ev}$, if $\mathcal{L}_x(C)=0$ for some $C \in \mathfrak R$, then $C=0$. Let $c \in \mathfrak C$ be an arbitrary element such that $\pi_x(c)=0$ for all  $x \in X_{\mathfrak C} \cong X_{\mathfrak R\ev}$.\\
Construct $C=c+0\,\theta^1\in \mathfrak{R}$. Then, $\mathcal{L}_x(C)=\pi_x(c)=0$ for all $x$. Since $\mathfrak{R}$ is fair, $C=0$, which implies $c=0$. Therefore, $\mathfrak{C}$ is a fair $\C$-ring.  
\end{proof} 
Proposition \ref{prop:equivfair} is instrumental in constructing examples of both fair and non-fair split $\C$-superrings. Moreover, this proposition implies that the $\C$-superring $\mathfrak{R}$ of Example  \ref{nonfair} is not split\\

Analogous to the non-graded setting \cite[Definition 2.20]{J}, a fair $\C$-superring can be associated with any given finitely generated $\C$-superring. This association is functorial, as detailed in the subsequent definition.

\begin{definition}\label{def:fairfication}
Let $\mathfrak{R}$ be a finitely generated $\C$-superring. Define the superideal $I_{\mathfrak R}$ as $I_{\mathfrak R}=\{r \in \mathfrak{R}\mid\mathcal{L}_x(r)=0,\, \forall x\in X_{\mathfrak{R}\ev}\}$. The \emph{fairfication of $\mathfrak{R}$}, denoted $\mathfrak{R}^{fa}$, is the fair $\C$-superring given by  $\mathfrak{R}^{fa}:=\mathfrak{R}/I_{\mathfrak{R}}$. Moreover, if $\varphi:\mathfrak{R}\rightarrow\mathfrak{S}$ is a morphism of finitely generated $\C$-superrings, then the following diagram commutes:

\[ 
\begin{tikzcd}
\mathfrak{R}\arrow{r}{} \arrow[swap]{d}{\varphi} & \mathfrak{R}^{fa}\arrow{d}{\varphi^{fa}} \\
\mathfrak{S} \arrow{r}{}& \mathfrak{S}^{fa}
\end{tikzcd}
\] 
\end{definition}

The following lemma provides a criterion for the maximal spectrum  of a $\C$-superring to be dense in its $\C$-spectrum.

\begin{lemma}\label{lem:densemax}
Let $\mathfrak{R} $ be a $\C$-superring. The maximal spectrum of $\mathfrak{R}$, denoted $\text{sMax}(\mathfrak{R})$, is dense in $\text{sSpec}^{\infty}(\mathfrak{R})$ if and only if the Jacobson ideal $J(\mathfrak{R})=\bigcap_{\mathfrak m \in \text{sMax}(\mathfrak R)} \; \mathfrak m$ of $\mathfrak{R}$ is equal to $\sqrt[\infty]{(0)_\mathfrak{R}}$.
\end{lemma}

\begin{proof}
From Corollary \ref{prime_intersection}, we know that $\sqrt[\infty]{(0)_\mathfrak{R}}=\bigcap_{\mathfrak p \in \text{sSpec}^\infty(\mathfrak R)}\; \mathfrak p$. Consequently, $\sqrt[\infty]{(0)_\mathfrak{R}}$ is always a subset of the Jacobson ideal $J(\mathfrak{R})$. Thus, we only need to prove the converse.\\

Since the principal open sets $\{ D^\infty(a)\}_{a \in \mathfrak R}$ form a basis for the topological space $\text{sSpec}^{\infty}(\mathfrak{R})$, $\text{sMax}(\mathfrak{R})$ is dense in $\text{sSpec}^{\infty}(\mathfrak{R})$ if and only if for any $r\in\mathfrak{R}$, either $D^\infty(r)=\emptyset$ or $\text{sMax}(\mathfrak{R})\cap D^\infty(r)\neq\emptyset$. The latter condition is equivalent to stating that either $r$ belongs to all $\C$-radical prime superideals, or there exists a maximal superideal $\mathfrak m$ such that $\mathfrak m \in D^\infty(r)$. This holds if and only if $r\in \sqrt[\infty]{(0)_\mathfrak{R}}=\bigcap_{\mathfrak p \in \text{sSpec}^\infty(\mathfrak R)}\; \mathfrak p $ or $r\notin J(\mathfrak{R})$. Therefore, $\text{sMax}(\mathfrak{R})$ is dense in $\text{sSpec}^{\infty}(\mathfrak{R})$ if and only if the following implication holds: if $r\in J(\mathfrak{R})$, then $r\in \sqrt[\infty]{(0)_\mathfrak{R}}$.
\end{proof}

\begin{proposition}\label{dense}
Let $\mathfrak{R}$ be a fair $\C$-superring. Then, the set of $\R$-points $X_\mathfrak{R}$ is homeomorphic to a dense subset of $\text{sSpec}^{\infty}(\mathfrak{R})$.
\end{proposition}

\begin{proof}
We can identified the set $X_\mathfrak{R}$ with $\text{sMax}$. Therefore, by Lemma \ref{lem:densemax}, we only need to prove that $J(\mathfrak{R})$ is a subset of $\sqrt[\infty]{(0)_\mathfrak{R}}$. 

To establish this, consider any $r\in J(\mathfrak{R})$. Since every maximal ideal of $\mathfrak{R}$ is of the form $M_x =\text{ker}(x)$ for some $x\in X_\mathfrak{R}$, it follows that $\mathcal{L}_x(r)=0$ for all such $x$. The fairness of $\mathfrak{R}$ then yields $r=0$. Therefore, $r\in \sqrt[\infty]{0_\mathfrak{R}}$, which completes the proof.\\
\end{proof}

\begin{remark}
In commutative algebra, a ring is defined as a \emph{Jacobson ring} if its nilradical coincides with its Jacobson radical in every quotient ring. As a consequence of the Nullstellensatz, finitely generated and reduced algebras over algebraically closed fields are examples of Jacobson rings (see \cite[Chapter 15]{AltKlei}). For such algebras, the maximal spectrum is a dense subset of the prime spectrum. This property directly leads to the concept of an algebraic scheme, which is defined as the spectrum of a finitely generated and reduced algebra over an algebraically closed field. Thus, algebraic schemes represents a natural generalization of algebraic varieties.\\

A possible interpretation of fair $\C$-ring is that they are the direct analogues of algebraic schemes in the smooth setting. Moreover, our previous results suggest that fair $\C$-superrings provide the appropriate setting for studying generalizations of smooth supergeometry. In fact, Corollary \ref{cor:fairestriction} provides strong evidence for this proposal, as it generalizes a well-known result for supermanifolds to fair affine superschemes (see \cite[Prop. 3.4.5]{CCF}).

\end{remark}

Let $\mathfrak{R}$ be a $\C$-superring. The fairfication $\mathfrak{R}^{fa}:=\mathfrak{R}/I_{\mathfrak{R}}$, as defined in Definition \ref{def:fairfication}, is a functorial construction. Indeed, we have a functor $\mathcal{F}a:{\bf \C Superings}  \rightarrow  {\bf \C Superrings}^{fa}$, where ${\bf \C Superrings}^{fa}$ denotes the category of fair $\C$-superrings. This category is a full subcategory of the category of finitely generated $\C$-superrings. More precisely, the functor $\mathcal{F}a$ is defined on objects by $\mathcal{F}a(\mathfrak{R})=\mathfrak{R}/I_\mathfrak{R}$. For a morphism $\varphi:\mathfrak{R}\longrightarrow\mathfrak{S}$, $\mathcal{F}a(\varphi)$ is the induced morphism $\varphi^{fa}:\mathfrak{R}/I_\mathfrak{R}\longrightarrow\mathfrak{S}/I_\mathfrak{S}$. This induced morphism is well-defined
since $\varphi(I_\mathfrak{R})\subseteq I_\mathfrak{S}$. We therefore call this functor the \textit{fairfication} functor.

\begin{lemma}\label{lem:fairfication}
The fairfication functor is a left adjoint to the inclusion functor from fair $\C$-superrings to finitely generated $\C$-superrings.
\end{lemma}

\begin{proof}
Let $\varphi:\mathfrak{S}\longrightarrow \mathfrak{R}$ be a morphism of $\C$-superrings, where $\mathfrak{S}$ is finitely generated and $\mathfrak{R}$ is fair. Recall that the fairfication functor is defined on an object $\mathfrak S$ as $\mathcal Fa (\mathfrak S)=\mathfrak S/I_{\mathfrak S}$ and on a morphism $\varphi$ as $\mathcal Fa (\varphi)=\varphi^{fa}$.

If $\rho:\mathfrak{S}\longrightarrow \mathfrak{S}/I_\mathfrak{S}$ is the quotient projection, then there is a natural morphism $\varphi^{fa}$ making the following diagram commutative:

\begin{center}
\begin{tikzcd}
\mathfrak S \arrow{d}[swap]{\varphi} \arrow{r}{\rho} & \mathfrak S/I_{\mathfrak S} \arrow{dl}{\varphi^{fa}} \\
 \mathfrak R
\end{tikzcd}
\end{center}

We claim that $\varphi^{fa}$ is well defined. This follows from the functoriality of localization on $\R$-points. Specifically, for any $x \in X_{\mathfrak{R}}$ and $\varphi(x) \in X_{\mathfrak{S}}$, the diagram below is commutative: 

\[ 
\begin{tikzcd}
\mathfrak{S}\arrow{r}{\varphi} \arrow[swap]{d}{\rho_{\mathfrak{S}}} & \mathfrak{R}\arrow{d}{\rho_{\mathfrak{R}}} \\
\mathfrak{S}/I_{\mathfrak{S}} \arrow{r}{\varphi_{x}} & \mathfrak{R}/I_{\mathfrak{R}}
\end{tikzcd}
\]

Consequently, if $a\in I_\mathfrak{S}$, then $\mathcal{L}_{\varphi(x)}(a)=0$. Thus, we have $0=\varphi_x(\mathcal{L}_{\varphi(x)}(a))=\mathcal{L}_x(\varphi(a))$. Since $x$ is arbitrary and $\mathfrak{R}$ is fair, it follows that $\varphi(a)=0$ which implies $\varphi^{fa}(\rho(a))=0$. The mapping $\varphi\longmapsto\varphi^{fa}$ is clearly injective. Moreover, it is surjective. To see this, consider any morphism $\Psi:\mathfrak{S}/I_\mathfrak{S}\longrightarrow\mathfrak{R}$ of fair $\C$-superrings. We can define 
$\psi: \mathfrak{S}\longrightarrow \mathfrak{R}$ by setting $\psi(a)=\Psi(\rho(a))$. It then follows
$\Psi=\psi^{fa}$, showing surjectivity.
\end{proof}

\subsection{Spectrum of a $\C$-superring as a locally ringed space.} Now that we have described the topological space $X_\mathfrak R$ of $\R$-points of a $\C$-superring $\mathfrak R$, we will define a structure sheaf on it. This will allow us to obtain a locally ringed space associated with $\mathfrak R$.
\begin{definition}
 A $\C$-ringed superspace or simply a $\C$-superspace is a ringed superspace $(X,\sh_X)$ such that $\sh_X$ is a sheaf of $\C$-superrings. We say that $X$ is a locally $\C$-ringed superspace if for any $x\in X$ the stalk $\sh_{X, x}$ is a local $\C$-superring.
  
\end{definition}
\begin{definition}
A morphism between two $\C$-superspaces $(X,\sh_X)$ and $(Y,\sh_Y)$ is a pair $(f,f^\#)$ where $f:X\rightarrow Y$ is a continuous function and $f^\#:\sh_Y\rightarrow f_*\sh_X$ is a morphism of sheaves of $\C$-superrings over $Y$.  
\end{definition}
\begin{remark}
Let $(X,\mathcal{O}_X)$ a $\C$-ringed superspace, the sheaf $\sh_X$ is a sheaf of $\C$-superrings which can be expressed as the sum $\sh_X{\ev}\oplus\sh_X{\od}$ were the even part $\sh_X{\ev}$ is a sheaf of $\C$-rings and $\sh_X{\od}$ is a $\sh_X{\ev}$-module. Additionally, the sheaf morphism $f^\#$ can be decomposed as $f^\#\ev\oplus f^\#\od$ where $f^\#\ev:(\sh_{Y})\ev\rightarrow (f_*\sh_X)\ev=f_*(\sh_X)$ is a morphism of sheaves of $\C$-rings and $f^\#\od:(\sh_{Y})\od\rightarrow (f_*\sh_X)\od=f_*(\sh_X)$ is a morphism of $(\sh_{Y})\ev$-modules.  
\end{remark}
As a consequence of last remark, the following notation for a morphism $(f,f^\#):(X,\sh_X)\rightarrow (Y,\sh_Y)$ will be useful despite it might not make sense as a direct sum of morphism of any concrete category.   $$(f,f^\#):=(f,f^\#\ev)\oplus f^\#\od$$

\begin{example}
Let $M$ be a smooth supermanifold of dimension $p\mid q$. Consider $\C(M_{red})[\theta^1,\ldots, \theta^q]^{\pm}$, which is the $\C$-superring of global sections. The resulting pair $(M_{red},S\C)$, with $S\C(U)=\C(U)[\theta^1,\ldots, \theta^q]^{\pm}$ for $U\subseteq M_{red}$ open, is the archetypal example of a locally $\C$-ringed superspace. Its stalks at $p\in M_{red}$ are given by $\C(M_{red})_{p}[\theta^1,\ldots, \theta^q]^{\pm}$.
\end{example}

Similar to the relationship between $\C$-spaces and smooth manifolds, the category of $\C$-superspaces is considerably broader than that of smooth supermanifolds. In subsequent work (see \cite{ORTG}), we will provide further examples of $\C$-superspaces. These examples are motivated by various concepts of differentiable spaces found in the literature, all of which are instances of $\C$-ringed spaces.\\
\begin{definition}
Let $\mathfrak{R}=\mathfrak{R}\ev\oplus \mathfrak{R}\od$ be a $\C$-superring. We define a sheaf of $\C$-superrings, denoted $\mathcal{O}_\mathfrak{R}$, on the space $X_\mathfrak{R}$ as follows: for any open set $U\subseteq X_{\mathfrak{R}}$,
$$
\mathcal{O}_\mathfrak{R}(U)=\mathcal{O}_{X_{\mathfrak{R}\ev}}(U)\oplus (\text{Mspec} \, \mathfrak{R}\od)(U) \, .
$$ 
Here, $\sh_{X_{R\ev}}$ represents the structure sheaf of $\C$-rings over $X_{\mathfrak{R}\ev}$. The term $\text{Mspec}\, \mathfrak{R}\od$ denotes the $\sh_{X_{R\ev}}$-module induced by $\mathfrak{R}\od$ (see Definition \ref{Mspec}).\\ 

The restriction maps for $\mathcal{O}_\mathfrak{R}$ are induced by the corresponding restriction maps of $\mathcal{O}_{X_{\mathfrak{R}\ev}}$ and $ Mspec\mathfrak{R}\od$, respectively. Furthermore, for every $x \in X_{\mathfrak R}$, the stalk of $\mathcal{O}_\mathfrak{R}$ at $x$ is naturally given by the localization  $(\mathcal{O}_{X_{\mathfrak{R}\ev}})_x \oplus (\text{Mspec}\, \mathfrak{R}\od)_x$.
\end{definition}

For any morphism $\varphi:\mathfrak{R}\rightarrow \mathfrak{S}$ between $\C$-superrings, we can define a morphism $(f_\varphi, f_{\varphi}^{\#} )$ between the ringed spaces $(X_{\mathfrak S}, \mathcal O_{\mathfrak S})$ and $(X_{\mathfrak R}, \mathcal O_{\mathfrak R})$ as follows:

\begin{itemize}
\item {\bf On the topological spaces:} The continuous function $f_\varphi:X_{\mathfrak{S}}\rightarrow X_\mathfrak{R}$ is defined for every $x \in X_{\mathfrak S}$ by $f_\varphi(x)= x\circ \varphi$. 

\item {\bf On the structural sheaves:} The morphism of sheaves $f_{\varphi}^{\#}:\mathcal{O}_{\mathfrak{R}}\longrightarrow  (f_\varphi)_\ast \mathcal{O}_{\mathfrak{S}}$ is defined over each open subset $U\subseteq X_\mathfrak{R}$. It corresponds to the morphism of $\C$-superrings:
$$ 
f_{\varphi}^{\#}(U):\mathcal{O}_{X_{\mathfrak{R}\ev}}(U)\oplus (\text{Mspec}\,  \mathfrak{R}\od)(U)\longrightarrow \mathcal{O}_{X_{\mathfrak{S}\ev}}(f_{\varphi}^{-1}(U))\oplus (\text{Mspec} \, \mathfrak{S}\od)(f_{\varphi}^{-1}(U))
$$ 
defined by $s\oplus e\longmapsto f_{\varphi\ev}^{\#}(s)\oplus f_{\varphi\od}^{\#}(e)$ 

where $f_{\varphi\ev}^{\#}$ (resp. $f_{\varphi\od}^{\#}$) corresponds to the morphism induced by the even (resp. odd) part of $\varphi$ via the Spectrum functor (resp. Module Spectrum) in Definition \ref{def:twofunctors} (resp. Definition \ref{Mspec}). 

\end{itemize} 

\begin{definition}\label{def:Supertwofunctors}
We define two covariant functors between the (opposite) category of $\C$-superrings, ${\bf \C Superrings}$, and the category of locally $\C$-superringed spaces, ${\bf L\C SRS}$. These functors are the analogues of those given in Definition \ref{def:twofunctors}).

\begin{itemize}
\item {\bf Super-Spectrum Functor:} The functor $sSpec:  {\bf \C Superrings}^{op} \rightarrow {\bf L\C SRS}$ is defined as follows:
\begin{enumerate}[(i)]
\item On objects, for any $\C$-superring $\mathfrak{R}$, $sSpec(\mathfrak{R})=(X_{\mathfrak{R}},\mathcal{O}_{X_{\mathfrak{R}}})$. This locally $\C$-superringed space is called an affine $\C$-superscheme.

\item On morphisms, for any $\C$-ring morphism $\varphi: \mathfrak{R} \rightarrow \mathfrak{S}$, $sSpec(\varphi)=(f_{\varphi},f^{\#}_{\varphi})$.\\
\end{enumerate}

\item {\bf Global Sections Functor:} The functor $s\Gamma:\bf{L\C SRS}\rightarrow {\bf \C Superrings}^{op}$ is defined as follows:
\begin{enumerate}[(i)]
\item On objects, for any locally $\C$-superringed space $(X, \mathcal O_X)$, $s\Gamma(X,\mathcal{O}_X)=\mathcal{O}_X(X)$.

\item On morphisms, for any morphism $(f, f^{\#}) : (X, \mathcal O_X) \rightarrow (Y, \mathcal O_Y)$ of locally $\C$-superringed spaces, we have  $s\Gamma(f)=f^{\#}(X):\mathcal{O}_Y(Y)\rightarrow\mathcal{O}_X(X)$.
\end{enumerate}
\end{itemize}
\end{definition}

\begin{remark}\label{rmk:sglobal}
For a locally $\C$-superring space $(X, \mathcal O_X)$, where $\mathcal O_X=\mathcal O_{X,\overline 0} \oplus \mathcal O_{X,\overline 1}$, the global sections functor $s\Gamma$ can be naturally decomposed. This decomposition is given by $s \Gamma=\Gamma\ev \oplus \Gamma\od$, where:
\begin{itemize}
\item $\Gamma\ev:\bf{L\C RS}\rightarrow {\bf \C Rings}^{op}$ is the global sections functor for $\C$-rings, defined as  $\Gamma\ev(X,\mathcal O\ev)=(\mathcal O_X(X))\ev$.

\item $\Gamma\od:(\mathcal{O}_X)\ev{\bf -mod}\rightarrow (\mathcal{O}_X(X))\ev{\bf- mod}$ is the global sections functor for modules (Definition \ref{modglobalsections}), define as $\Gamma\od(X, \mathcal O_{X, \overline 1})=(\mathcal O_X(X))\od$.
\end{itemize}
\end{remark}

\begin{example}
Consider the split $\C$-superring $\mathfrak{R} = \C(\R^{p|q})$. The global sections of $sSpec(\mathfrak{R})=(X_{\mathfrak R},\mathcal O_{X_{\mathfrak R}} )$ is $\mathfrak{R}$ itself. This is because: $$s\Gamma(sSpec(\mathfrak{R}))=\Gamma\ev(X_{\mathfrak R},\mathcal O_{X_{\mathfrak R}, \overline 0})\oplus \Gamma\od(X_{\mathfrak R},\mathcal O_{X_{\mathfrak R}, \overline 1})=(O_{X_{\mathfrak R}}(X_{\mathfrak R}))\ev \oplus (O_{X_{\mathfrak R}}(X_{\mathfrak R}))\od=\mathfrak R\ev \oplus \mathfrak R\od=\mathfrak R.$$
\end{example}

\begin{theorem}\label{Adjunction}
The functor $sSpec$ is right adjoint to the functor $s\Gamma$ . 
\end{theorem}

\begin{proof}
To prove this adjuntion, we will show that there exists a natural bijection between the hom-sets ${\rm Hom}_{\bf \C Superrings}(\mathfrak{R},s\Gamma(X,\mathcal{O}_X))$ and ${\rm Hom}_{\bf L\C SRS}((X,\mathcal{O}_X), sSpec(\mathfrak{R}))$, for any $\mathfrak{R}\in {\bf \C Superrings}$ and any $(X,\mathcal{O}_X)\in {\bf L\C SRS}$. This bijection must be natural in both variables.\\

Let $\mathfrak{R}=\mathfrak{R}\ev\oplus\mathfrak{R}\od$ be a $\C$-superring and let $(Y,\mathcal{O}_{Y})=sSpec(\mathfrak{R})$. Consider a morphism of $\C$-superrings $\varphi:\mathfrak{R}\longrightarrow s\Gamma(X)$. By Definition \ref{def:smorph} and Lemma \ref{escalarext}, this morphism $\varphi$ can be decomposed into its even and odd components: $\varphi\ev:\mathfrak{R}\ev\longrightarrow \Gamma\ev(X)$ is a morphism of $\C$-superrings and $\varphi\od:\mathfrak{R}\od\longrightarrow \Gamma\od(X)$ is a morphism of $\mathfrak{R}\ev$-modules. 

Following the notation of Remark \ref{rmk:sglobal}, the adjunction in the non-graded case (Theorem \ref{adjoint}) implies a natural bijection:
$$
L\ev:{\rm Hom}_{\bf \C Rings}      (\mathfrak{R}\ev, \Gamma\ev(X))\longrightarrow {\rm Hom}_{\bf L\C RS}((X,\mathcal{O}_{X, \overline 0}), (Y,\mathcal{O}_{Y, \overline 0})).
$$
This bijection yields a morphism
\begin{equation}\label{eq:evenadj} L\ev(\varphi\ev):= (\psi\ev,\psi\ev^{\#}):(X, \sh_{X, \overline{0}})\longrightarrow (Y,\sh_{Y, \overline{0}})
\end{equation}
where $\psi\ev:X\longrightarrow Y$ is a continuous function and $\psi\ev^{\#}:\sh_{Y, \overline{0}}\longrightarrow(\psi\ev)_{\ast}\sh_{X,\overline{0}}$ is a morphism of sheaves of $\C$-rings on Y.

Similarly, by $\C$-modules adjunction (see \cite[Theorem 5.19]{J}), there exists a natural bijection:
\begin{equation}\label{eq:oddadj} 
L\od:{\rm Hom}_{\mathfrak R\ev-mod}(\mathfrak{R}\od, \Gamma\od(X))\longrightarrow {\rm Hom}_{\sh_{Y, \overline{0}}-mod}(\sh_{Y, \overline{1}}, (\psi\ev)_{\ast}\sh_{X, \overline{1}}).
\end{equation}
This provides a morphism of $\sh_{Y, \overline{0}}$-modules, $L\od(\varphi\od):=\psi\od:\sh_{Y, \overline{1}} \longrightarrow (\psi\ev)_{\ast}\sh_{X, \overline{1}}$.

In conclusion, we define the map
$$
L:{\rm Hom}_{\bf \C Superrings}(\mathfrak{R},s\Gamma(X))\longrightarrow {\rm Hom}_{\bf L\C SRS}((X,\mathcal{O}_X), (Y,\mathcal{O}_Y))
$$
by mapping $\varphi$ to $L(\varphi)=(\psi,\psi^\#)$. Here, $\psi=L\ev(\varphi\ev):X\rightarrow Y$ is the continuous function between the topological spaces X and Y derived from Theorem \ref{adjoint}. Furthermore, $\psi^{\#}=\psi^{\#}\ev\oplus \psi^{\#}\od $, where $\psi^{\#}\ev$ and $\psi^{\#}\od:=\psi\od$ are as described by (\ref{eq:evenadj}) and (\ref{eq:oddadj}), respectively. Thus, for any open subset $U\subseteq Y$, $\psi^{\#}(U):\sh_Y(U)\rightarrow \sh_X(\psi^{-1}(U))$ is a morphism of $\C$-superrings, given by 
$$
\psi^{\#}(U)(a\oplus b)=\psi^{\#}\ev (U)(a)\oplus \psi^{\#}\od(U)(b).
$$
This map is a ring homomorphism that preserves parity, and by definition, $\psi^{\#}\ev(U)$ is a $\C$-ring morphism. Since both $L\ev$ and $L\od$ are bijections, we denote their inverses by $L\ev^{-1}$ and $L\od^{-1}$, respectively. We then define the function
$$
R:{\rm Hom}_{\bf L\C SRS}((X,\mathcal{O}_X), (Y,\mathcal{O}_Y))\rightarrow {\rm Hom}_{\bf \C Superrings}(\mathfrak{R},s\Gamma(X))
$$
as $(f,f^\#)\longrightarrow R(f,f^\#):=L^{-1}\ev(f,f^{\#}\ev)\oplus L\od^{-1}(f^{\#}\od)$. We claim that $R$ is the inverse of $L$. Indeed,
\begin{align*}
         L(R(f,f^\#))&=L(L^{-1}\ev(f,f^\#)\oplus L_1^{-1}(f^\#))\\
         &= L{\ev}(L{\ev}^{-1}(f,f^\#))\oplus L\od(L\od^{-1}(f,f^\#))\\
         &= (f, f^{\#}\ev)\oplus f^{\#}\od\\
         &= (f, f^\#).
     \end{align*}

 Conversely, we obtain that 
 \begin{eqnarray*}
     R(L(\varphi))&= & R  (L{\ev}(\varphi\ev), L{\ev}(\varphi^\#\ev)\oplus L{\od}(\varphi^\#\od ) \\
     &= & L^{-1}{\ev}((L{\ev}(\varphi{\ev}), L{\ev}(\varphi^{\#}\ev ))\oplus    L^{-1}{\od}(L{\od}(\varphi))\\
     &= & (\varphi{\ev},\varphi^{\#}\ev)\oplus \varphi^{\#}\od\\
     &= & \varphi{\ev}\oplus\varphi{\od}. 
 \end{eqnarray*}
 
Finally, the bijection is natural, as it follows directly from the naturality of the spectrum and global sections functors in the non-graded case.

\end{proof}

In the non-graded setting (see \cite{J, Olarte}), the adjunction between the global sections functor and the spectrum functor establishes a crucial link for understanding the relationship between algebraic and geometric structures. The following propositions generalize \cite[Theorem  4.22]{J} to the super setting, demonstrating this connection by showing how abstract super-algebraic properties are precisely reflected in geometric super-spaces.
\begin{proposition} The category of split finitely generated affine $\C$-superschemes admits fiber products. Specifically, for any pair of split finitely generated affine $\C$-superrings $\mathfrak{R}$ and $\mathfrak{S}$ there exists a $\C$-superring $\mathfrak{T}$ such that $sSpec(\mathfrak{T})$ satisfies the universal property of fiber products for  $sSpec(\mathfrak{R})$ and $sSpec(\mathfrak{S})$.
  \end{proposition}
 
\begin{proof}
In Proposition \ref{super-coproduct}, we constructed the coproduct $\mathfrak{R}\otimes_\infty\mathfrak{S}$ of two split finitely generated $\C$-superrings $\mathfrak R$ and $\mathfrak S$. We claim that $\mathfrak{T}=\mathfrak{R}\otimes_\infty\mathfrak{S}$. The functor $sSpec$ is a right adjoint and therefore transform colimits into limits. Consequently, it maps the coproduct of $\C$-superrings to the fiber products of affine $\C$-superschemes. Thus, the fiber product $sSpec(\mathfrak{R})$ and $sSpec(\mathfrak{S})$ is given by  $sSpec(\mathfrak{R}\otimes_\infty\mathfrak{S})$. That is, $sSpec(\mathfrak{R}) \times_{sSpec (\R)} sSpec(\mathfrak{S})=sSpec(\mathfrak{R}\otimes_\infty\mathfrak{S})$.
\end{proof}
\begin{proposition}\label{prop:fairequiv}
For any $\C$-superring $\mathfrak{R}$, there is an isomorphism of $\C$-superrings:
$$
s\Gamma(sSpec(\mathfrak{R})\cong \mathcal{F}a(\mathfrak{R}), 
$$
Here, $\mathcal{F}a(\mathfrak{R})$ is the fairfication of $\mathfrak{R}$ (see Lemma \ref{lem:fairfication}). Furthermore, the natural transformation  $\mathcal{F}a \Rightarrow s\Gamma\circ sSpec$ is a natural isomorphism of functors.
\end{proposition}

\begin{proof}
We will construct a $\C$-superring morphism $\Phi:\mathfrak{R}\rightarrow s\Gamma(sSpec(\mathfrak{R}))$.

A global section of $\mathcal O_{X_{\mathfrak R}}$ can be viewed as a map $s:X_\mathfrak{R}\rightarrow \bigsqcup_{x\in X_\mathfrak{R}}\mathfrak{R}_x$ such that for any $x\in X_{\mathfrak{R}}$, there exists an open neighborhood $U \subseteq X_{\mathfrak R}$ and an element $r\in \mathfrak{R}$ with $s(y)=\mathcal{L}_x(r)$ for $y \in U$. 

We define $\Phi(r)$ for any $r\in\mathfrak{R}$ as the global section that maps $x$ to $\mathcal{L}_x(r)$ in $\mathfrak{R}_x$. It is clear that $\Phi$ is surjective and its kernel is $I_\mathfrak{R}=\{r \in \mathfrak{R}\mid\mathcal{L}_x(r)=0,\, \forall x\in X_{\mathfrak{R}}\}$. 

By the First Isomorphism Theorem, we have an isomorphism between  $s\Gamma(sSpec(\mathfrak{R}))$ and $\mathfrak R/I_{\mathfrak R}=:\mathcal{F}a(\mathfrak{R})$. We denote this isomorphism by $\Phi_\mathfrak{R}$. This completes the proof of the isomorphism at the object level.

Now, consider $\varphi:\mathfrak{R}\rightarrow\mathfrak{S}$ a morphism of $\C$-superrings. We claim that the following diagram is commutative:
\begin{equation}\label{eq:fairdiagram}
\begin{tikzcd}
\mathcal{F}a(\mathfrak{R})\arrow{r}{\Phi_\mathfrak{R}} \arrow[swap]{d}{\epsilon=\mathcal{F}a(\varphi)} & s\Gamma(sSpec(\mathfrak{R}))\arrow{d}{ \delta=s\Gamma(sSpec(\varphi))} \\
\mathcal{F}a(\mathfrak{S}) \arrow{r}{\Phi_\mathfrak{S}} &  s\Gamma(sSpec(\mathfrak{S}))
\end{tikzcd}
\end{equation} 
Indeed, for any $\tilde{r} \in \mathcal{F}a(\mathfrak{R})$ we have that $\Phi_\mathfrak{S}(\epsilon(\tilde {r}))$ is by definition a global section of $\mathcal O_{X_{\mathfrak{S}}}$ given by $x\longmapsto\mathcal{L}_x(\varphi(r))$ for $x \in X_{\mathfrak S}$. On the other hand, $\delta(\Phi_\mathfrak{R}(\tilde {r}))$ is a global section of $\mathcal O_{X_{\mathfrak{S}}}$ given by $x\longmapsto \varphi_x(\mathcal{L}_{x\circ\varphi}(\tilde{r}))=\mathcal{L}_x(\varphi(r))$ by Proposition \ref{Rpointmor}. Hence we have that  $\delta(\Phi_\mathfrak{R}(\tilde {r}))=\Phi_\mathfrak{S}(\tilde {r})$. 
In conclusion, $\Phi:\mathcal{F}a\Rightarrow s\Gamma\circ sSpec$ is, actually, a natural transformation. Additionally, it is an isomorphism of functors since the horizontal arrows in the diagram (\ref{eq:fairdiagram}) are isomorphisms, which proves the proposition.
\end{proof} 
\begin{corollary}\label{cor:fairestriction}
Under the hypotheses and notations in Theorem \ref{Adjunction}, the functor $sSpec$ is full and faithful when restricted to the category of fair $\C$-superrings. This means that for any pair of of fair $\C$-superrings $\mathfrak{R}$ and $\mathfrak{S}$, there is a natural bijection:
\begin{equation}\label{eq:fffair}
{\rm Hom}_{\bf \C Superrings}(\mathfrak{R},\mathfrak{S})\cong{\rm Hom}_{\bf L\C SRS}(sSpec(\mathfrak{S}), sSpec(\mathfrak{R}))\, .
\end{equation}
\end{corollary}

\begin{proof}
By the adjunction theorem, Theorem \ref{Adjunction}, we have a bijection   
\begin{equation*}
{\rm Hom}_{\bf \C Superrings}(\mathfrak{R},s\Gamma(sSpec(\mathfrak{S})))\cong{\rm Hom}_{\bf L\C SRS}(sSpec(\mathfrak{S}), sSpec(\mathfrak{R})). 
\end{equation*} 
Since $\mathfrak{S}$ is a fair $\C$-superring, Proposition \ref{prop:fairequiv} states that $s\Gamma(sSpec(\mathfrak{S}))\cong\mathfrak{S}$. This implies the desired result.
\end{proof}

\section*{Acknowledgements}
A. Torres-Gomez is grateful to the Shanghai Institute for Mathematics and Interdisciplinary Sciences (SIMIS, China) for their hospitality and partial financial support during his visit in December 2024 and January 2025, where a part of this project was completed.

\section*{Funding}
C. D. Olarte and P. Rizzo were partially supported by CODI (Universidad de Antioquia, UdeA) through project numbers 2022-52654 and 2023-62291, respectively.


\end{document}